\documentclass[envcountsame]{svmult}
\usepackage[utf8]{inputenc}
\usepackage[T1]{fontenc}
\usepackage{indentfirst}
\usepackage{textcomp,lscape,rotating}
\usepackage[english]{babel}
\usepackage{enumitem}
\usepackage{amsmath,amsfonts,amssymb,amsopn,amscd}
\usepackage{mathrsfs,bbm}
\usepackage{graphicx}
\usepackage[dvipsnames]{xcolor}
\usepackage{tikz}

    \def\O{\mathcal{O}}
    \def\WDep{G}
    
    \def\ka{\kappa}
    
    \newcommand{\N}{\mathbb{N}}
    \newcommand{\Z}{\mathbb{Z}}
    \newcommand{\R}{\mathbb{R}}
    \newcommand{\C}{\mathbb{C}}
    \newcommand{\esper}{\mathbb{E}}
    \newcommand{\proba}{\mathbb{P}}

    \newcommand{\leb}{\mathrm{L}}
    \newcommand{\dkol}{d_\mathrm{Kol}}

    \newcommand{\testf}{\mathscr{T}}
    \newcommand{\schwartz}{\mathscr{S}}
    \newcommand{\sinc}{\mathrm{sinc}}
    \newcommand{\card}{\mathrm{card}}
    \newcommand{\eps}{\varepsilon}
    \newcommand{\spa}{\mathfrak{X}}
    
    \newcommand{\id}{\mathrm{id}}
    \newcommand{\var}{\mathrm{Var}}
    \newcommand{\cov}{\mathrm{Cov}}
    \newcommand{\gauss}{\mathcal{N}_{\R}(0,1)}
    \newcommand{\lle}{\left[\!\left[} 
    \newcommand{\rre}{\right]\!\right]} 
    \newcommand{\scal}[2]{\left\langle #1\vphantom{#2}\,\right |\left.#2 \vphantom{#1}\right\rangle}
    
    \newcommand{\DD}[1]{\,d\hspace*{-0.3mm}{#1}}
    
    \def\HHH{\mathcal{H}}
    
    \def\PPP{\mathcal{P}}
    \DeclareMathOperator{\ST}{ST}
    \DeclareMathOperator{\Var}{Var}
    \DeclareMathOperator{\Cov}{Cov}
    \smartqed

    \title*{Mod-$\phi$ convergence, II: Estimates on the speed of convergence}
	\author{Valentin F\'eray \and Pierre-Lo\"ic M\'eliot \and Ashkan Nikeghbali}
	\date{\today}

\begin{document}

    \maketitle
	\setlist[enumerate]{itemsep=10pt,topsep=10pt}
	\setlist[itemize]{itemsep=5pt,topsep=5pt}

    \abstract{In this paper, we give estimates for the speed of convergence towards a limiting stable law 
      in the recently introduced setting of mod-$\phi$ convergence. 
      Namely, we define a notion of \emph{zone of control}, closely related to mod-$\phi$ convergence,
      and we prove estimates of Berry--Esseen type under this hypothesis.
      Applications include:
      \begin{itemize}
        \item the winding number of a planar Brownian motion;
        \item classical approximations of stable laws by compound Poisson laws;
        \item examples stemming from determinantal point processes (characteristic polynomials of random matrices and zeroes of random analytic functions);
        \item sums of variables with an underlying dependency graph (for which we recover a result of Rinott, obtained by Stein's method);
        \item the magnetization in the $d$-dimensional Ising model;
        \item and functionals of Markov chains.
      \end{itemize}}

\section{Introduction}

\subsection{Mod-$\phi$ convergence}
Let $(X_n)_{n \in \N}$ be a sequence of real-valued random variables. In many situations, there exists a scale $s_n$ and a limiting law $\phi$ which is infinitely divisible, such that $(X_n/s_n)_{n\in \N}$ converges in law towards $\phi$. For instance, in the classical central limit theorem, if $X_n=\sum_{i=1}^n A_i$ is a sum of centered i.i.d.~random variables with $\esper[(A_1)^2]<\infty$, then $$s_n=\sqrt{n\,\esper[(A_1)^2]}$$ and the limit is the standard Gaussian distribution $\gauss$. In \cite{JKN11} and the subsequent papers \cite{DKN11,FMN16}, the notion of mod-$\phi$ convergence was developed in order to get quantitative estimates on the convergence $\frac{X_n}{s_n} \rightharpoonup \phi$
(throughout the paper, $\rightharpoonup$ denotes convergence in distribution).

\begin{definition}\label{def:modphi}
Let $\phi$ be an infinitely divisible probability measure, and $D \subset \C$ be a subset of the complex plane, which we assume to contain $0$. We assume that the Laplace transform of $\phi$ is well defined over $D$, with L\'evy exponent $\eta$:
$$\forall z \in D,\,\,\,\int_{\R} \E^{zx}\,\phi(\!\DD{x})=\E^{\eta(z)}.$$
We then say that $(X_n)_{n \in \N}$ converges mod-$\phi$ over $D$, with parameters $(t_n)_{n \in \N}$ and limiting function $\psi : D \to \C$, if $t_n \to + \infty$ and if, locally uniformly on $D$,
$$ \lim_{n \to \infty} \esper[\E^{zX_n}]\,\E^{-t_n\eta(z)} = \psi(z).$$
\end{definition}

\noindent If $D=\I \R$, we shall just speak of mod-$\phi$ convergence; it is then convenient to use the notation 
\begin{align*}
\theta_n(\xi)&=\esper[\E^{\I \xi X_n}]\,\E^{-t_n\eta(\I \xi)};\\
\theta(\xi)&= \psi(\I\xi),
\end{align*}
so that mod-$\phi$ convergence corresponds to $\lim_{n \to \infty}\theta_n(\xi)= \theta(\xi)$
(uniformly for $\xi$ in compact subsets of $\R$).
When nothing is specified, in this paper, we implicitly consider that $D=\I\R$.
When $D=\C$ we shall speak of \emph{complex} mod-$\phi$ convergence.
In some situations, it is also appropriate to study mod-$\phi$ convergence on a band $\R \times \I[-b,b]$, or $[-c,c] \times \I\R$ (see \cite{FMN16,MN15}).
\bigskip

Intuitively, a sequence of random variables $(X_n)_{n \in \N}$ converges mod-$\phi$ if it can be seen as a large renormalization of the infinitely divisible law $\phi$, plus some residue which is asymptotically encoded in the Fourier or Laplace sense by the limiting function $\psi$. Then, $\phi$ will typically be:
\begin{enumerate}
 	\item in the case of lattice-valued distributions, a Poisson law or a compound Poisson law (\emph{cf.} \cite{BKN09,KN10,FMN16,CDMN15});
 	\item or, a \emph{stable distribution}, for instance a Gaussian law.
 \end{enumerate} 
In this paper, we shall only be interested in the second case. Background on stable distributions is given at the end of this introduction (Section~\ref{Sect:Background_Stable}).
In particular we will see that, if $\phi$ is a stable distribution,
then the mod-$\phi$ convergence of $X_n$ implies
the convergence in distribution of a renormalized version $Y_n$ of $X_n$ to $\phi$
(Proposition~\ref{prop:clt}).\medskip

We believe that mod-$\phi$ is a kind of universality class behind the central limit theorem (or its stable version)
in the following sense.
For many sequences $(X_n)_{n \in \N}$ of random variables that are proven to be asymptotically normal (or converging to a stable distribution),
it is possible to prove mod-$\phi$ convergence;
we refer to our monograph \cite{FMN16} or Sections \ref{Sect:First}-\ref{sec:markov} below for such examples.
These estimates on the Laplace transform/characteristic function can then be used to derive
in an automatic way some {\em companion theorems}, refining the central limit theorem.
In \cite{FMN16}, we discuss in details the question of moderate/large deviation estimates
and of finding the {\em normality zone}.\medskip

In the present paper, we shall be interested in the speed of convergence towards the Gaussian (or more generally the stable) distribution
of the appropriate renormalization $Y_n$ of $X_n$.
To obtain sharp bounds on this speed of convergence, we do not work with mod-$\phi$ convergence,
but we introduce the notion of \emph{zone of control} for the renormalized characteristic function $\theta_n(\xi)$.
In many examples, such a zone of control can be obtained by essentially the same arguments 
used to prove mod-$\phi$ convergence, and in most examples, mod-$\phi$ convergence actually holds.

\subsection{Results and outline of the paper}
We take as reference law a stable distribution $\phi$ of index $\alpha \in (0,2]$.
Let $(X_n)_{n\in \N}$ be a sequence of variables that admits a zone of control
(this notion will be defined in Definition~\ref{def:zone};
this is closely related to the mod-$\phi$ convergence of $(X_n)_{n\in \N}$).
As we will see in Proposition~\ref{prop:zone_CLT},
this implies that some renormalization $Y_n$ of $X_n$ converges in distribution
towards $\phi$
and we are interested in the speed of convergence for this convergence. More precisely, we are interested in upper bounds for the \emph{Kolmogorov distance}
$$\dkol(Y_n,\phi)=\sup_{a \in \R}\, \left|\proba[Y_n \leq a]-\int_{-\infty}^a \phi(\!\DD{x})\right| .$$
The main theorem of Section~\ref{sec:testfunctions} (Theorem~\ref{thm:kolmogorov})
shows that this distance is $\O( t_n^{-\gamma-1/\alpha})$,
where $\gamma$ is a parameter describing how large our zone of control is.
We also obtain as intermediate result estimates for
$$\left|\esper[f(Y_n)]-\int_{\R} f(x)\,\phi(\!\DD{x})\right|, $$
where $f$ lies in some specific set of tests functions (Proposition~\ref{prop:zonedistribution}).
A detailed discussion on the method of proof of these bounds can be found at the beginning of Section~\ref{sec:testfunctions}.\medskip

Section~\ref{Sect:First} gives some examples of application of the theoretical results of Section~\ref{sec:testfunctions}. The first one is a toy example, while the other ones are new to the best of our knowledge.
\begin{itemize}
  \item We first consider sums of i.i.d.~random variables with finite third moment.
    In this case, the classical Berry--Esseen estimate ensures that
    $$\dkol(Y_n,\gauss) \leq \frac{3\,\esper[|A_1|^3]}{\sigma^3\,\sqrt{n}},$$
    see \cite{Berry1941} or \cite[\S XVI.5, Theorem 1]{Fel71}.
    Our general statement for variables with a zone of convergence
    gives essentially the same result, only the constant factor is not as good.

    \item We can extend the Berry--Esseen estimates to the case of independent but non identically distributed random variables. As an example, we look at the number of zeroes $Z_r$ of a random analytic series that fall in a disc of radius $r$; it has the same law as a series of independent Bernoulli variables of parameters $r^{2k},\,k\geq 1$. When the radius $r$ of the disc goes to $1$, one has a central limit theorem for $Z_r$, and the theory of zones of control yields an estimate $O((1-r)^{-1/2})$ on the Kolmogorov distance.

  \item We then look at the winding number $\varphi_t$ of a planar Brownian motion starting at $1$ 
    (see Section~\ref{Sect:Winding} for a precise definition). 
    This quantity has been proven to converge in the mod-Cauchy sense in \cite{DKN11},
    based on the computation of the characteristic function done by Spitzer \cite{Spi58}. The same kind of argument easily yields the existence of a zone of control and our general result applies:
    when $t$ goes to infinity, after renormalization, $\varphi_t$ converges in distribution towards a Cauchy law
    and the Kolmogorov distance in this convergence is  $\O((\log t)^{-1})$. 

  \item In the third example, we consider {\em compound Poisson laws} (see \cite[Chapter 1, \S4]{Sato99}).
    These laws appear in the proof of the Lévy--Khintchine formula for infinitely divisible laws (\emph{loc.~cit.}, Chapter 2, \S8, p.~44--45), and we shall be interested in those that approximate the stable distributions $\phi_{c,\alpha,\beta}$. Again, establishing the existence of a zone of control is straight-forward
    and our general result shows that the speed of convergence is $\O(n^{-1/\min(\alpha,1)})$ (Proposition~\ref{prop:compoundpoisson}),
    with an additional log factor if $\alpha=1$ and $\beta \ne 0$ (thus exhibiting an interesting phase transition phenomenon).

    \item Ratios of Fourier transforms of probability measures appear naturally in the theory of self-decom\-pos\-able laws and of the corresponding Ornstein--Uhlenbeck processes. Thus, any self-decomposable law $\phi$ is the limiting distribution of a Markov process $(U_t)_{t \geq 0}$, and when $\phi$ is a stable law, one has mod-$\phi$ convergence of an adequate renormalisation of $U_t$, \emph{with a constant residue}. This leads to an estimate of the speed of convergence which depends on $\alpha$, on the speed of $(U_t)_{t \geq 0}$ and on its starting point (Proposition~\ref{prop:ornstein}).

    \item Finally, logarithms of characteristic polynomials of random matrices in a classical compact Lie group are mod-Gaussian convergent (see for instance \cite[Section 7.5]{FMN16}), and one can compute a zone of control for this convergence, which yields an estimate of the speed of convergence $O((\log n)^{-3/2})$. For unitary groups, one recovers \cite[Proposition 5.2]{BHNY08}. This example shows how one can force the index $v$ of a zone of control of mod-Gaussian convergence to be equal to $3$, see Remark \ref{rem:clever}.
\end{itemize}
\vspace{2mm}

The last two sections concentrate on the case where the reference law is Gaussian ($\alpha=2$).
In this case, we show that a sufficient condition for having a zone of control is 
to have {\em uniform bounds on cumulants} (see Definition~\ref{def:virtual} and Lemma~\ref{lem:Cumulants_Zone}).
This is not surprising since such bounds are known to imply (with small additional hypotheses)
mod-Gaussian convergence \cite[Section 5.1]{FMN16}.
Combined with our main result, this gives bounds for the Kolmogorov distance
for variables having uniform bounds on cumulants -- see Corollary~\ref{cor:Cumulants_Kol}. Note nevertheless that similar results have been given previously by Statulevi\v{c}ius \cite{Sta66}
(see also Saulis and Statulevi\v{c}ius \cite{LivreOrange:Cumulants}).
Our Corollary~\ref{cor:Cumulants_Kol} coincides up to a constant factor to
one of their result. Our contribution here therefore consists in giving
a large variety of non-trivial examples where such bounds on cumulants hold:

\begin{itemize}
\item The first family of examples relies on a previous result by the authors \cite[Chapter 9]{FMN16} 
    (see Theorem~\ref{thm:boundcumulant} here), where bounds on cumulants for 
    sums of variables with an underlying {\em dependency graph} are given. Let us comment a bit. Though introduced originally in the framework of the probabilistic method \cite[Chapter 5]{AS08},
    dependency graphs have been used to prove central limit theorems on various objects:
    random graphs \cite{Jan88}, random permutations \cite{Bon10}, probabilistic geometry \cite{PY05},
    random character values of the symmetric group \cite[Chapter 11]{FMN16}.
    In the context of Stein's method,
    we can also obtain bounds for the Kolmogorov distance in these central limit theorems
    \cite{BR89,Rin94}.\vspace{2mm}

    \noindent The results of this paper give another approach to obtain bounds for this Kolmogorov distance
    for sums of bounded variables (see Section~\ref{subsec:dependency}).
    The bounds obtained are, up to a constant, the same as in \cite[Theorem 2.2]{Rin94}.
    Note that our approach is fundamentally different, since it relies on classical Fourier analysis,
    while Stein's method is based on a functional equation for the Gaussian distribution.
    We make these bounds explicit in the case of subgraph counts in Erd\"{o}s--Rényi random graphs
    and discuss an extension to sum of unbounded variables.
    
\item The next example is the finite volume magnetization in the Ising model on $\Z^d$.
    The Ising model is one of the most classical models of statistical mechanics, 
    we refer to
    \cite{Velenik} and references therein for an introduction to this vast topic.
    The magnetization $M_\Delta$ (that is the sum of the spins in $\Delta$) is known to have
    asymptotically normal fluctuations \cite{New80}.
    Based on a result of Duneau, Iagolnitzer and Souillard \cite{Duneau2},
    we prove that, if the magnetic field is non-zero or if the temperature is sufficiently large,
    $M_\Delta$ has uniform bounds on cumulants.
    This implies a bound on the Kolmogorov distance (Proposition~\ref{prop:Kol_Ising}):
  \[\dkol\left( \frac{M_\Delta-\esper[M_\Delta]}{\sqrt{\Var(M_\Delta)}}\,,\, \gauss\right) \le \frac{K}{\sqrt{|\Delta|}}.\]
  It seems that this result improves on what is known so far.
  In \cite{Bul96},  Bulinskii gave a general bound on the Kolmogorov distance
  for sums of {\em associated random variables}, which applied to $M_\Delta$,
  yields a bound with an additional $(\log |\Delta|)^d$ factor comparing to ours.
  In a slightly different direction, Goldstein and Wiroonsri \cite{GW16} have recently given a bound 
  of order $\O(|\Delta|^{1/(2d+2)})$ for
  the $\leb^1$-distance (the $\leb^1$-distance is another distance on distribution functions,
  which is a priori incomparable with the Kolmogorov distance;
  note also that their bound is only proved 
  in the special case where $\Delta=\{-m,-m+1,\ldots,m\}^d$).

\item The last example considers statistics of the form $S_n=\sum_{t=0}^n f_t(X_t)$,
  where $(X_t)_{t \ge 0}$ is an ergodic discrete time Markov chain on a finite space state.
  Again we can prove uniform bounds on cumulants and deduce from it bounds for the Kolmogorov distance
  (Theorem~\ref{thm:kolmogorovmarkov}).
  The speed of convergence in the central limit theorem for Markov chains
  has already been studied by Bolthausen \cite{Bol80}
  (see also later contributions of Lezaud \cite{Lez96} and Mann \cite{Mann96}).
  These authors study more generally Markov chains on infinite space state,
  but focus on the case of a statistics $f_t$ independent of the time.
  Except for these differences, the bounds obtained are of the same order; however our approach and proofs are again quite different.
\end{itemize}
It is interesting to note that the proofs of the bounds on cumulants in the last two examples
are highly non trivial and share some common structure.
Each of these statistics decomposes naturally as a sum.
In each case, we give an upper bound for joint cumulants of the summands,
which writes as a weighted enumeration of spanning trees.
Summing terms to get a bound on the cumulant of the sum is then easy.\medskip

To formalize this idea, we introduce in Section~\ref{sec:markov} the notion
of uniform weighted dependency graphs.
Both proofs for the bounds on cumulants (for magnetization of the Ising model
and functional of Markov chains) are presented in this framework.
We hope that this will find further applications in the future.

\subsection{Stable distributions and mod-stable convergence}
\label{Sect:Background_Stable}
Let us recall briefly the classification of stable distributions (see \cite[Chapter 3]{Sato99}). Fix $c>0$ (the \emph{scale} parameter), $\alpha \in (0,2]$ (the \emph{stability} parameter), and $\beta \in [-1,1]$ (the \emph{skewness} parameter). 

\begin{definition}
The stable distribution of parameters $(c,\alpha,\beta)$ is the infinitely divisible law 
\hbox{$\phi=\phi_{c,\alpha,\beta}$} whose Fourier transform
$$\widehat{\phi}(\xi)=\int_{\R}\E^{\I x\xi} \,\phi(\!\DD{x}) = \E^{\eta(\I\xi)}$$
has for L\'evy exponent $\eta(\I\xi)=\eta_{c,\alpha,\beta}(\I\xi)= -|c\xi|^\alpha \left(1-\I \beta \,h(\alpha,\xi)\,\mathrm{sgn}(\xi)\right),$
where
$$h(\alpha,\xi) =\begin{cases}
\tan\left( \frac{\pi \alpha}{2}\right)& \text{if }\alpha \neq 1,\\
-\frac{2}{\pi} \log |\xi|& \text{if }\alpha = 1
\end{cases}$$
and $\mathrm{sgn}(\xi) = \pm 1$ is the sign of $\xi$.
\end{definition}

\noindent The most usual stable distributions are:
\begin{itemize}
\item the standard Gaussian distribution $\frac{1}{\sqrt{2\pi}}\,\E^{-x^2/2}\DD{x}$ for $c=\frac{1}{\sqrt{2}}$, $\alpha=2$ and $\beta=0$;
\item the standard Cauchy distribution $\frac{1}{\pi(1+x^2)}\DD{x}$ for $c=1$, $\alpha=1$ and $\beta=0$; 
\item the standard L\'evy distribution $\frac{1}{\sqrt{2\pi}}\,\frac{\E^{-1/2x}}{x^{3/2}}\,\mathbbm{1}_{x \geq 0}\DD{x}$ for $c=1$, $\alpha=\frac{1}{2}$ and $\beta=1$.
\end{itemize}

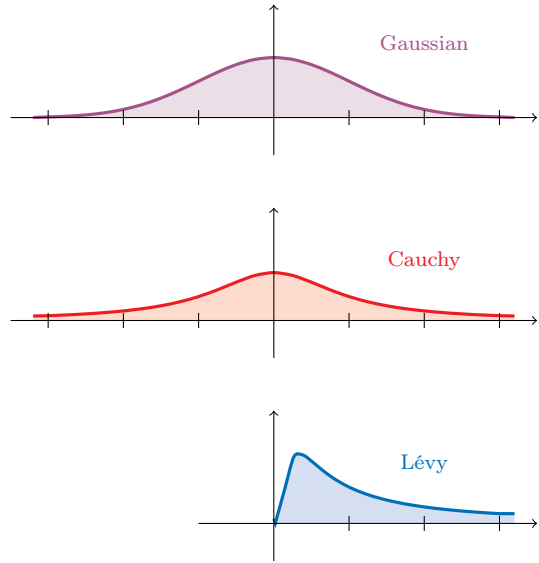
\begin{figure}[ht]
\begin{center}
\begin{tikzpicture}[xscale=1,yscale=2]
\fill[domain=-3:3, smooth, line width=1.2pt, color=DarkOrchid!15!white] (-3.2,0) -- plot (\x,{exp(-\x*\x/2)/sqrt(6.28)}) -- (3.2,0);
\draw[domain=-3:3, smooth, line width=1.2pt, color=DarkOrchid] (-3.2,0) -- plot (\x,{exp(-\x*\x/2)/sqrt(6.28)}) -- (3.2,0);
\draw[->] (-3.5,0) -- (3.5,0);
\draw[->] (-0,-0.25) -- (0,0.75);
\foreach \x in {-3,-2,-1,1,2,3}
		\draw (\x,-0.05) -- (\x,0.05);
\draw (2,0.5) node {\textcolor{DarkOrchid}{Gaussian}};
\begin{scope}[shift={(0,-1.35)}]
\fill[domain=-3:3, smooth, line width=1.2pt, color=Red!15!white] (-3.2,0)-- (-3.2,0.03) -- plot (\x,{1/(3.14*(1+\x*\x))}) -- (3.2,0.03) -- (3.2,0);
\draw[domain=-3:3, smooth, line width=1.2pt, color=Red] (-3.2,0.03) -- plot (\x,{1/(3.14*(1+\x*\x))}) -- (3.2,0.03);
\draw[->] (-3.5,0) -- (3.5,0);
\draw[->] (-0,-0.25) -- (0,0.75);
\foreach \x in {-3,-2,-1,1,2,3}
		\draw (\x,-0.05) -- (\x,0.05);
\draw (2,0.4) node {\textcolor{Red}{Cauchy}};
\end{scope}
\begin{scope}[shift={(0,-2.7)}]
\fill[domain=0.02:3, smooth, line width=1.2pt, color=NavyBlue!15!white] (0,0) -- (0,0.03) -- plot (\x,{exp(-1/(2*\x))/(sqrt(6.28*\x*\x*\x)}) -- (3.2,0.065) -- (3.2,0);
\draw[domain=0.02:3, smooth, line width=1.2pt, color=NavyBlue] (0,0.03) -- plot (\x,{exp(-1/(2*\x))/(sqrt(6.28*\x*\x*\x)}) -- (3.2,0.065);
\draw[->] (-1,0) -- (3.5,0);
\draw[->] (-0,-0.25) -- (0,0.75);
\foreach \x in {1,2,3}
		\draw (\x,-0.05) -- (\x,0.05);
\draw (2,0.4) node {\textcolor{NavyBlue}{L\'evy}};
\end{scope}
\end{tikzpicture}	
\end{center}
\caption{Densities of the standard Gaussian, Cauchy and L\'evy distribution.}
\end{figure}

We recall that mod-$\phi$ convergence on an open subset $D$ of $\C$ containing $0$
can only occur when the characteristic function of $\phi$
is analytic around $0$.
Among stable distributions, only Gaussian laws (which correspond to $\alpha=2$) satisfy this property.
Mod-$\phi$ convergence on $D=\I \R$ can however be considered for any stable distribution $\phi$.\bigskip

Since $|\E^{\eta(\I\xi)}|=\E^{-|c\xi|^\alpha}$ is integrable, any stable law $\phi_{c,\alpha,\beta}$ has a density $m_{c,\alpha,\beta}(x)\DD{x}$ with respect to the Lebesgue measure. 
Moreover, the corresponding L\'evy exponents have the following scaling property: for any $t>0$,
$$t \,\eta_{c,\alpha,\beta}\left(\frac{\I\xi}{t^{1/\alpha}}\right)=
\begin{cases}
\eta_{c,\alpha,\beta}(\I\xi)&\text{if }\alpha \neq 1,\\
\eta_{c,\alpha,\beta}(\I\xi) - \left(\frac{2c\beta}{\pi}\,\log t\right)\I\xi &\text{if }\alpha=1.
\end{cases}
$$
This will be used in the following proposition:
\begin{proposition}\label{prop:clt}
If $(X_n)_{n \in \N}$ converges in the mod-$\phi_{c,\alpha,\beta}$ sense, then
$$Y_n=\begin{cases}
\frac{X_n}{(t_n)^{1/\alpha}}&\text{if }\alpha \neq 1,\\
\frac{X_n}{t_n}-\frac{2c\beta}{\pi} \,\log t_n&\text{if }\alpha=1
\end{cases}$$
converges in law towards $\phi_{c,\alpha,\beta}$.
\end{proposition}

\begin{proof}
In both situations,
$$\esper[\E^{\I \xi Y_n}] = \E^{\eta(\I\xi)}\,\,\theta_n\!\left(\frac{\xi}{(t_n)^{1/\alpha}}\right)=\E^{\eta(\I\xi)}\,\theta(0)\,(1+o(1)) = \E^{\eta(\I\xi)}\,(1+o(1))$$
thanks to the uniform convergence of $\theta_n$ towards $\theta$, and to the scaling property of the L\'evy exponent $\eta$.\qed
\end{proof}

\section{Speed of convergence estimates}\label{sec:testfunctions}

The goal of this section is to introduce the notion of zone of control (Section~\ref{sec:zone})
and to estimate the speed of convergence in the resulting central limit theorem.
More precisely, we take as reference law a stable distribution $\phi_{c,\alpha,\beta}$
and a sequence $(X_n)_{n \in \N}$ that admits a zone of control (with respect to $\phi_{c,\alpha,\beta}$).
As for mod-$\phi_{c,\alpha,\beta}$ convergent sequences, it is easy to prove
that in this framework, an appropriate renormalization $Y_n$ of $X_n$ converges in distribution
towards $\phi_{c,\alpha,\beta}$ (see Proposition~\ref{prop:zone_CLT} below).\medskip

If $Y$ has distribution $\phi_{c,\alpha,\beta}$, we then want to estimate
\begin{equation}
  \dkol(Y_n,Y) = \sup_{s \in \R} |\proba[Y_n \leq s] - \proba[Y \leq s]|.
  \label{eq:dKol}
\end{equation}
To do this, we follow a strategy proposed by Tao (see \cite[Section 2.2]{Tao12})
in the case of sums of i.i.d.~random variables
with finite third moment.
The right-hand side of \eqref{eq:dKol} can be rewritten as
$$\sup_{f \in \mathcal{F}} |\esper[f(Y_n)] - \esper[f(Y)]|,$$
where $\mathcal{F}$ is the class of measurable functions $y \mapsto 1_{y \leq s}$. 
Therefore, it is natural to approach the problem of speed of convergence by looking at general estimates on test functions.
The basic idea is then to use the Parseval formula to compute the difference $\esper[f(Y_n)] - \esper[f(Y)]$, since we have estimates on the Fourier transforms of $Y_n$ and $Y$.
A difficulty comes from the fact that the functions $y \mapsto 1_{y \leq s}$ are not smooth,
and in particular, their Fourier transforms are only defined in the sense of distributions.
This caveat is dealt with by standard techniques of harmonic analysis (Sections \ref{subsec:distributions} to \ref{subsec:smoothing}): namely, we shall work in a space of distributions instead of functions, and use an adequate smoothing kernel in order to be able to work with compactly supported Fourier transforms.
Section~\ref{sec:berry} gathers all these tools to give an upper bound for \eqref{eq:dKol}.
This is the main result of this section and can be found in Theorem~\ref{thm:kolmogorov}.

\begin{remark}
  An alternative way to get an upper bound  for \eqref{eq:dKol} from estimates on characteristic functions
  is to use the following inequality due to Berry (see \cite{Berry1941} or \cite[Lemma XVI.3.2]{Fel71}).
  Let $X$ and $Y$ be random variables with characteristic functions $f^{*}(\zeta)$ and $g^{*}(\zeta)$.
  Then, provided that $Y$ has a density bounded by $m$, we have, for any $s \in \R$,
$$\left|\proba[X \leq s] - \proba[Y \leq s]\right|\leq \frac{1}{\pi} \int_{-T}^{T}
\left|\frac{f^{*}(\zeta)-g^{*}(\zeta)}{\zeta}\right|\DD{\zeta} +\frac{24m}{\pi T}.$$
 Using this inequality in our context should lead to similar estimates as the ones we obtain,
 possibly with different constants.
 The proof we use here however has the advantage of being more self-contained,
 and to provide estimates for test functions as intermediate results.
\end{remark}

\subsection{The notion of zone of control}\label{sec:zone}

\begin{definition}\label{def:zone}
Let $(X_n)_{n \in \N}$ be a sequence of real random variables, $\phi_{c,\alpha,\beta}$ a reference stable law, and $(t_n)_{n \in \N}$ a sequence growing to infinity. Consider the following assertions:
\begin{enumerate}[label=(Z\arabic*)]
\item\label{hyp:firstcondition} Fix $v > 0$, $w > 0$ and $\gamma \in \R$. There exists a zone $[-K(t_n)^\gamma,K(t_n)^\gamma]$ such that, for all $\xi$ in this zone, if $\theta_n(\xi) = \esper[\E^{\I \xi X_n}]\,\E^{-t_n \eta_{c,\alpha,\beta}(\I\xi)}$, then
 $$|\theta_n(\xi)-1| \leq K_1|\xi|^v\,\exp(K_2|\xi|^w)$$ for some positive constants $K_1$ and $K_2$ that are independent of $n$.
\item\label{hyp:secondcondition} One has 
$$\alpha \leq w \quad;\quad -\frac{1}{\alpha}\leq  \gamma \leq \frac{1}{w-\alpha}\quad;\quad 0<K \leq \left(\frac{c^\alpha}{2K_2}\right)^{\frac{1}{w-\alpha}}.$$ 
\end{enumerate}
Notice that \ref{hyp:secondcondition} can always be forced by increasing $w$, and then decreasing $K$ and $\gamma$ in the bounds of Condition \ref{hyp:firstcondition}. If Conditions \ref{hyp:firstcondition} and \ref{hyp:secondcondition} are satisfied, then we say that we have a \emph{zone of control} $[-K(t_n)^\gamma,K(t_n)^\gamma]$ with \emph{index} $(v,w)$. 
\end{definition}
Note that although the definition of zone of control depends on 
the reference law $\phi_{c,\alpha,\beta}$,
the latter does not appear in the terminology
(throughout the paper, it is considered fixed).

\begin{proposition}\label{prop:zone_CLT}
Let $(X_n)_{n \in \N}$ be a sequence of random variables, $\phi_{c,\alpha,\beta}$ a reference stable law, $Y$ with distribution $\phi_{c,\alpha,\beta}$ and $Y_n$ as in Proposition \ref{prop:clt}.
Assume that $(X_n)_{n \in \N}$ has a zone of control $[-K(t_n)^\gamma,K(t_n)^\gamma]$ with index $(v,w)$.
If $\gamma>-\frac{1}{\alpha}$, then one has the convergence in law $Y_n \rightharpoonup Y$.
\end{proposition}

\begin{proof}
Condition \ref{hyp:firstcondition} implies that, if $Y_n$ is the renormalization of $X_n$ and $Y \sim \phi_{c,\alpha,\beta}$, then for fixed $\xi$,
\begin{align*}
\left|\frac{\esper[\E^{\I \xi Y_n}]}{\esper[\E^{\I \xi Y}]}-1\right| &= \left|\theta_n\!\left(\frac{\xi}{(t_n)^{1/\alpha}}\right)-1\right| \leq  \frac{K_1|\xi|^v}{(t_n)^{v/\alpha}}\,\exp\left(\frac{K_2|\xi|^w}{(t_n)^{w/\alpha}}\right)
\end{align*}
for $t_n$ large enough, and the right-hand side goes to $0$. This proves the convergence in law $Y_n \rightharpoonup Y$.\qed
\end{proof}

The goal of the next few sections will be to get some speed of convergence estimates
for this convergence in distribution.

\begin{remark}
In the definition of zone of control, we do not assume the mod-$\phi_{c,\alpha,\beta}$ convergence of the sequence $(X_{n})_{n \in \N}$ with parameters $(t_n)_{n \in \N}$ and limit $\lim_{n \to \infty} \theta_n(\xi) = \theta(\xi)$. However, in almost all the examples considered, we shall indeed have (complex) mod-$\phi$ convergence (convergence of the residues $\theta_n$), with the same parameters $t_n$ as for the notion of zone of control. We shall then speak of \emph{mod-$\phi$ convergence with a zone of convergence $[-K(t_n)^\gamma,K(t_n)^\gamma]$ and with index of control $(v,w)$}. Mod-$\phi$ convergence implies other probabilistic results than estimates of Berry--Esseen type: central limit theorem with a large range of normality, moderate deviations (\emph{cf.} \cite{FMN16}), local limit theorem (\cite{DKN11}), \emph{etc.}
\end{remark}

\begin{remark}
If one has mod-$\phi_{c,\alpha,\beta}$ convergence of $(X_n)_{n \in \N}$, then there is at least a zone of convergence $[-K,K]$ of index $(v,w)=(0,0)$, with $\gamma=0$; indeed, the residues $\theta_n(\xi)$ stay locally bounded under this hypothesis. Thus, Definition \ref{def:zone} is an extension of this statement. However, we allow in the definition the exponent $\gamma$ to be negative (but not smaller than $-\frac{1}{\alpha}$). Indeed, in the computation of Berry--Esseen type bounds, we shall sometimes need to work with smaller zones than the one given by mod-$\phi$ convergence,
see the hypotheses of Theorem~\ref{thm:kolmogorov}, and Sections \ref{subsec:compoundpoisson} and \ref{subsec:ornstein} for examples.
\end{remark}

\begin{remark}
In our definition of zone of control, we ask for a bound on $|\theta_n(\xi)-1|$ that holds for any $n \in \N$. Of course, if the bound is only valid for $n \geq n_0$ large enough, then the corresponding bound on the Kolmogorov distance (Theorem \ref{thm:kolmogorov}) will only hold for $n \geq n_0$.
\end{remark}

\subsection{Spaces of test functions}\label{subsec:distributions}
Until the end of Section \ref{sec:testfunctions}, all the spaces of functions considered will be spaces of complex valued functions on the real line. If $f \in \leb^1(\R)$, we denote its Fourier transform
$$\widehat{f}(\xi)=\int_{\R} \E^{\I x\xi}\,f(x)\DD{x}.$$
Recall that the Schwartz space $\schwartz(\R)$ is by definition the space of 
infinitely differentiable functions whose derivatives tend to $0$ at infinity faster than any power of $x$.
Restricted to $\schwartz(\R)$, the Fourier transform is an automorphism, and it satisfies the Parseval formula
$$\forall f,g \in \schwartz(\R),\\ \int_{\R} f(x) \,\overline{g(x)}\DD{x} = \frac{1}{2\pi} \int_{\R} \widehat{f}(\xi)\,\overline{\widehat{g}(\xi)}\DD{\xi}.$$
We refer to \cite[Chapter VIII]{Lang93} and \cite[Part II]{Rud91} for a proof of this formula, and for the theory of Fourier transforms. The Parseval formula allows to extend by duality and/or density the Fourier transform to other spaces of functions or distributions. In particular, if $f \in \leb^2(\R)$, then its Fourier transform $\widehat{f}$ is well defined in $\leb^2(\R)$, although in general the integral $\int_{\R} \E^{\I x\xi}\,f(x)\DD{x}$ does not converge;  and we have again the Parseval formula
$$\forall f,g \in \leb^2(\R),\\ \int_{\R} f(x) \,\overline{g(x)}\DD{x} = \frac{1}{2\pi} \int_{\R} \widehat{f}(\xi)\,\overline{\widehat{g}(\xi)}\DD{\xi},$$
which amounts to the fact that $f \mapsto \frac{1}{\sqrt{2\pi}}\,\widehat{f}$ is an isometry of $\leb^2(\R)$ (see \cite[\S7.9]{Rud91}).\medskip

We denote $\mathscr{M}^1(\R)$ the set of probability measures on Borel subsets of $\R$.
In the sequel, we will need to apply a variant of Parseval's formula,
where $\overline{g(x)}\DD{x}$ is replaced by $\mu(\!\DD{x})$, with $\mu$ in $\mathscr{M}^1(\R)$.
This is given in the following lemma (see \cite[Lemma 2.3.3]{Stroock}, or \cite[p. 134]{Mal95}).
\begin{lemma}\label{lem:parsevalmeasures}
  For any function $f \in \leb^1(\R)$ with $\widehat{f} \in \leb^1(\R)$, and any Borel probability measure
$\mu \in \mathscr{M}^{1}(\R)$, the pairing $\scal{\mu}{f} = \int_{\R} f(x)\,\mu(\!\DD{x})$
is well defined, and the Parseval formula holds:
$$\int_{\R} f(x)\,\mu(\!\DD{x}) = \frac{1}{2\pi}\int_{\R} \widehat{f}(\xi)\,\widehat{\mu}(-\xi)\DD{\xi},$$
where $\widehat{\mu}(\xi) = \int_{\R} \E^{\I \xi x}\,\mu(\!\DD{x})$. The formula also holds for finite signed measures.
\end{lemma}
\medskip

Let us now introduce two adequate spaces of test functions, for which we shall be able to prove speed of convergence estimates. We first consider functions $f \in \leb^1(\R)$ with compactly supported Fourier transforms:
\begin{definition}\label{def:T0}
We call \emph{smooth test function of order $0$}, or simply \emph{smooth test function} an element $f \in \leb^1(\R)$ whose Fourier transform is compactly supported.
We denote $\testf_0(\R)$ the subspace of $\leb^1(\R)$ that consists in smooth test functions; it is an ideal for the convolution product.
\end{definition}

\begin{example}
If $$\sinc(x):=\frac{\sin x}{x}=\frac{1}{2}\int_{-1}^1 \E^{\I x\xi}\DD{\xi},$$
then by Fourier inversion $\widehat{\sinc}(\xi) = \pi\,1_{|\xi| \leq 1}$ is compactly supported on $[-1,1]$. Therefore, $f(x) = (\sinc(x))^2$ is an element of $\leb^1(\R)$ whose Fourier transform is compactly supported on $[-1,1]+[-1,1]=[-2,2]$, and $f \in \testf_0(\R)$.
\end{example}

Let us comment a bit Definition \ref{def:T0}. If $f$ is in $\testf_0(\R)$, then its Fourier transform
$\widehat{f}$ is bounded by $\|f\|_{\leb^1}$ and vanishes outside an interval $[-C,C]$, so $\widehat{f} \in \leb^1(\R)$. Since $f$ and $\widehat{f}$ are integrable, we can apply Lemma~\ref{lem:parsevalmeasures} with $f$.
Moreover, $f$ is then known to satisfy the Fourier inversion formula (see \cite[\S7.7]{Rud91}):
$$f(x) = \frac{1}{2\pi}\int_{\R} \widehat{f}(\xi)\,\E^{-\I \xi x}\DD{\xi}.$$
As the integral above is in fact on a compact interval $[-C,C]$, the standard convergence theorems ensure that $f$ is infinitely differentiable in $x$, hence the term "smooth". Also, by applying the Riemann--Lebesgue lemma to the continuous compactly supported functions $\xi \mapsto (-\I \xi)^k\,\widehat{f}(\xi)$, one sees that $f(x)$ and all its derivatives $f^{(k)}(x)$ go to $0$ as $x$ goes to infinity.
To conclude, $\testf_0(\R)$ is included in 
the space $\mathscr{C}_0^\infty(\R)$ of smooth functions whose derivatives all vanish at infinity.
\bigskip

Actually, we will need to work with more general test functions, defined by using the theory of tempered distributions. We endow the Schwartz space $\schwartz(\R)$ of smooth rapidly decreasing functions with its usual topology of metrizable locally convex topological vector space, defined by the family of semi-norms
$$\|f\|_{k,l} = \sum_{a \leq k}\sum_{b \leq l}\sup_{x \in \R} |x^a \,(\partial^b f)(x)| .$$
We recall that a tempered distribution $\psi$ is a continuous linear form $\psi : \schwartz(\R) \to \C$. The value of a tempered distribution $\psi$ on a smooth function $f$ will be denoted $\psi(f)$ or $\scal{\psi}{f}$. The space of all tempered distributions is classically denoted $\schwartz'(\R)$, and it is endowed with the $*$-weak topology.
The spaces of integrable functions, of square integrable functions and of probability measures 
can all be embedded in the space $\schwartz'(\R)$ as follows: if $f$ is a function in \hbox{$\leb^1(\R) \cup \leb^2(\R)$}, or if $\mu$ is in $\mathscr{M}^{1}(\R)$, then we associate to them the distributions
$$\scal{f}{g}=\int_\R f(x) g(x) \DD{x}\qquad;\qquad \scal{\mu}{g} = \int_\R g(x)\, \mu(\!\DD{x}).$$
We then say that these distributions are represented by the function $f$ and by the measure $\mu$.\medskip

The Fourier transform of a tempered distribution $\psi$ is defined by duality: it is the unique tempered distribution $\widehat{\psi}$ such that
$$\scal{\widehat{\psi}}{f} = \scal{\psi}{\widehat{f}}$$
for any $f \in \schwartz(\R)$. This definition agrees with the previous definitions of Fourier transforms for integrable functions, square integrable functions, or probability measures (all these elements can be paired with Schwartz functions).
Similarly, if $\psi$ is a tempered distribution, then one can also define by duality its derivative: thus, $\partial\psi$ is the unique tempered distribution such that
$$\scal{\partial\psi}{f} = -\scal{\psi}{\partial f}$$
for any $f \in \schwartz(\R)$. The definition agrees with the usual one when $\psi$ comes from a derivable function, by the integration by parts formula. On the other hand, Fourier transform and
derivation define linear endomorphisms of $\schwartz'(\R)$; also note that the Fourier transform is bijective.
\begin{definition}
A \emph{smooth test function of order $1$}, or \emph{smooth test distribution} is a distribution $f \in \schwartz'(\R)$, such that $\partial f$ is in $\testf_0(\R)$, that is to say that the distribution $\partial f$ can be represented by an integrable function with compactly supported Fourier transform. We denote $\testf_1(\R)$ the space of smooth test distributions.
\end{definition}

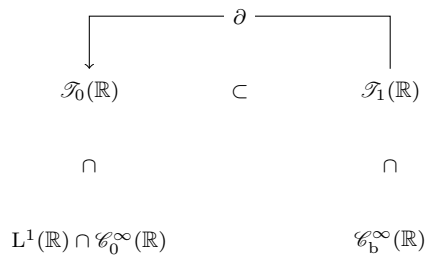
\begin{figure}[ht]
\begin{center}
\begin{tikzpicture}[scale=1]
\draw (0,0) node {$\testf_0(\R)$};
\draw (4,0) node {$\testf_1(\R)$};
\draw (2,0) node {$\subset$};
\draw [->] (4,0.3) -- (4,1) -- (0,1) -- (0,0.3);
\fill [white] (1.8,0.8) rectangle (2.2,1.2);
\draw (2,1) node {$\partial$};
\draw (0,-1) node {$\cap$};
\draw (0,-2) node {$\leb^1(\R) \cap \mathscr{C}^\infty_0(\R)$};
\draw (4,-1) node {$\cap$};
\draw (4,-2) node {$\mathscr{C}^\infty_\mathrm{b}(\R)$};
\end{tikzpicture}
\caption{The two spaces of test functions $\testf_0(\R)$ and $\testf_1(\R)$.}
\end{center}
\end{figure}
We now discuss Parseval's formula for functions in $\testf_1(\R)$.
\begin{proposition}\label{prop:parsevaldistribution}
Any smooth test distribution $f \in \testf_1(\R)$ can be represented by a bounded function in $\mathscr{C}^\infty(\R)$. 
Moreover, for any smooth test distribution in $\testf_1(\R)$:
 \begin{enumerate}[label=(TD\arabic*)]
    \item\label{item:distribution1} If $\mu$ is a Borel probability measure, then the pairing $\scal{\mu}{f} = \int_{\R} f(x)\,\mu(\!\DD{x})$ is well defined.
    \item\label{item:distribution2} The tempered distribution $\widehat{f}$ can be paired with the Fourier transform $\widehat{\mu}$ of a probability measure with finite first moment, in a way that extends the pairing between $\schwartz'(\R)$ and $\schwartz(\R)$ when $\mu$ (and therefore $\widehat{\mu}$) is given by a Schwartz density.
    \item\label{item:distribution3} The Parseval formula holds: if $f \in \testf_1(\R)$ and $\mu$ has finite expectation, then
    $$\scal{\mu}{f} = \frac{1}{2\pi}\,\scal{\widehat{f}}{\overline{\widehat{\mu}}} .$$
 \end{enumerate}
\end{proposition}

\begin{proof}
We start by giving a better description of the tempered distributions $f$ and $\widehat{f}$. Denote $\phi  = \partial f$; by assumption, this tempered distribution can be represented by a function $\phi \in \testf_0(\R)$, which in particular is of class $\mathscr{C}^\infty$ and integrable. Set
\begin{align*}
\widetilde{f}(x) &= \int_{y=0}^x \phi(y)\DD{y}.
\end{align*}
This is a function of class $\mathscr{C}^\infty$, whose derivative is $\phi$, and which is bounded since $\phi$ is integrable. Therefore, it is a tempered distribution, and for any $g\in \schwartz(\R)$,
$$\scal{\partial f}{g} = \scal{\phi}{g} = \scal{\partial\widetilde{f}}{g}.$$
We conclude that $\partial(f-\widetilde{f})$ is the zero distribution. It is then a standard result that, given a tempered distribution $\psi$, one has $\partial \psi = 0$ if and only if $\psi$ can be represented a constant. So, 
$$ f(x) = \int_{y=0}^x \phi(y)\DD{y} + f(0).$$ 
This shows in particular that $f$ is a smooth bounded function.
\medskip

A similar description can be provided for $\widehat{f}$. Recall that the principal value distribution, denoted $\mathrm{pv}(\frac{1}{x})$, is the tempered distribution defined for any $g \in \schwartz(\R)$ by
$$\scal{\mathrm{pv}\!\left(\frac{1}{x}\right)}{g} = \lim_{\eps \to 0} \left(\int_{|x|\geq \eps} \frac{g(x)}{x}\DD{x}\right).$$
The existence of the limit is easily proved by making a Taylor expansion of $g$ around $0$. 
Denote 
\begin{align*}
\schwartz^{[1]} &= \{g \in \schwartz(\R)\,|\,g(x) = x\,h(x) \text{ with }h \in \schwartz(\R)\} ;\\
\schwartz_{[1]} &= \{g \in \schwartz(\R)\,|\,g = \partial h \text{ with }h \in \schwartz(\R)\} ;
\end{align*}
the Fourier transform establishes an homeomorphism between $\schwartz^{[1]}$ and $\schwartz_{[1]}$, and the restriction of $\mathrm{pv}(\frac{1}{x})$ to $\schwartz^{[1]}$ is
$$\scal{\mathrm{pv}\!\left(\frac{1}{x}\right)}{g} = \int_{\R} \frac{g(x)}{x}\DD{x}.$$
Let $\widehat{g}(\xi)$ be an element of $\schwartz^{[1]}$, which we write as $\widehat{g}(\xi) = (-\I\xi)\,\widehat{h}(\xi)$ for some $h \in \schwartz(\R)$. This is equivalent to $g(x) = (\partial h)(x)$.
Let us denote $g_-(x) = g(-x)$, $h_-(x) = h(-x)$,
and $\mathrm{pv}(\frac{\I\,\widehat{\phi}(\xi)}{\xi})$ the tempered distribution defined by 
$$\mathrm{pv}\!\left(\frac{\I\,\widehat{\phi}(\xi)}{\xi}\right) = \I\, \mathrm{pv}\!\left(\frac{1}{\xi}\right)\circ m_{\widehat{\phi}},$$
with $m_{\widehat{\phi}}:\schwartz(\R) \to \schwartz(\R)$ equal to the multiplication by $\widehat{\phi}$.
Then we can make the following computation:
\begin{align*}
\scal{\widehat{f}}{\widehat{g}} &= \scal{f}{\widehat{\widehat{g}}} = 2\pi\,\scal{f}{g_-} = - 2\pi\, \scal{f}{\partial h_-} = 2\pi\, \scal{\phi}{h_-} = \scal{\phi}{\widehat{\widehat{h}}}\\
&= \scal{\widehat{\phi}}{\widehat{h}} = \scal{\widehat{\phi}}{\frac{\I \widehat{g}(\xi)}{\xi}} = \scal{\mathrm{pv}\!\left(\frac{\I\,\widehat{\phi}(\xi)}{\xi}\right)}{\widehat{g}}.
\end{align*}
Thus, the tempered distributions $\widehat{f}$ and $\mathrm{pv}(\frac{\I\,\widehat{\phi}(\xi)}{\xi})$ agree on the codimension $1$ subspace $\schwartz^{[1]}$ of $\schwartz(\R)$. However, $\schwartz^{[1]}$ is also the space of functions in $\schwartz(\R)$ that vanish at $0$, so, if $g_0$ is any function in $\schwartz(\R)$ such that $g_0(0)=1$, then for $g \in \schwartz(\R)$,
\begin{align*}
\scal{\widehat{f}}{g} &= \scal{\widehat{f}}{g-g(0)g_0} + g(0)\scal{\widehat{f}}{g_0} \\
&= \scal{\mathrm{pv}\!\left(\frac{\I\,\widehat{\phi}(\xi)}{\xi}\right)}{g-g(0)g_0} + \scal{\widehat{f}}{g_0}\,\scal{\delta_0}{g} \\
&= \scal{\mathrm{pv}\!\left(\frac{\I\,\widehat{\phi}(\xi)}{\xi}\right)}{g} + \left(\scal{\widehat{f}-\mathrm{pv}\!\left(\frac{\I\,\widehat{\phi}(\xi)}{\xi}\right)}{g_0}\right)\,\scal{\delta_0}{g}\\
&=\scal{\mathrm{pv}\!\left(\frac{\I\,\widehat{\phi}(\xi)}{\xi}\right)+ L\,\delta_0}{g} 
\end{align*}
where $L$ is some constant. Thus,
$$\widehat{f}(\xi) = \mathrm{pv}\!\left(\frac{\I\,\widehat{\phi}(\xi)}{\xi}\right)+ L\,\delta_0$$
and a computation against test functions shows that 
$$L = 2\pi f(0) - \I\,\scal{\mathrm{pv}\!\left(\frac{1}{\xi}\right)}{\widehat{\phi}}.$$
The three parts of the proposition are now easily proven. For \ref{item:distribution1}, since $f(x)$ is smooth and bounded, we can indeed consider the convergent integral $\int_{\R} f(x)\,\mu(\!\DD{x})$. For  \ref{item:distribution2}, assuming that $\mu$ has a finite first moment, $\widehat{\mu}$ is a function of class $\mathscr{C}^1$ and with bounded derivative. The same holds for $\widehat{\phi}\widehat{\mu}$, and therefore, one can define
\begin{align*}
\int_{\R}\widehat{f}(\xi) \widehat{\mu}(-\xi) \DD{\xi} &= \I \,\scal{\mathrm{pv}\!\left(\frac{1}{\xi}\right)}{\widehat{\phi}\,\overline{\widehat{\mu}}} + L \\
&= \left(\lim_{\eps\to \infty}\int_{|x|\geq \eps} \frac{\I\,\widehat{\phi}(\xi)\,\widehat{\mu}(-\xi)}{\xi}\DD{\xi} \right)+L.
\end{align*}
Indeed, if $f \in \mathscr{C}^1(\R)$, then $\lim_{\eps \to \infty} \int_{1\geq |x|\geq \eps} \frac{f(x)}{x}\DD{x}$ always exists, as can be seen by replacing $f$ by its Taylor approximation at $0$. 
Finally, let us prove the Parseval formula \ref{item:distribution3}. The previous calculations show that
\begin{align*}
\frac{1}{2\pi}\int_{\R} \widehat{f}(\xi)\,\widehat{\mu}(-\xi)\DD{\xi} &=  \frac{\I}{2\pi} \,\scal{\mathrm{pv}\!\left(\frac{1}{\xi}\right)}{\widehat{\phi}\,\overline{\widehat{\mu}}-\widehat{\phi} } + f(0) \\
&=\lim_{\eps \to 0} \left(\frac{\I}{2\pi}\int_{|\xi|\geq \eps} \widehat{\phi}(\xi)\left(\frac{\widehat{\mu}(-\xi)-1}{\xi}\right)\DD{\xi}\right) + f(0) \\
&=\frac{1}{2\pi}\int_{\R} \widehat{\phi}(\xi)\left(\frac{\widehat{\mu}(-\xi)-1}{-\I\xi}\right)\DD{\xi} + f(0). 
\end{align*}
Indeed, the function $\xi \mapsto \frac{\widehat{\mu}(-\xi)-1}{-\I\xi}$ is continuous on $\R$ and bounded, with value $\frac{\widehat{\mu}'(0)}{\I} = \int_{\R} x\,\mu(\!\DD{x})$ at $\xi=0$; it can therefore be integrated against the function $\widehat{\phi}$ which is integrable (and even with compact support). On the other hand,
\begin{align*}
\int_{\R}f(x)\,\mu(\!\DD{x}) &= \int_{x \in \R} \int_{y=0}^x \phi(y)\DD{y}\,\mu(\!\DD{x})+f(0) \\
&=\int_{(x,y) \in \R^2} (1_{x>y>0} - 1_{x<y<0})\, \phi(y)\DD{y}\,\mu(\!\DD{x})+f(0) \\
&= \int_{y \in \R} \phi(y)\,F(y)\DD{y}+f(0),\quad\text{with }F(y) = \mu((y,\infty)) - 1_{y\leq 0}.
\end{align*}
One has $\int_{\R}|F(y)|\DD{y} = \int_{y=0}^\infty\mu((y,\infty)) + \int_{y=-\infty}^0 \mu((-\infty,y)) = \int_{\R} |x|\,\mu(\!\DD{x})$, which is finite. In the integral above, we can therefore consider $F(y)\DD{y}$ as a finite signed measure, and the Parseval formula applies by Lemma \ref{lem:parsevalmeasures}. One computes readily
$$\widehat{F}(\xi) = \frac{\widehat{\mu}(\xi)-1}{\I\xi},$$
which ends the proof.\qed
\end{proof}

\begin{remark}
  \label{rmk:parsevaldistribution}
The Parseval formula of Proposition \ref{prop:parsevaldistribution} extends readily to finite signed measures $\mu$ such that $\int_{\R}|x|\,|\mu|(\!\DD{x})<+\infty$. Actually, it is sufficient to have a finite signed measure $\mu$ such that
$$\frac{\widehat{\mu}(\xi)-\widehat{\mu}(0)}{\xi^v}$$
is bounded in a vicinity of $0$, for some $v>0$. Then, $(\widehat{\mu}(\xi)-\widehat{\mu}(0))/\xi$ is integrable in a neighborhood of $0$. This ensures that the distribution $f(x)$ (respectively, the distribution $\widehat{f}(\xi)$) can be evaluated against the measure $\mu(x)$ (respectively, against $\widehat{\mu}(\xi)$), and then, the proof of Parseval's formula is analogous to the previous arguments.
\end{remark}

\subsection{Estimates for test functions}
We now give an estimate of $\esper[f_n(Y_n)]-\esper[f_n(Y)]$, where $(f_n)_{n \in \N}$ is a sequence of test functions in $\testf_0(\R)$ or $\testf_1(\R)$, and $(Y_n)_{n\in\N}$ is a sequence of random variables associated to a sequence $(X_n)_{n\in \N}$ which has a zone of control.

\begin{proposition}\label{prop:zonedistribution}
Let $(X_n)_{n \in \N}$ be a sequence of random variables, $\phi_{c,\alpha,\beta}$ a reference stable law, $Y$ with law $\phi_{c,\alpha,\beta}$ and $Y_n$ as in Proposition \ref{prop:clt}. 
We assume that:
\begin{enumerate}[label=(\arabic*)]
     \item $(X_n)_{n\in \N}$ has a zone of control $[-K(t_n)^\gamma,K(t_n)^\gamma]$ with index $(v,w)$;
     \item $(f_n)_{n \in \N}$ is a sequence of smooth test functions in $\testf_0(\R)$,
such that the support of $\widehat{f}_n$ is included into $[-K(t_n)^{\gamma+1/\alpha},K(t_n)^{\gamma+1/\alpha}]$.
 \end{enumerate}  
Then,
$$|\esper[f_n(Y_n)] - \esper[f_n(Y)]|\leq C_0(c,\alpha,v)\,K_1\,\frac{\|f_n\|_{\leb^1}}{(t_n)^{v/\alpha}},$$
where $C_0(c,\alpha,v)=\frac{2^{\frac{v+1}{\alpha}}\,\Gamma((v+1)/\alpha)}{\pi\alpha\,c^{v+1}}$.\medskip

\noindent If instead of (2) we assume:
\begin{enumerate}[label=({\arabic*}'),start=2]
    \item $(f_n)_{n \in \N}$ is a sequence of smooth test distributions in $\testf_1(\R)$
such that if $\phi_n = \partial f_n$ is the derivative of the distribution $f_n$, then the support of $\widehat{\phi}_n$ is included into $[-K(t_n)^{\gamma+1/\alpha},K(t_n)^{\gamma+1/\alpha}]$.
\end{enumerate}
Then 
$$|\esper[f_n(Y_n)] - \esper[f_n(Y)]|\leq C_1(c,\alpha,v)\,K_1\,\frac{\|\phi_n\|_{\leb^1}}{(t_n)^{v/\alpha}},$$
where $C_1(c,\alpha,v)=\frac{2^{v/\alpha}\,\Gamma(v/\alpha)}{\pi\alpha\,c^{v}}$.
\end{proposition}

\begin{proof}
Consider first a sequence $(f_n)_{n \in \N}$ of test functions in $\testf_0(\R)$, which satisfies (2). Using Parseval formula and the zone of control assumption, we have
\begin{align*}
\esper[f_n(Y_n)]-\esper[f_n(Y)] &= \frac{1}{2\pi} \!\int_{-K(t_n)^{\gamma+\frac{1}{\alpha}}}^{K(t_n)^{\gamma+\frac{1}{\alpha}}} \!\!\widehat{f}_n(\xi)\,\E^{\eta(-\I\xi)}\left(\theta_n\!\left(-\xi/(t_n)^{\frac{1}{\alpha}}\right)-1\right)\!\DD{\xi};\\
|\esper[f_n(Y_n)]-\esper[f_n(Y)]| &\leq \frac{K_1\|\widehat{f}_n\|_\infty}{2\pi (t_n)^{v/\alpha}} \int_{-K(t_n)^{\gamma+\frac{1}{\alpha}}}^{K(t_n)^{\gamma+\frac{1}{\alpha}}} |\xi|^v\,\E^{-|c\xi|^\alpha + K_2 \left(\frac{|\xi|}{(t_n)^{1/\alpha}}\right)^w} \DD{\xi}.
\end{align*}
For $\xi$ in $[-K(t_n)^{\gamma+1/\alpha},K(t_n)^{\gamma+1/\alpha}]$,
since $(t_n)^{\gamma-1/(w-\alpha)}\leq 1$,
the second term in the exponent can be bounded as follows:
\begin{align*}
K_2 \left(\frac{|\xi|}{(t_n)^{1/\alpha}}\right)^w = K_2 |\xi|^{\alpha} \left(\frac{|\xi|}{(t_n)^{\frac{1}{\alpha}+\frac{1}{w-\alpha}}}\right)^{w-\alpha} &\leq K_2 |\xi|^{\alpha} \left(K (t_n)^{\gamma-\frac{1}{w-\alpha}}\right)^{w-\alpha} \\
&\leq \frac{|c\xi|^\alpha}{2}.
\end{align*}
This is compensated by the term $-|c\xi|^\alpha$ and, therefore,
\begin{align*}
|\esper[f_n(Y_n)]-\esper[f_n(Y)]| &\leq \frac{K_1\|\widehat{f}_n\|_\infty}{2\pi (t_n)^{v/\alpha}}\int_{\R} |\xi|^v \E^{-\frac{|c\xi|^\alpha}{2}}\DD{\xi} \\
&\leq \frac{2^{\frac{v+1}{\alpha}}\,K_1}{\pi\alpha\,c^{v+1} (t_n)^{v/\alpha}}\,\,\Gamma\!\left(\frac{v+1}{\alpha}\right)\,\|f_n\|_{\leb^1}.
\end{align*}
This ends the proof of the first case. For test distributions $f_n \in \testf_1(\R)$ which satisfies the condition $(2')$, let us introduce the signed measure $\mu = \proba_{Y_n}-\proba_Y$. One has $\widehat{\mu}(0)=0$, and by hypothesis, 
$$\left|\frac{\widehat{\mu}(\xi)}{\xi}\right| \leq \frac{K_1\,|\xi|^{v-1}}{(t_n)^{v/\alpha}}\,\E^{-|c\xi|^\alpha+K_2\left(\frac{|\xi|}{(t_n)^{1/\alpha}}\right)^w}.$$
Remark \ref{rmk:parsevaldistribution} applies, and thus,
$$|\esper[f_n(Y_n)]-\esper[f_n(Y)]|=\left|\scal{\mu}{f_n}\right| = \frac{1}{2\pi}\left|\scal{ \widehat{f_n}}{\overline{\widehat{\mu}}}\right| = \frac{1}{2\pi}\left|\int_{\R} \widehat{\phi_n}(\xi)\,\frac{\widehat{\mu}(-\xi)}{\xi}\DD{\xi}\right| .$$
From there, the computations are exacly the same as before, with an index $v-1$ instead of $v$.\qed
\end{proof}

\subsection{Smoothing techniques}\label{subsec:smoothing}
We now explain how to relate the estimates on test functions or distributions to estimates on the Kolmogorov  distance. The main tool with respect to this problem is the following:

\begin{lemma}\label{lem:kernelD1}
There exists a function $\rho$ (called kernel) on $\R$ with the following properties.
\begin{enumerate}
 	\item The kernel $\rho$ is non-negative, with $\int_{\R} \rho(x)\DD{x}=1$.
 	\item The support of $\widehat{\rho}$ is $[-1,1]$ (hence, $\rho$ is a test function in $
 \testf_0(\R)$).
 	\item The functions $\rho$ and $\widehat{\rho}$ are even, and 
 	$$\rho(K) \leq \min\left(\frac{3}{8\pi},\frac{96}{\pi\,K^4}\right).$$
 \end{enumerate} 
 \end{lemma}

\begin{figure}[ht]
\begin{center}
\begin{tikzpicture}[scale=0.8]
\draw[->] (-5.5,0) -- (5.5,0);
\draw[->] (-0,-0.25) -- (0,2);
\foreach \x in {-5,-4,-3,-2,-1,1,2,3,4,5}
        \draw (\x,-0.05) -- (\x,0.05);
\draw (3,0.5) node {\textcolor{NavyBlue}{$\rho_{\frac{1}{4}}(x)$}};
\draw[domain=0.1:5, smooth, line width=1.2pt, color=NavyBlue] (0,1.431) -- plot (\x,{1.431 * (sin(57.3*\x)/\x)^4 }); 
\draw[domain=-5:0.1, smooth, line width=1.2pt, color=NavyBlue] plot (\x,{1.431 * (sin(57.3*\x)/\x)^4 }) -- (0,1.431); 
\begin{scope}[shift={(0,-4.5)}]
\draw[->] (-5.5,0) -- (5.5,0);
\draw[->] (-0,-0.25) -- (0,3);
\foreach \x in {-5,-3,-2,-1,1,2,3,5}
        \draw (\x,-0.05) -- (\x,0.05);
\draw[line width=1.2pt, color=Red] (4,-0.05) -- (4,0.05);
\draw[line width=1.2pt, color=Red] (-4,-0.05) -- (-4,0.05);
\draw (-4,-0.5) node {$-4$};        
\draw (4,-0.5) node {$4$};
\draw (3,0.7) node {\textcolor{Red}{$\widehat{\rho_{\frac{1}{4}}}(\xi)$}};
\draw[domain=0:2, smooth, line width=1.2pt, color=Red] plot (\x,{8/3-\x*\x+\x*\x*\x/4}); 
\draw[domain=0:2, smooth, line width=1.2pt, color=Red] plot (-\x,{8/3-\x*\x+\x*\x*\x/4}); 
\draw[domain=2:4, smooth, line width=1.2pt, color=Red] plot (\x,{((4-\x)^3)/12}); 
\draw[domain=2:4, smooth, line width=1.2pt, color=Red] plot (-\x,{((4-\x)^3)/12}); 
\end{scope}     
\end{tikzpicture}
\caption{The smoothing kernel $\rho_{\frac{1}{4}}$, and its Fourier transform which is supported on $[-4,4]$.\label{fig:smooth}}
\end{center}
\end{figure}
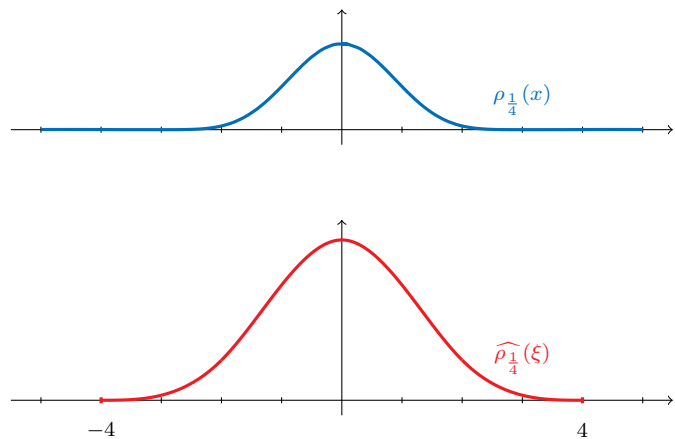

\begin{proof}
Set 
$$\rho(x) = \frac{3}{8\pi}\,\left(\sinc\left(\frac{x}{4}\right)\right)^4 .$$
It has its Fourier transform supported on $[-\frac{1}{4},\frac{1}{4}] + [-\frac{1}{4},\frac{1}{4}] + [-\frac{1}{4},\frac{1}{4}] + [-\frac{1}{4},\frac{1}{4}] = [-1,1]$. On the other hand, an easy computation gives $\int_{\R} \rho(x)\DD{x}=1$:
 use for example the Plancherel formula $$\int_{\R} |f(x)|^2\DD{x} = \frac{1}{2\pi}\,\int_{\R} |\widehat{f}(\xi)|^2\DD{\xi} $$
with $f(x)=\sinc(x)^2$ and thus $\widehat{f}(\xi) = \frac{1}{2\pi} \,\widehat{\sinc}*\widehat{\sinc}(\xi) = \frac{\pi}{2} (2-|\xi|)_+$.
Finally, $\sinc(x) \leq \min(1,\frac{1}{|x|})$, which leads to the inequality stated for $\rho(K)$.\qed
\end{proof}

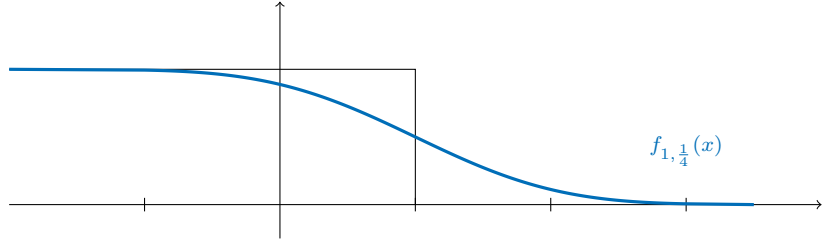
\begin{figure}[ht]
\begin{center}      
\begin{tikzpicture}[scale=1.8]
\draw[->] (-2,0) -- (4,0);
\draw[->] (-0,-0.25) -- (0,1.5);
\foreach \x in {-1,1,2,3}
        \draw (\x,-0.05) -- (\x,0.05);
\draw (-2,1) -- (1,1) -- (1,0);
\draw (3,0.4) node {\textcolor{NavyBlue}{$f_{1,\frac{1}{4}}(x)$}};
\draw [line width=1.2pt, color=NavyBlue](-2,1) -- (-1.0000, 0.99424) -- (-0.98000, 0.99381) -- (-0.96000, 0.99336) -- (-0.94000, 0.99287)
 -- (-0.92000, 0.99234) -- (-0.90000, 0.99177) -- (-0.88000, 0.99116) -- (-0.86000, 0.99051)
 -- (-0.84000, 0.98982) -- (-0.82000, 0.98907) -- (-0.80000, 0.98828) -- (-0.78000, 0.98743)
 -- (-0.76000, 0.98654) -- (-0.74000, 0.98558) -- (-0.72000, 0.98457) -- (-0.70000, 0.98349)
 -- (-0.68000, 0.98235) -- (-0.66000, 0.98115) -- (-0.64000, 0.97988) -- (-0.62000, 0.97853)
 -- (-0.60000, 0.97712) -- (-0.58000, 0.97562) -- (-0.56000, 0.97405) -- (-0.54000, 0.97240)
 -- (-0.52000, 0.97066) -- (-0.50000, 0.96884) -- (-0.48000, 0.96692) -- (-0.46000, 0.96492)
 -- (-0.44000, 0.96282) -- (-0.42000, 0.96063) -- (-0.40000, 0.95833) -- (-0.38000, 0.95594)
 -- (-0.36000, 0.95344) -- (-0.34000, 0.95083) -- (-0.32000, 0.94812) -- (-0.30000, 0.94529)
 -- (-0.28000, 0.94235) -- (-0.26000, 0.93930) -- (-0.24000, 0.93613) -- (-0.22000, 0.93283)
 -- (-0.20000, 0.92942) -- (-0.18000, 0.92588) -- (-0.16000, 0.92222) -- (-0.14000, 0.91843)
 -- (-0.12000, 0.91451) -- (-0.10000, 0.91046) -- (-0.080000, 0.90628) -- (-0.060000, 0.90197)
 -- (-0.040000, 0.89752) -- (-0.020000, 0.89294) -- (0.00000, 0.88822) -- (0.020000, 0.88336)
 -- (0.040000, 0.87837) -- (0.060000, 0.87323) -- (0.080000, 0.86796) -- (0.10000, 0.86255)
 -- (0.12000, 0.85700) -- (0.14000, 0.85131) -- (0.16000, 0.84549) -- (0.18000, 0.83952)
 -- (0.20000, 0.83342) -- (0.22000, 0.82717) -- (0.24000, 0.82079) -- (0.26000, 0.81428)
 -- (0.28000, 0.80763) -- (0.30000, 0.80084) -- (0.32000, 0.79393) -- (0.34000, 0.78688)
 -- (0.36000, 0.77971) -- (0.38000, 0.77240) -- (0.40000, 0.76497) -- (0.42000, 0.75742)
 -- (0.44000, 0.74975) -- (0.46000, 0.74196) -- (0.48000, 0.73406) -- (0.50000, 0.72604)
 -- (0.52000, 0.71792) -- (0.54000, 0.70968) -- (0.56000, 0.70135) -- (0.58000, 0.69291)
 -- (0.60000, 0.68438) -- (0.62000, 0.67576) -- (0.64000, 0.66705) -- (0.66000, 0.65825)
 -- (0.68000, 0.64937) -- (0.70000, 0.64042) -- (0.72000, 0.63139) -- (0.74000, 0.62229)
 -- (0.76000, 0.61314) -- (0.78000, 0.60392) -- (0.80000, 0.59465) -- (0.82000, 0.58532)
 -- (0.84000, 0.57596) -- (0.86000, 0.56655) -- (0.88000, 0.55711) -- (0.90000, 0.54764)
 -- (0.92000, 0.53814) -- (0.94000, 0.52862) -- (0.96000, 0.51909) -- (0.98000, 0.50955)
 -- (1.0000, 0.50000) -- (1.0200, 0.49045) -- (1.0400, 0.48091) -- (1.0600, 0.47137)
 -- (1.0800, 0.46186) -- (1.1000, 0.45236) -- (1.1200, 0.44289) -- (1.1400, 0.43344)
 -- (1.1600, 0.42404) -- (1.1800, 0.41467) -- (1.2000, 0.40535) -- (1.2200, 0.39608)
 -- (1.2400, 0.38686) -- (1.2600, 0.37770) -- (1.2800, 0.36861) -- (1.3000, 0.35958)
 -- (1.3200, 0.35062) -- (1.3400, 0.34175) -- (1.3600, 0.33295) -- (1.3800, 0.32424)
 -- (1.4000, 0.31561) -- (1.4200, 0.30708) -- (1.4400, 0.29865) -- (1.4600, 0.29031)
 -- (1.4800, 0.28208) -- (1.5000, 0.27395) -- (1.5200, 0.26594) -- (1.5400, 0.25803)
 -- (1.5600, 0.25024) -- (1.5800, 0.24257) -- (1.6000, 0.23502) -- (1.6200, 0.22759)
 -- (1.6400, 0.22029) -- (1.6600, 0.21311) -- (1.6800, 0.20607) -- (1.7000, 0.19915)
 -- (1.7200, 0.19237) -- (1.7400, 0.18572) -- (1.7600, 0.17920) -- (1.7800, 0.17282)
 -- (1.8000, 0.16658) -- (1.8200, 0.16047) -- (1.8400, 0.15451) -- (1.8600, 0.14868)
 -- (1.8800, 0.14299) -- (1.9000, 0.13744) -- (1.9200, 0.13203) -- (1.9400, 0.12676)
 -- (1.9600, 0.12163) -- (1.9800, 0.11663) -- (2.0000, 0.11178) -- (2.0200, 0.10706)
 -- (2.0400, 0.10248) -- (2.0600, 0.098028) -- (2.0800, 0.093714) -- (2.1000, 0.089533)
 -- (2.1200, 0.085485) -- (2.1400, 0.081566) -- (2.1600, 0.077776) -- (2.1800, 0.074113)
 -- (2.2000, 0.070576) -- (2.2200, 0.067162) -- (2.2400, 0.063870) -- (2.2600, 0.060698)
 -- (2.2800, 0.057643) -- (2.3000, 0.054704) -- (2.3200, 0.051878) -- (2.3400, 0.049164)
 -- (2.3600, 0.046558) -- (2.3800, 0.044058) -- (2.4000, 0.041663) -- (2.4200, 0.039369)
 -- (2.4400, 0.037174) -- (2.4600, 0.035076) -- (2.4800, 0.033072) -- (2.5000, 0.031160)
 -- (2.5200, 0.029337) -- (2.5400, 0.027600) -- (2.5600, 0.025947) -- (2.5800, 0.024375)
 -- (2.6000, 0.022882) -- (2.6200, 0.021465) -- (2.6400, 0.020121) -- (2.6600, 0.018849)
 -- (2.6800, 0.017645) -- (2.7000, 0.016507) -- (2.7200, 0.015433) -- (2.7400, 0.014419)
 -- (2.7600, 0.013465) -- (2.7800, 0.012567) -- (2.8000, 0.011722) -- (2.8200, 0.010929)
 -- (2.8400, 0.010186) -- (2.8600, 0.0094899) -- (2.8800, 0.0088387) -- (2.9000, 0.0082304)
 -- (2.9200, 0.0076628) -- (2.9400, 0.0071341) -- (2.9600, 0.0066422) -- (2.9800, 0.0061852) -- (3.5,0);
\end{tikzpicture}
\caption{The approximation $f_{1,\frac{1}{4}}$ of the Heaviside function $1_{(-\infty,1]}$.\label{fig:smoothheaviside}}
\end{center}
\end{figure}

In the following, for $\eps>0$, we set $\rho_{\eps}(x)=\frac{1}{\eps}\,\rho(\frac{x}{\eps})$, which has its Fourier transform compactly supported on $[-\frac{1}{\eps},\frac{1}{\eps}]$; see Figure \ref{fig:smooth}. We also denote $f_{a,\eps}(x)=f_\eps(x-a)$, where $f_\eps$ is the function $1_{(-\infty,0]} * \rho_\eps$; \emph{cf.} Figure \ref{fig:smoothheaviside}. 
For all $a,\eps$, $f_{a,\eps}$ is an approximation of the Heaviside function $1_{(-\infty,a]}$, and one has the following properties:
 \begin{proposition}
 The function $f_{a,\eps}$ is a smooth test distribution in $\testf_1(\R)$ whose derivative $\partial f_{a,\eps}$ has its Fourier transform compactly supported on $[-\frac{1}{\eps},\frac{1}{\eps}]$, and satisfies $\|\partial f_{a,\eps}\|_{\leb^1} = 1$. Moreover:
 \begin{enumerate}
 	\item The function $f_{a,\eps}$ has a non-positive derivative, and decreases from $1$ to $0$.
 	\item One has $f_1(x)=1-f_1(-x)$, and for all $K\geq 0$,
\begin{align*}
f_1(K)&=\int_{0}^\infty \rho(K+y)\DD{y} \leq \frac{32}{\pi\,K^3};\\
\int_0^\infty f_1(u-K)\DD{u} & \leq K + \int_{w=0}^\infty \min\left(1,\frac{32}{\pi\,w^3}\right)\DD{w} = K  + 3\sqrt[3]{\frac{3}{\pi}}.
\end{align*}
 \end{enumerate}
 \end{proposition}
 
\begin{proof}
The derivative of $f_{a,\eps}$ is
\begin{align*}
\partial(f_{a,\eps})(x) &= \partial(1_{(-\infty,a]}*\rho_\eps)(x) = \left(\partial(1_{(-\infty,a]}) * \rho_\eps\right)(x)  \\
&= \left(-\delta_a * \rho_\eps\right)(x) = -\rho_{\eps}(x-a),
\end{align*}
so it is indeed in $\testf_0(\R)$, and non-positive. Its Fourier transform is supported by $[-\frac{1}{\eps},\frac{1}{\eps}]$, and $\|\partial f_{a,\eps}\|_{\leb^1}=\int_{\R}\rho_{\eps}(x)\DD{x} = 1$. Then,
$$\lim_{x \to +\infty} f_{a,\eps}(x) = \lim_{x \to +\infty} f_{1}\!\left(\frac{x-a}{\eps}\right) = \lim_{y \to +\infty} f_1(y),$$
so $\lim_{x \to +\infty} f_{a,\eps}(x) = 0$.
Since by definition $f_1(y)=\int_y^\infty \rho(u) \DD{u}$,
the symmetry relation \hbox{$f_1(x)=1-f_1(-x)$} follows from $\rho$ even;
it implies the other limit statement $\lim_{x \to -\infty} f_{a,\eps}(x) = 1$.
The inequalities in part ii) are immediate consequences of those of Lemma \ref{lem:kernelD1}.\qed
\end{proof}

Let us now state a result which converts estimates on smooth test distributions into estimates of Kolmogorov distances. It already appeared in \cite[Lemma 16]{MN15}, and is inspired by \cite[p.~87]{Tao12} and \cite[Chapter XVI, \S3, Lemma 1]{Fel71}:
\begin{theorem}\label{theorem:tao}
Let $X$ and $Y$ be two random variables with cumulative distribution functions $F_X(a)=\proba[X \leq a]$ and $F_Y(a)=\proba[Y\leq a]$. Assume that for some $\eps>0$ and $B>0$,
$$\sup_{a \in \R} \, |\esper[f_{a,\eps}(X)]-\esper[f_{a,\eps}(Y)] | \leq B\eps.$$
We also suppose that $Y$ has a density w.r.t. Lebesgue measure that is bounded by $m$. Then, for every $\lambda>0$,
\begin{align*}
\dkol(X,Y)&=\sup_{a \in \R} |F_X(a)-F_Y(a)|\\
& \leq (1+\lambda)\left(B+\frac{m}{\sqrt[3]{\pi}}\left(4\sqrt[3]{1+\frac{1}{\lambda}}+3\sqrt[3]{3}\right)\right)\eps.
\end{align*}

\end{theorem}
\noindent The choice of the parameter $\lambda$ allows one to optimize constants according to the reference law of $Y$ and to the value of $B$. A general bound is obtained by choosing $\lambda=\frac{1}{2}$; this gives after some simplifications
	$$\dkol(X,Y) \leq \frac{3}{2}(B+7m)\,\eps,$$
which is easy to remember and manipulate.

\begin{proof}
For the convenience of the reader, we reproduce here the proof given in \cite{MN15}. Fix a positive constant $K$, and denote $\Delta=\sup_{a \in \R}|F_X(a)-F_Y(a)|$ the Kolmogorov distance between $X$ and $Y$. One has
\begin{align}F_X(a) = \esper[1_{X \leq a}] &\leq \esper[f_{a+K\eps,\eps}(X)] + \esper[(1-f_{a+K\eps,\eps}(X))\,1_{X \leq a}] \nonumber\\
&\leq \esper[f_{a+K\eps,\eps}(Y)] + \esper[(1-f_{a+K\eps,\eps}(X))\,1_{X \leq a}] + B\eps.
\label{eq:UB_FX}
\end{align}
The second expectation writes as
\begin{align*}
\esper[(1-&f_{a+K\eps,\eps}(X))\,1_{X \leq a}]=\int_{\R} (1-f_{a+K\eps,\eps}(x))\,1_{(-\infty,a]}(x)\, \proba_X(\!\DD{x})\\
&=-\int_{\R} ((1-f_{a+K\eps,\eps}(x))\,1_{(-\infty,a]}(x))'\, F_X(x)\DD{x} = A_1+A_2,
\end{align*}
where $A_1=(1-f_{a+K\eps,\eps}(a))\,F_X(a)$, $A_2=\int_{\R} f_{a+K\eps,\eps}'(x)\,1_{(-\infty,a]}(x)\, F_X(x)\DD{x}$.
Indeed, in the space of tempered distributions,
$((1-f_{a+K\eps,\eps}(x))\,1_{(-\infty,a]}(x))' = -(1-f_{a+K\eps,\eps}(x))\,\delta_a(x)-f'_{a+K\eps,\eps}(x)\,1_{(-\infty,a]}(x).$
We evaluate the two terms $A_1$ and $A_2$ as follows:
\begin{itemize}
	\item Since $F_X(a) \leq F_Y(a)+\Delta$, 
\begin{align*}
A_1 &\leq  (1-f_{a+K\eps,\eps}(a))\, F_Y(a) + (1-f_{a+K\eps,\eps}(a))\,\Delta \\
&\leq \int_{\R} (1-f_{a+K\eps,\eps}(x))\,\delta_a(x)\, F_Y(x)\DD{x} + (1-f_{1}(-K))\,\Delta .
\end{align*}
\item For $A_2$, since $F_X(x)\geq F_Y(x)-\Delta$ and the derivative of $f_{a+K\eps,\eps}$ is negative, an upper bound is obtained as follows:
\begin{align*}A_2 &\le \int_{\R} f_{a+K\eps,\eps}'(x)\,1_{(-\infty,a]}(x)\, F_Y(x)\DD{x} - \Delta\int_{\R}f_{a+K\eps,\eps}'(x)\,1_{(-\infty,a]}(x)\DD{x} \\
&= \int_{\R} f_{a+K\eps,\eps}'(x)\,1_{(-\infty,a]}(x)\, F_Y(x)\DD{x} +  (1-f_{a+K\eps,\eps}(a))\,\Delta\\
&=\int_{\R} f_{a+K\eps,\eps}'(x)\,1_{(-\infty,a]}(x)\, F_Y(x)\DD{x} + (1-f_{1}(-K))\,\Delta.
\end{align*}
\end{itemize}
Therefore, by gathering the bounds on $A_1$ and $A_2$, we get
\begin{equation}
  \esper[(1-f_{a+K\eps,\eps}(X))\,1_{X \leq a}]\leq \esper[(1-f_{a+K\eps,\eps}(Y))\,1_{Y \leq a}]+2 (1-f_1(-K))\,\Delta.
  \label{eq:Tech1}
\end{equation}
On the other hand, if $m$ is a bound on the density $f_Y$ of $Y$, then
\begin{align*}\esper[f_{a+K\eps,\eps}(Y)\,1_{Y \geq a}] &= \int_a^{\infty} f_{a+K\eps,\eps}(y)\,f_Y(y)\DD{y}\\
&\leq m\int_{a}^\infty f_{\eps}(y-a-K\eps)\DD{y} = m\int_0^\infty f_\eps(y-K\eps)\DD{y}\\\
&\leq m\eps \int_0^\infty f_1(u-K)\DD{u} \leq m \eps \left(K+3\sqrt[3]{\frac{3}{\pi}}\right);\end{align*}
and
\begin{align}
\esper[f_{a+K\eps,\eps}(Y)] &\leq \esper[f_{a+K\eps,\eps}(Y)\,1_{Y \leq a}] + m\left(K+3\sqrt[3]{\frac{3}{\pi}}\right)\eps.
\label{eq:Tech2}
\end{align}
Putting together Eqs. \eqref{eq:UB_FX}, \eqref{eq:Tech1} and \eqref{eq:Tech2}, we get
\[
F_X(a) \leq F_Y(a) + \left(B+m\left(K+3\sqrt[3]{\frac{3}{\pi}}\right)\right) \eps + \frac{64}{\pi\,K^3}\,\Delta.
\]
Similarly, $F_X(a)\geq F_Y(a) - \left(B+m(K+3\sqrt[3]{\frac{3}{\pi}})\right) \eps - \frac{64}{\pi\,K^3}\,\Delta$, so in the end
$$\Delta = \sup_{a \in \R}|F_X(a)-F_Y(a)| \leq \left(B+m\left(K+3\sqrt[3]{\frac{3}{\pi}}\right)\right)\eps + \frac{64}{\pi\,K^3}\,\Delta.$$
As this is true for every $K$, setting $K=4\sqrt[3]{\frac{1+\lambda}{\pi\lambda}}$ with $\lambda>0$, one obtains
\begin{equation*}
\Delta \leq (1+\lambda)\left(B+\frac{m}{\sqrt[3]{\pi}}\left(4\sqrt[3]{1+\frac{1}{\lambda}}+3\sqrt[3]{3}\right)\right)\eps.\qquad\qed
\end{equation*}
\end{proof}
\medskip

In the next Section \ref{sec:berry}, we shall combine Theorem \ref{theorem:tao} and the estimates on smooth test distributions given by Proposition \ref{prop:zonedistribution} to get a Berry--Esseen type bound on the Kolmogorov distances in the setting of mod-$\phi$ convergence.

\subsection{Bounds on the Kolmogorov distance}
\label{sec:berry}
We are now ready to get an estimate for the Komogorov distance
under a zone of control hypothesis.

\begin{theorem}\label{thm:kolmogorov}
Fix a reference stable distribution $\phi_{c,\alpha,\beta}$
and consider a sequence $(X_n)_{n \in \N}$ of random variables
with a zone of control \hbox{$[-K(t_n)^\gamma,K(t_n)^\gamma]$} of \emph{index} $(v,w)$.
Assume in addition that $ \gamma \leq \tfrac{v-1}{\alpha}$.
As before, we denote $Y$ a random variable with law $\phi_{c,\alpha,\beta}$, and $Y_n$ the renormalization of $X_n$ as in Proposition \ref{prop:clt}.
Then, there exists a constant $C(\alpha,c,v,K,K_1)$ 
such that
$$\dkol(Y_n,Y) \leq C(\alpha,c,v,K,K_1)\,
\frac{1}{(t_n)^{1/\alpha+\gamma}}.$$
The constant $C(\alpha,c,v,K,K_1)$ is explicitly given by
\[\min_{\lambda>0} \left(\frac{1+\lambda}{\alpha \pi\,c}\left(\frac{2^\frac{v}{\alpha}\,\Gamma(\frac{v}{\alpha})\,K_1}{c^{v-1}} + \frac{\Gamma(\frac{1}{\alpha})}{\sqrt[3]{\pi}\,K }\left(4\sqrt[3]{1+\frac{1}{\lambda}}+3\sqrt[3]{3}\right)\right)\right).\]
\end{theorem}
\noindent 
Note that the additional hypothesis $ \gamma \leq \tfrac{v-1}{\alpha}$ can always be ensured by decreasing $\gamma$
(but this makes the resulting bound weaker).
 
\begin{proof}
We apply Proposition \ref{prop:zonedistribution} with the smooth test distributions $f_n=f_{a,\eps_n}$,
with $\eps_n:=\tfrac{1}{K\,(t_n)^{1/\alpha+\gamma}}$;
we know that $\|\partial f_n\|_{\leb^1}=1$ and that
$\widehat{f_n}$ is supported by the zone $[-K(t_n)^{1/\alpha+\gamma},K(t_n)^{1/\alpha+\gamma}]$, so that
\begin{equation*}
\left|\esper[f_{a,\eps_n}(Y_n)]-\esper[f_{a,\eps_n}(Y)]\right|\leq C_2(c,\alpha,v)\,\frac{K_1}{(t_n)^{v/\alpha}} \leq \frac{2^\frac{v}{\alpha}\,\Gamma(\frac{v}{\alpha})\,K_1\,K}{\alpha\pi\,c^v}\,\eps_n.
\end{equation*}
This allows to apply Theorem \ref{theorem:tao} with a constant 
$$B = \frac{2^\frac{v}{\alpha}\,\Gamma(\frac{v}{\alpha})\,K_1\,K}{\alpha\pi\,c^v},$$ and with $\eps=\eps_n = \tfrac{1}{K\,(t_n)^{1/\alpha+\gamma}}$. Indeed, note that the density of the law of $Y$ is bounded by 
\begin{equation*}
m=\frac{1}{2\pi}\|\E^{\eta(\I\xi)}\|_{\leb^1}=\frac{1}{\alpha\pi\,c}\,\,\Gamma\!\left(\frac{1}{\alpha}\right).\qquad\qed
\end{equation*}
\end{proof}
\begin{remark}
Suppose $\alpha=2$, $c = \frac{1}{\sqrt{2}}$ (mod-Gaussian convergence), and $v=w=3$. The maximal value allowed for the exponent $\gamma$ in the size of the zone of control is then $\gamma = 1$, and later we shall encounter many examples of this situation. Then, we obtain
\begin{equation}
  \dkol(Y_n,Y) \leq \frac{1+\lambda}{\sqrt{2 \pi}}\left(2^\frac{3}{2}\,K_1 + \frac{1}{\sqrt[3]{\pi}\,K }\left(4\sqrt[3]{1+\frac{1}{\lambda}}+3\sqrt[3]{3}\right)\right)\frac{1}{(t_n)^{3/2}}.
  \label{eq:dKol_MG}
\end{equation}
In Section \ref{sec:dependency}, we shall give conditions on cumulants of random variables that lead to mod-Gaussian convergence with a zone of control of size $O(t_n)$ and with index $(3,3)$, so that \eqref{eq:dKol_MG} holds.
We shall then choose $K$, $K_1$ and $\lambda$ to make the constant in the right-hand side as small as possible.
\end{remark}
\begin{remark}
In the general case, taking $\lambda=\frac{1}{2}$ in Theorem \ref{thm:kolmogorov} leads to the inequality
$$\dkol(Y_n,Y) \leq C_3(\alpha,c,v,K_1,K)\,\frac{1}{(t_n)^{1/\alpha +\gamma}},$$
where $C_3(\alpha,c,v,K_1,K) = \frac{3}{2\pi\,\alpha\,c}\left(\frac{2^{\frac{v}{\alpha}}\,\Gamma(\frac{v}{\alpha})\,K_1}{c^{v-1}}+\frac{7\,\Gamma(\frac{1}{\alpha})}{K}\right)$. 
\end{remark}

\section{Examples with an explicit Fourier transform}
\label{Sect:First}

\subsection{Sums of independent random variables}\label{subsec:independent}
As a direct application of Theorem \ref{thm:kolmogorov}, one recovers the classical Berry--Esseen estimates. Let $(A_n)_{n \in \N}$ be a sequence of centered i.i.d.~random variables with a third moment. We denote $\esper[(A_i)^2]=\sigma^2$ and $\esper[|A_i|^3]=\rho$. Set $S_n=\sum_{i=1}^n A_i$, $X_n=S_n/(\sigma n^{1/3})$,
$$t_n=n^{1/3}\qquad;\qquad K=\frac{\sigma^3}{\rho}\qquad;\qquad v=w=3 \qquad;\qquad \gamma=1.$$
Notice that $K \leq 1$ as a classical application of Hölder inequality.
On the zone $\xi \in [-Kn^{1/3},Kn^{1/3}]$, we have:
\begin{align*}
|\theta_n&(\xi)-1| = \left|\left(\esper\!\left[\E^{\I \xi \frac{A_1}{\sigma n^{1/3}}}\right]\,\E^{\frac{\xi^2}{2n^{2/3}}}\right)^n-1\right| \\
&\leq n\left|\esper\!\left[\E^{\I \xi \frac{A_1}{\sigma n^{1/3}}}\right]\,\E^{\frac{\xi^2}{2n^{2/3}}}-1\right|\,\left(\max\left(\left|\esper\!\left[\E^{\I \xi \frac{A_1}{\sigma n^{1/3}}}\right]\,\E^{\frac{\xi^2}{2n^{2/3}}}\right|,1\right)\right)^{n-1}.
\end{align*}
For any $t$, $|\E^{\I t} - 1 - \I t + \frac{t^2}{2}| \leq \frac{|t^3|}{6}$, so
\begin{align*}
&\left|\esper\!\left[\E^{\I \xi \frac{A_1}{\sigma n^{1/3}}}\right]\,\E^{\frac{\xi^2}{2n^{2/3}}}-1\right| \\
&\leq \left|\esper\!\left[\E^{\I \xi \frac{A_1}{\sigma n^{1/3}}}\right] - 1+\frac{\xi^2}{2n^{2/3}}\right| \E^{\frac{\xi^2}{2n^{2/3}}} + \left|\E^{-\frac{\xi^2}{2n^{2/3}}}-1+\frac{\xi^2}{2n^{2/3}}\right| \E^{\frac{\xi^2}{2n^{2/3}}} \\
&\leq \left(\frac{|\xi|^3}{6Kn}+\frac{\xi^4}{8n^{4/3}}\right)\E^{\frac{\xi^2}{2n^{2/3}}} \leq \frac{7\E^{1/2}}{24}\,\frac{|\xi|^3}{Kn}.
\end{align*}
For the same reasons,
\begin{align*}
\left|\esper\!\left[\E^{\I \xi \frac{A_1}{\sigma n^{1/3}}}\right]\,\E^{\frac{\xi^2}{2n^{2/3}}}\right| & \leq \frac{|\xi|^3}{6Kn}\,\E^{\frac{\xi^2}{2n^{2/3}}} + \left(1-\frac{\xi^2}{2n^{2/3}}\right)\E^{\frac{\xi^2}{2n^{2/3}}} \\
& \leq \frac{|\xi|^3}{6Kn}\,\E^{1/2} + 1 \leq \E^{\frac{\E^{1/2}}{6}\, \frac{|\xi|^3}{Kn}}
\end{align*}
We conclude that
\begin{align*}
|\theta_n(\xi)-1| \leq \frac{7\E^{1/2}}{24}\,\frac{|\xi|^3}{K} \,\E^{\frac{\E^{1/2}}{6}\, \frac{|\xi|^3}{K}}
\end{align*}
on the zone of control $[-Kn^{1/3},Kn^{1/3}]$. If we want Condition \ref{hyp:secondcondition} to be satisfied, we need to change $K$ and set
$$K = \frac{3}{2\E^{1/2}}\,\frac{\sigma^3}{\rho},$$
which is a little bit smaller than before. We then have a zone of control with constants $K_1 = \frac{7\E^{1/2}\rho}{24\,\sigma^3}$ and $K_2 = \frac{\E^{1/2}\rho}{6\,\sigma^3}$, and the inequality $K \leq (\frac{c^\alpha}{2K_2})^{\frac{1}{w-\alpha}}$ is an equality. By Theorem \ref{thm:kolmogorov},
$$\dkol(Y_n,\mathcal{N}_\R) \leq \frac{1+\lambda}{\sqrt{2\pi}}\left(\frac{7}{24}\,2^{3/2}\E^{1/2} + \frac{2\E^{1/2}}{3\sqrt[3]{\pi} }\left(4\sqrt[3]{1+\frac{1}{\lambda}}+3\sqrt[3]{3}\right)\right) \frac{\rho}{\sigma^3\,\sqrt{n}}$$
with $Y_n = \frac{1}{\sigma \sqrt{n}}\,\sum_{i=1}^n A_i$. Taking $\lambda=0.183$, we obtain a bound with a constant $C\leq 4.815$, so we recover
$$\dkol(Y_n,\gauss) \leq 4.815\,\frac{\rho}{\sigma^3\sqrt{n}},$$
which is almost as good as the statement in the introduction, where a constant $C=3$ was given (the best constant known today is, as far as we know, $C=0.4748$, see \cite{KS10}). 
Of course, the advantage of our method is its large range of applications,
as we shall see in the next sections.\medskip

Our notion of zone of control allows one to deal with sums of random variables that are independent but not identically distributed. As an example, consider for $r <1$ a random series
$$Z_r=\sum_{k=1}^\infty \mathcal{B}(r^{2k}),$$
with Bernoulli variables of parameters $r^{2k}$ that are independent. The random variable  $Z_r$ has the same law as the number of zeroes with module smaller than $r$ of a random analytic series $S(z)=\sum_{n=0}^\infty a_n\,z^n$, where the $a_n$'s are independent standard complex Gaussian variables (see \cite[Section 7.1]{FMN16}). If $h = \frac{4\pi r^2}{1-r^2}$ is the hyperbolic area of the disc of radius $r$ and center $0$, then we showed in \emph{loc.~cit.} that as $h$ goes to infinity and $r$ goes to $1$, denoting $Z_r=Z^h$, the sequence 
$$X_h=\frac{1}{h^{1/3}}\left(Z^h - \frac{h}{4\pi}\right)$$
is mod-Gaussian convergent with parameters $t_h = \frac{h^{1/3}}{8\pi}$ and limit $\theta(\xi)=\exp(\frac{(\I\xi)^3}{144\pi})$. Let us compute a zone of control for this mod-Gaussian convergence. We change a bit the parameters of the mod-Gaussian convergence and take
$$\widetilde{t}_h = \var (X_h)= \frac{1}{h^{2/3}}\sum_{k=1}^\infty r^{2k}(1-r^{2k}) = \frac{h^{1/3}(h+4\pi)}{4\pi(2h+4\pi)}.$$
The precise reason for this small modification will be given in Remark \ref{rem:clever}. Then, 
$$\theta_h(\xi) = \esper[\E^{\I \xi X_h}]\,\E^{\frac{\widetilde{t}_h\xi^2}{2}} = \prod_{k=1}^\infty \left(1+r^{2k}(\E^{\frac{\I \xi}{h^{1/3}}}-1)\right)\,\E^{-\frac{r^{2k}\I\xi}{h^{1/3}}+\frac{r^{2k}(1-r^{2k})\xi^2}{2h^{2/3}}}.$$
Denote $\theta_{h,k}(\xi)$ the terms of the product on the right-hand side. For $|\xi| \leq \frac{ h^{1/3}}{4}$, we are going to compute bounds on $|\theta_{h,k}(\xi)|$ and $|\theta_{h,k}(\xi)-1|$. The holomorphic function
$$f_\alpha(z) = \log(1+\alpha(\E^z-1)) - \alpha z - \frac{\alpha(1-\alpha)\,z^2}{2}$$
has its two first derivatives at $0$ that vanish, and its third complex derivative is
$$f_\alpha'''(z) = \alpha(1-\alpha)\,\E^z\,\frac{(1-\alpha(1+\E^z))}{(1+\alpha(\E^z-1))^3}.$$
If $|\xi| \leq \frac{h^{1/3}}{4}$, then $|\E^{\frac{\I \xi}{h^{1/3}}}| \leq \E^{1/4}$ and $|\E^{\frac{\I \xi}{h^{1/3}}}-1|\leq \frac{1}{4}\,\E^{1/4} \leq \frac{1}{2} $, so
\begin{align*}
|\log \theta_{h,k}(\xi)| &\leq \frac{|\xi|^3}{6h}\,r^{2k}(1-r^{2k})\,\E^{1/4}\,\frac{1+\frac{r^{2k}}{2}}{(1-\frac{1}{4}\E^{1/4}\,r^{2k})^3} \\
&\leq \frac{|\xi|^3}{4h} \frac{\E^{1/4}\,r^{2k}}{(1-\frac{1}{4}\E^{1/4})^2}\leq \frac{|\xi|^3r^{2k}}{h}.
\end{align*}
Therefore, 
$|\theta_{h,k}(\xi)| \leq \exp(\frac{|\xi^3|r^{2k}}{h})$ and $|\theta_{h,k}(\xi)-1| \leq \frac{|\xi^3|r^{2k}}{h} \,\exp(\frac{|\xi^3|r^{2k}}{h})$.
We then obtain on the zone $|\xi| \leq \frac{h^{1/3}}{4}$
\begin{align*}
|\theta_h(\xi)-1| &\leq \sum_{k=1}^\infty |\theta_{h,k}(\xi)-1|\prod_{j \neq k} |\theta_{h,j}(\xi)| \leq S\,\exp S
\end{align*}
with $S = \sum_{k=1}^\infty \frac{|\xi^3|r^{2k}}{h} = \frac{|\xi^3|}{4\pi}$. The inequalities of Condition \ref{hyp:secondcondition} forces us to look at a slightly smaller zone $\xi \in [-\pi \widetilde{t}_h,\pi \widetilde{t}_h]$; then, this zone of control has index $(3,3)$ and constants $K_1=K_2 = \frac{1}{4\pi}$. We can then apply Theorem \ref{thm:kolmogorov}, and we obtain for $h$ large enough
$$\dkol\left(\frac{Z^h - \frac{h}{4\pi}}{\sqrt{\var(Z^h)}}\,,\,\gauss\right) \leq \frac{C}{\sqrt{h}}$$
with a constant $C \leq 166$.

\subsection{Winding number of a planar Brownian motion}
\label{Sect:Winding}
In this section, we consider a standard planar Brownian motion $(Z_t)_{t \in \R_+}$ starting from $z=1$. 
It is well known that, a.s., $Z_t$ never touches the origin.
One can thus write $Z_t=R_t\,\E^{\I\varphi_t}$, for continuous functions $t\mapsto R_t$ and $t\mapsto \varphi_t$,
where $\varphi_0=0$, see Figure \ref{fig:winding}.
\begin{figure}[ht]
\begin{center}
\includegraphics[scale=0.7]{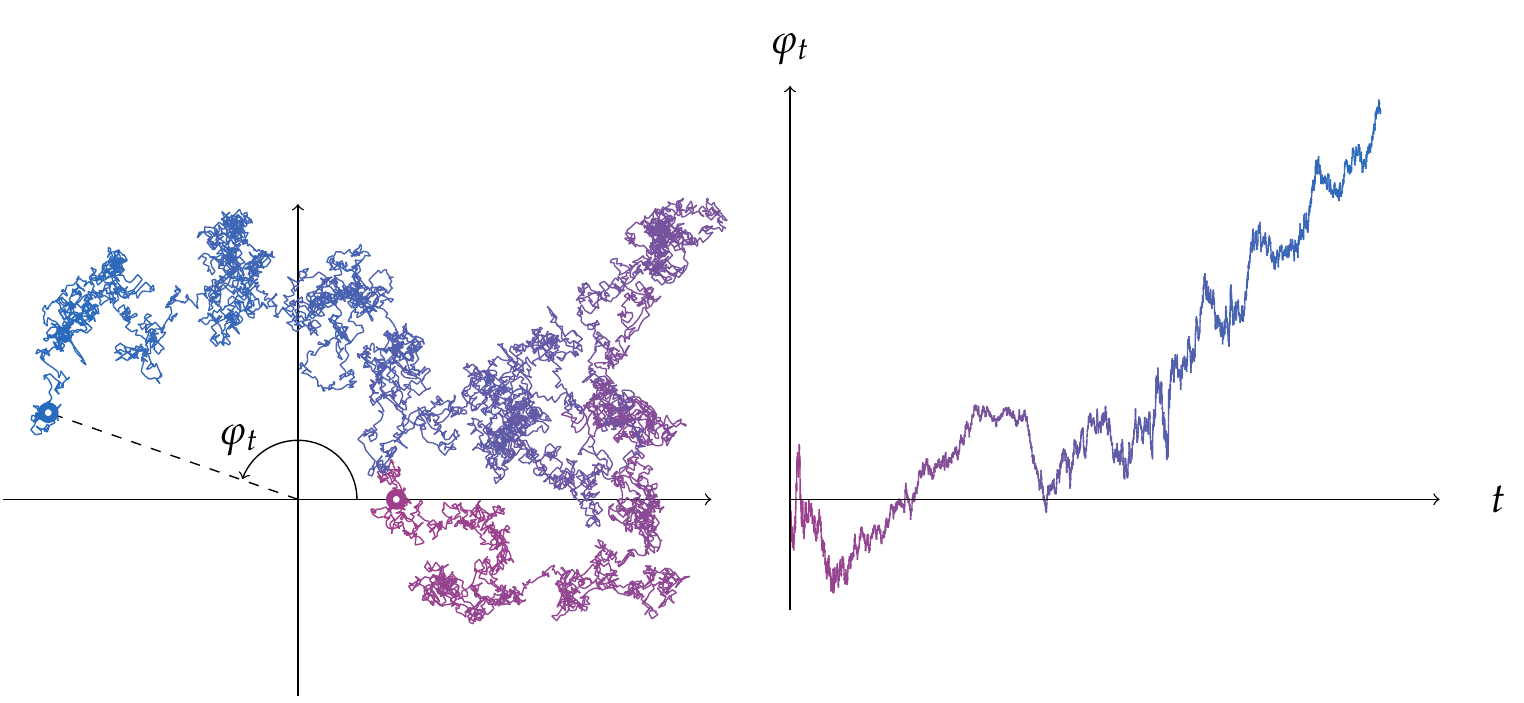}
\caption{Planar Brownian motion and its winding number; we will see that the latter is asymptotically mod-Cauchy.\label{fig:winding}}
\end{center}
\end{figure}

The Fourier transform of the \emph{winding number} $\varphi_t$ was computed by Spitzer in \cite{Spi58}:
$$\esper[\E^{\I\xi \varphi_t}]=\sqrt{\frac{\pi}{8t}}\,\,\E^{-\frac{1}{4t}}\,\left(I_{\frac{|\xi|-1}{2}}\left(\frac{1}{4t}\right)+I_{\frac{|\xi|+1}{2}}\left(\frac{1}{4t}\right)\right),$$
where $I_{\nu}(z)=\sum_{k=0}^\infty \frac{1}{k!\,\Gamma(\nu+k+1)}\,\left(\frac{z}{2}\right)^{\nu+2k}$ is the modified Bessel function of the first kind. As a consequence, and as was noticed in \cite[\S3.2]{DKN11}, $(\varphi_t)_{t \in \R_+}$ converges mod-Cauchy with parameters $\tfrac{\log 8t}{2}$ and limiting function
$\theta(\xi)=\sqrt{\pi} \ \Gamma\big(\tfrac{|\xi|+1}{2}\big)^{-1}$. Indeed,
\begin{align*}
&\esper[\E^{\I\xi \varphi_t}] \exp\left(|\xi|\,\frac{\log 8t}{2}\right) \\
&= \sqrt{\pi}\,\E^{-\frac{1}{4t}} \,
\left(\sum_{k=0}^\infty \frac{1}{k!}\left(\frac{1}{8t}\right)^{\!2k} \left(\frac{1}{\Gamma(k+\frac{|\xi|+1}{2})} + \frac{1}{8t\,\Gamma(k+\frac{|\xi|+3}{2})}\right)\right),
\end{align*}
and the limit of the power series as $t$ goes to infinity is its constant term $\frac{1}{\Gamma(\frac{|\xi|+1}{2})}$. \bigskip

Here, $|\theta(\xi)-1|$ is of order $O(|\xi|)$ around $0$, since the first derivative of $\Gamma$ is not zero at $\frac{1}{2}$. Therefore, if the mod-convergence can be given a zone of control, then the index of this control will be $v=1$, which forces for Berry--Esseen estimates $\gamma\leq 0$ since $\gamma\leq \min(\frac{v-1}{\alpha},\frac{1}{w-\alpha})$.  Conversely, for any $\xi$, notice that the function $x \mapsto \frac{1}{\Gamma(x+\frac{1}{2})}$ has derivative bounded on $\R_+$ by $$-\frac{\Gamma'(\frac{1}{2})}{\left(\Gamma(\frac{1}{2})\right)^2}=1.11_- < \frac{2}{\sqrt{\pi}},$$ and therefore that
\begin{align*}
|\theta_t(\xi)-1| &\leq 2\,\E^{-\frac{1}{4t}} \sum_{k=0}^\infty \frac{1}{k!}\left(\frac{1}{8t}\right)^{2k} \left(1+\frac{1}{8t}\right) \frac{|\xi|}{2} \\
&\leq \E^{\left(\frac{1}{8t}\right)^2-\frac{1}{4t}}\,\left(1+\frac{1}{8t}\right)\,|\xi| \leq |\xi|
\end{align*}
for $t$ large enough. So, in particular, one has mod-Cauchy convergence with index of control $(1,1)$, zone of control $[-K,K]$ with $K$ as large as wanted, and constants $K_1=1$ and $K_2=0$. It follows then from Theorem \ref{thm:kolmogorov} that if $\mathcal{C}$ follows a standard Cauchy law, then
$$\dkol\left(\frac{2\varphi_t}{\log 8t},\,\mathcal{C}\right)\leq \frac{4}{\log 8t}$$
for $t$ large enough. As far as we know, this estimate is new.

\subsection{Approximation of a stable law by compound Poisson laws}\label{subsec:compoundpoisson}
Let $\phi_{c,\alpha,\beta}$ be a stable law; the Lévy--Khintchine formula for its exponent allows one to write
$$\E^{\eta_{c,\alpha,\beta}(\I\xi)} = \begin{cases}
	\E^{-|c\xi|^2} &\text{if }\alpha=2,\\
	\exp\left(\I m\xi+\int_{\R}(\E^{\I \xi x}-1-1_{|x|<1}\I\xi x)\,\pi_{c,\alpha,\beta}(\!\DD{x})\right)& \text{if }\alpha\in(0,2),
\end{cases}$$
where $\pi_{c,\alpha,\beta}(\!\DD{x})$ is the Lévy measure defined for $\alpha \in (0,2)$ by
$$\pi_{c,\alpha,\beta}(\!\DD{x}) = \frac{c_+\,1_{x>0}}{x^{1+\alpha}} +\frac{c_-\,1_{x<0}}{|x|^{1+\alpha}} ,$$
with $m \in \R$ and $c_+,c_-\in \R_+$  related to $(c,\alpha,\beta)$ by $\beta = \frac{c_+-c_-}{c_++c_-} $ and
\begin{align*}
m&=\begin{cases}
	\frac{c_+-c_-}{1-\alpha} &\text{if }\alpha \neq 1,\\
	\left(\int_0^1 \frac{\sin t - t}{t^2}\,dt+\int_1^\infty \frac{\sin t}{t^2}\,dt\right)(c_+-c_-)&\text{if } \alpha = 1;
\end{cases} \\
c_++c_- &= \begin{cases}
	\frac{\alpha\,c^\alpha}{\Gamma(1-\alpha)\,\sin\left(\frac{\pi(1-\alpha)}{2}\right)}  &\text{if }\alpha \neq 1,\\
	\frac{2c}{\pi}&\text{if }\alpha=1.
\end{cases}
\end{align*}
The proof of the Lévy--Khintchine formula in the general case of an infinitely divisible law involves the following elementary fact (\emph{cf.} \cite[Chapter 2]{Sato99}): if $\mu$ is infinitely divisible and $\mu = (\rho_n)^{*n}$ for $n \geq 1$, then the compound Poisson law $\mu_n$ of intensity $n\,\rho_n$, which has Fourier transform
$$\widehat{\mu_n}(\xi) = \exp\left(\int_{\R} (\E^{\I x\xi}-1)\,n\rho_n(\!\DD{x})\right),$$
converges in law towards $\mu$; thus, any infinitely divisible law is a limit of compound Poisson laws. In the case of stable laws, this approximation result can be precised in terms of Kolmogorov distances:
\begin{proposition}
\label{prop:compoundpoisson}
Let $Y$ be a random variable with stable law $\phi_{c,\alpha,\beta}$, and $Y_n$ be a random variable with the following compound Poisson distribution: if $\mu_n$ is the law of $Y_n$, then its Fourier transform is
$$\widehat{\mu_n}(\xi) = \exp\left(\int_{\R} (\E^{\I x\xi}-1)\,n\,\phi_{\frac{c}{n^{1/\alpha}},\alpha,\beta}(\!\DD{x})\right).$$
The Kolmogorov distance between $Y_n$ and $Y$ is
$$\dkol(Y_n,Y) \leq \begin{cases}
	C(\alpha)\,n^{-1/\alpha} &\text{if }\alpha \in (1,2],\\
	C(\alpha)\,n^{-1}&\text{if }\alpha \in (0,1) \text{ or }\alpha=1,\beta=0,\\
    C'\,(\log n)^2\,n^{-1}&\text{if } \alpha=1,\beta\neq 0,
\end{cases}$$
with constants $C(\alpha)$ or $C'$ that depend only on the exponent $\alpha$.
\end{proposition}

\noindent We thus get a phase transition between the cases $\alpha>1$ and $\alpha<1$, with the case $\alpha=1$ that exhibit distinct transition behaviors according to the value of $\beta$.

\begin{proof}
Let us distinguish the following cases:
\begin{itemize}
    \item Suppose first $\alpha \notin \{1,2\}$. The definition of $Y_n$ implies that 
$$\esper[\E^{\I \xi Y_n}]=\widehat{\mu_n}(\xi) = \exp\left(n\!\left(\E^{\frac{\eta_{c,\alpha,\beta}(\I \xi)}{n}}-1\right)\right).$$
Set $X_n = n^{1/(2\alpha)}Y_n$, $t_n = \sqrt{n}$ and $\theta_n(\xi) = \esper[\E^{\I \xi X_n}]\,\E^{-t_n\,\eta_{c,\alpha,\xi}(\I \xi)}$. We have
\begin{align*}
\theta_n(\xi) &= \exp\left(n\!\left(\E^{\frac{\eta_{c,\alpha,\beta}(\I \xi)}{n^{1/2}}}-1-\frac{\eta_{c,\alpha,\beta}(\I \xi)}{n^{1/2}}\right)\right).
\end{align*}
On the zone $[-K(t_n)^{1/\alpha},K(t_n)^{1/\alpha}]$ with $K = \frac{|\cos(\frac{\pi \alpha}{2})|^{\frac{2}{\alpha}}}{c}$, we can use a Taylor formula with an integral form remainder:
\begin{align*}
n\!\left(\E^{\frac{\eta_{c,\alpha,\beta}(\I \xi)}{n^{1/2}}}\!-1-\frac{\eta_{c,\alpha,\beta}(\I \xi)}{n^{1/2}}\right) &= (\eta_{c,\alpha,\beta}(\I \xi))^2\left(\int_{0}^1\!(1-u)\,\E^{\frac{u\eta_{c,\alpha,\beta}(\I \xi)}{n^{1/2}}}\DD{u}\right) \\
\left|n\!\left(\E^{\frac{\eta_{c,\alpha,\beta}(\I \xi)}{n^{1/2}}}\!-1-\frac{\eta_{c,\alpha,\beta}(\I \xi)}{n^{1/2}}\right) \right| &\leq  \frac{1}{2}\,|\eta_{c,\alpha,\beta}(\I \xi)|^2 \leq \frac{1}{2}\left(\frac{c^\alpha}{\cos\left(\frac{\pi\alpha}{2}\right)}\right)^{\!2}\,|\xi|^{2\alpha}.
\end{align*}
We thus obtain a zone of control for $(X_n)_{n\in \N}$ with $v=w=2\alpha$, $\gamma = \frac{1}{\alpha}$,  
$$K_1=K_2 = \frac{1}{2}\left(\frac{c^\alpha}{\cos\left(\frac{\pi\alpha}{2}\right)}\right)^{\!2},$$ and one checks that
$$\left(\frac{c^\alpha}{2K_2}\right)^{\frac{1}{w-\alpha}} = \frac{|\cos(\frac{\pi \alpha}{2})|^{\frac{2}{\alpha}}}{c} = K.$$
Since we need $\gamma \leq \min\left(\frac{1}{w-\alpha},\frac{v-1}{\alpha}\right)$ to obtain a bound on the Kolmogorov distance (see the hypotheses of Theorem~\ref{thm:kolmogorov}), this leads to a reduction of $\gamma$ when $\alpha<1$:
$$\gamma+\frac{1}{\alpha} = \begin{cases}
    \frac{2}{\alpha} &\text{if }\alpha>1,\\
    2 &\text{if }\alpha<1.
\end{cases}$$
With Theorem \ref{thm:kolmogorov}, we obtain the following upper bound for $\dkol(Y_n,Y)$:
$$ \frac{1+\lambda}{\alpha\pi}\left(\frac{2}{\left(\cos(\frac{\pi\alpha}{2})\right)^2}+\frac{\Gamma(\frac{1}{\alpha})}{\sqrt[3]{\pi}\,\left|\cos(\frac{\pi\alpha}{2})\right|^{2/\alpha}}\left(4\sqrt[3]{1+\frac{1}{\lambda}}+3\sqrt[3]{3}\right)\right)\frac{1}{n^{\frac{\gamma}{2}+\frac{1}{2\alpha}}},$$
Then, any choice of $\lambda>0$ gives a constant $C(\alpha)$ that depends only on $\alpha$.

\item When $\alpha=2$, the result follows from the usual Berry--Esseen estimates, since $\sqrt{n}\,Y_n$ has the law of a sum of $n$ independent random variables with same law and finite variance and third moment.

\item If $\alpha=1$ and $\beta=0$, then the same computations as above can be performed with a constant $K= \frac{1}{c}$, $v=w=2$, $\gamma=1$,
$$K_1=K_2=\frac{c^2}{2},$$
and this leads to
$$\dkol(Y_n,Y) \leq \frac{1+\lambda}{\pi}\left(2+\frac{1}{\sqrt[3]{\pi}}\left(4\sqrt[3]{1+\frac{1}{\lambda}}+3\sqrt[3]{3}\right)\right)\frac{1}{n},$$
and a constant $C=3.04$ when $\lambda=0.2$. 

\item Let us finally treat the case $\alpha=1$, $\beta \neq 0$. Recall that we then have $\eta_{c,\alpha,\beta}(\I \xi) = -|c\xi|\,(1+\frac{2\I \beta}{\pi}\,\mathrm{sgn}(\xi)\,\log|\xi|)$. We choose $t_n$ such that $t_n \log t_n = \sqrt{n}$, and set 
$$X_n = t_n\,Y_n + \frac{2c\beta}{\pi}\,\sqrt{n}.$$
We then have 
\begin{align*}
\theta_n(\xi) &= \esper[\E^{\I\xi X_n}]\,\E^{-t_n\,\eta_{c,\alpha,\beta}(\I \xi)} \\
&= \exp\left(\frac{2c\beta \I \xi}{\pi} \,t_n \log t_n + n \!\left(\E^{\frac{\eta_{c,\alpha,\beta}(\I t_n  \xi)}{n}}-1\right) -t_n\,\eta_{c,\alpha,\beta}(\I \xi)\right) \\
&= \exp\left(n \!\left(\E^{\frac{\eta_{c,\alpha,\beta}(\I t_n  \xi)}{n}} -1 - \frac{\eta_{c,\alpha,\beta}(\I t_n  \xi)}{n}\right) \right),
\end{align*}
and the Taylor formula with integral remainder yields:
\begin{align*}
\left|n \!\left(\E^{\frac{\eta_{c,\alpha,\beta}(\I t_n  \xi)}{n}} -1 - \frac{\eta_{c,\alpha,\beta}(\I t_n  \xi)}{n}\right)\right| &\leq \frac{1}{2n}\,|\eta_{c,\alpha,\beta}(\I t_n  \xi)|^2 \\
&\leq \frac{c^2|\xi|^2}{2}\left(\frac{1+ \frac{4}{\pi^2}\,(\log |t_n\xi|)^2}{(\log t_n)^2}\right).
\end{align*}
On the zone $[-\frac{t_n}{2c} , \frac{t_n}{2c}]$, we thus have 
\begin{align*}
  &\left|n \!\left(\E^{\frac{\eta_{c,\alpha,\beta}(\I t_n  \xi)}{n}} -1 - \frac{\eta_{c,\alpha,\beta}(\I t_n  \xi)}{n}\right)\right| \\
  &\leq \frac{c^2|\xi|^2}{2}\left(\frac{1+ \frac{4}{\pi^2}\,(2 \log t_n -\log 2c)^2}{(\log t_n)^2}\right) \\
  &\leq c^2|\xi|^2\quad \text{for $t_n$ large enough.}
 \end{align*}
 So, there is a zone of control with constants $K_1=K_2 = c^2$, $v=w=2$ and $\gamma=1$, and $K= \frac{1}{2c}$. We thus get as before an estimate of $\dkol(Y_n,Y)$ of order $\O((t_n)^{-2})$, and since $(t_n \log t_n)^2 = n$, $(t_n)^2$ is asymptotically equivalent to $\frac{4n}{\log^2 n}$.\qed
\end{itemize}
\end{proof}

\subsection{Convergence of Ornstein--Uhlenbeck processes to stable laws}\label{subsec:ornstein}
Another way to approximate a stable law $\phi_{c,\alpha,\beta}$ is by using the marginales of a random process of Ornstein--Uhlenbeck type. Consider more generally a self-decomposable law $\phi$ on $\R$, that is an infinitely divisible distribution such that for any $b\in (0,1)$, there exists a probability measure $p_b$ on $\R$ such that
\begin{equation}
\widehat{\phi}(\xi) = \widehat{\phi}(b\xi)\,\widehat{p}_b(\xi);\label{eq:selfdecomposable}
\end{equation}
see \cite[Chapter 3, Definition 15.1]{Sato99}. In Equation \eqref{eq:selfdecomposable}, the laws $p_b$ are the marginale laws of certain Markov processes. Fix a Lévy--Khintchine triplet $(l\in \R,\,\, \nu^2 \in \R_+,\,\,\rho)$ with $\rho$ probability measure on $\R \setminus \{0\}$ that integrates $\min(1,|x|^2)$, and consider the Lévy process $(Z_t)_{t \in \R_+}$ associated to this triplet:
\begin{align*}
\esper[\E^{\I \xi Z_t}] &= \exp(t\psi(\I \xi)) \\
&=  \exp\left(t\left(\I l \xi - \frac{\nu^2\xi^2}{2}+\int_{\R} (\E^{\I \xi x} - 1 - 1_{|x|<1}\I \xi x)\,\rho(\!\DD{x})\right)\right).
\end{align*}
The Ornstein--Uhlenbeck process with triplet $(l,\nu^2,\rho)$, speed $v$ and starting point $x$ is the solution $(U_t)_{t \geq 0}$ of the stochastic differential equation
$$U_t = \E^{-vt}x + \int_0^t \E^{-v(t-s)}\,dZ_s.$$
The Ornstein--Uhlenbeck process $(U_t)_{t \geq 0}$ can be shown to be a Markov process whose transition kernel
$(P_t(x,\!\DD{y}))_{t\geq 0}$ satisfies:
$$\widehat{P_t(x,\cdot)}(\xi) =\int_{\R} \E^{\I \xi y}\,P_t(x,\!\DD{y}) = \exp\left(\I \xi \E^{-vt} x + \int_{0}^t \psi(\I \xi\E^{-vs})\DD{s} \right)$$
see \cite[Lemma 17.1]{Sato99}. The connection with self-decomposable laws is provided by:
\begin{theorem}[Sato--Yamazato, 1983]\label{thm:selfdecomposable}
For any self-decomposable law $\phi$ and any fixed speed $v$, there exists a unique Lévy--Khintchine triplet $(l,\nu^2,\rho)$ with $\int_{|x|\geq 1}\log |x|\,\rho(\!\DD{x})<+\infty$, such that the associated Ornstein--Uhlenbeck process $(U_t)_{t \geq 0}$ with speed $v$ satisfies:
$$\forall x \in \R,\,\,P_t(x,\cdot) \rightharpoonup \phi.$$
If $\psi(\I \xi)$ is the exponent associated to $(l,\nu^2,\rho)$, then 
$$\widehat{\phi}(\xi) = \exp\left(\int_{s=0}^\infty \psi(\I \xi\E^{-vs}) \DD{s}\right).$$
\end{theorem}
\noindent We refer to \cite{SY83} and \cite[Theorem 17.5]{Sato99}. In the setting of Theorem \ref{thm:selfdecomposable}, one has the relation 
$$\widehat{\phi}(\xi) = \widehat{\phi}(\E^{-vt}\xi)\,\left(\widehat{P_t(x,\cdot)}(\xi)\,\widehat{\delta_{-\E^{-vt}x}}(\xi)\right),$$
so if $b \in (0,1)$, setting $b=\E^{-vt}$, one recovers $p_b$ as the law of $U_t-\E^{-vt}x$, where $(U_t)_{t \in \R_+}$ is the Ornstein--Uhlenbeck process that converges in distribution to $\phi$ and that has speed $v$ and starting point $x$. \medskip

Suppose that $\phi=\phi_{c,\alpha,\beta}$ is a stable law. Then, the previous computations can be reinterpreted in the framework of mod-$\phi$ convergence. We set 
$$\theta(\xi) =  \frac{\widehat{\delta_{x}}(\xi)}{\widehat{\phi}(\xi)} = \exp(\I \xi x - \eta(\I \xi)),$$
and 
$$X_t = \begin{cases}
     \E^{vt} U_t &\text{ if }\alpha \neq 1,\\
     \E^{vt} \left(U_t - \frac{2c\beta vt}{\pi}\right) &\text{ if }\alpha = 1.
\end{cases}$$ 
Then,
\begin{align*}
\esper[\E^{\I \xi X_t}] &= \begin{cases}
    \widehat{\mu}(\E^{vt} \xi) \,\theta(\xi)&\text{ if }\alpha \neq 1,\\
    \widehat{\mu}(\E^{vt} \xi)\,\E^{-\I \xi \E^{vt} \frac{2c\beta vt}{\pi}} \,\theta(\xi)&\text{ if }\alpha = 1,
\end{cases}\\
&= \E^{\E^{\alpha vt}\eta(\I \xi)} \,\theta(\xi),
\end{align*}
so $(X_t)_{t \geq 0}$ converges mod-$\phi$ with parameters $\E^{\alpha vt}$, and with limit equal to the residue $\theta(\xi)$. Note that $\theta_t = \theta$ for any $t \geq 0$, so we are in a special situation where the residues are constant (time-independent). Assuming that $x \neq 0$, one has for any $\xi \in \R$
$$|\theta(\xi)-1| \leq \begin{cases}
K_1\, |\xi|\,\exp(K_2\,|\xi|^\alpha) &\text{if }\alpha \in (1,2],\\
K_1\, |\xi|^{\alpha}\,\exp(K_2\,|\xi|^\alpha) & \text{if }\alpha \in (0,1) \text{ or }\alpha=1,\beta=0,\\
K_1\, |\xi|\log |\xi| \,\exp(K_2\,|\xi|) &\text{if }\alpha=1,\beta \neq 0.
\end{cases} $$
For the two first cases, the condition $\gamma \leq \min(\frac{1}{w-\alpha},\frac{v-1}{\alpha})$ in Theorem \ref{thm:kolmogorov} imposes the following choices of $\gamma$ when computing Berry--Esseen estimates: $\gamma = \frac{\alpha-1}{\alpha}$ when $\alpha\leq 1$, and $\gamma = 0$ when $\alpha \geq 1$. In these cases, one obtains:
$$\dkol(U_t, \phi_{c,\alpha,\beta}) = \begin{cases}
    O(\E^{-vt}) &\text{ if } \alpha \in (1,2],\\
    O(\E^{-\alpha vt}) &\text{ if }\alpha \in (0,1) \text{ or }\alpha=1,\beta=0.
\end{cases}$$
Because of the term $\log |\xi|$, the last case does not exactly fit the framework of zones of control, but it is easy to adapt the proofs and one gets an estimate $O(vt\,\E^{-vt})$. On the other hand, when $x=0$, the only difference with the previous discussion is the case $\alpha \in (1,2]$, where we obtain
$$|\theta(\xi)-1| \leq K_1\, |\xi|^\alpha\,\exp(K_2\,|\xi|^\alpha) $$
and by Theorem \ref{thm:kolmogorov}, $\dkol(U_t,\phi_{c,\alpha,\beta}) = O(\E^{-\alpha vt})$, choosing $\gamma = \frac{\alpha-1}{\alpha}$. So, to summarise:
\begin{proposition}\label{prop:ornstein}
Let $Y$ be a random variable with stable law $\phi_{c,\alpha,\beta}$, and $(U_t)_{t \geq 0}$ be the corresponding Ornstein--Uhlenbeck process with starting point $x$ and speed $v$. We have:
$$\dkol(U_t,Y) = \begin{cases}
     O(\E^{-vt})\!&\text{if }\alpha \in (1,2],x \neq 0,\\
    O(\E^{-\alpha vt})\!&\text{if }\alpha \in (0,1) \text{ or }\alpha =1,\beta =0 \text{ or }\alpha \in (1,2],x=0,\\
    O(vt\,\E^{-vt})\! &\text{if }\alpha=1,\beta \neq 0,
\end{cases}$$
with constants in the $O(\cdot)$ depending only on $x$ and $\alpha$.
\end{proposition}

\subsection{Logarithms of characteristic polynomials of random matrices}
In \cite[Sections 3 and 4]{KN12} and \cite[Section 7.5]{FMN16}, the mod-Gaussian convergence of the following random variables was proven:
\begin{table}
\begin{center}
\begin{tabular}{|c|c|c|c|}
\hline random matrix $M_n$ & random variable $X_n$ & parameters $t_n$  & residue $\theta(\xi)$ \\
\hline\hline
$\mathrm{Haar}(\mathrm{U}(n))$ & $\mathrm{Re}(\log \det(I_n-M_n))$ & $\frac{\log n}{2}$ & $ \frac{(G(1+\frac{\I \xi}{2}))^2}{G(1+\I \xi)}$\\
$\mathrm{Haar}(\mathrm{USp}(n))$ & $\log \det(I_{2n}-M_n) - \frac{1}{2} \log \frac{\pi n}{2}$ & $\log \frac{n}{2}$   &
$\frac{G(\frac{3}{2})}{G(\frac{3}{2}+\I\xi)}$\\
$\mathrm{Haar}(\mathrm{SO}(2n))$ & $\log \det(I_{2n}-M_{n}) - \frac{1}{2} \log \frac{8\pi}{n}$ & $\log \frac{n}{2}$   & $\frac{G(\frac{1}{2})}{G(\frac{1}{2}+\I\xi)}$\\
\hline
\end{tabular}
\end{center}
\caption{Mod-Gaussian convergence of the characteristic polynomials of Haar-distributed random matrices in compact Lie groups.}
\end{table}

\noindent Here, $G$ is Barnes' function, which is the unique entire solution of the equations $G(1)=1$ and $G(z+1)=G(z)\,\Gamma(z)$. Moreover, the mod-Gaussian convergence holds in fact on an half-plane $H =\{z \in \C\,|\,\mathrm{Re}(z)>-\alpha\}$. In the sequel, we denote $X_n^{\mathrm{A}}$, $X_n^{\mathrm{C}}$ and $X_n^{\mathrm{D}}$ the mod-Gaussian convergent random variables, according to the type of the classical group ($\mathrm{A}$ for unitary groups, $\mathrm{C}$ for compact symplectic groups and $\mathrm{D}$ for even orthogonal groups). Before computing zones of control for these variables, let us make the following essential remark:
\begin{remark}\label{rem:clever}
Let $(X_n)_{n \in \N}$ be a sequence of random variables that is mod-Gaussian convergent on a domain $D \subset \C$ which contains a neighborhood of $0$ (this ensures that $\theta_n$ and all its derivatives converge towards those of $\theta$). We denote $(t_n)_{n\in \N}$ the parameters of mod-Gaussian convergence of $(X_n)_{n \in \N}$. Then, without generality, one can assume $\theta_n'(0)=\theta_n''(0) = 0$ and $\theta'(0)=\theta''(0)=0$. Indeed, set 
$$\widetilde{X}_n = X_n + \I \theta_n'(0) \qquad;\qquad \widetilde{t}_n = t_n - \theta_n''(0).$$
We then have
\begin{align*}
\widetilde{\theta}_n(\xi):=\esper[\E^{\I \xi \widetilde{X}_n}]\, \E^{\widetilde{t}_n\frac{\xi^2}{2}} = \theta_n(\xi)\,\E^{-\theta_n'(0) \xi - \theta_n''(0) \frac{\xi^2}{2}}
\end{align*}
and this new residue satisfies $\widetilde{\theta}_n'(0) = \widetilde{\theta}_n''(0) = 0$. For the construction of zones of control, it allows us to force $v=3$, up to a translation of $X_n$ and of the parameter $t_n$. 
\end{remark}

In the following, we only treat the case of unitary groups, the two other cases being totally similar (one could also look at the imaginary part of the log-characteristic polynomial). There is an exact formula for the Fourier transform of $X_n^\mathrm{A}$ \cite[Formula (71)]{KS2000}:
$$\esper[\E^{\I \xi X_n^{\mathrm{A}}}] = \prod_{k=1}^n \frac{\Gamma(k)\Gamma(k+\I \xi)}{\left(\Gamma(k+\frac{\I\xi}{2})\right)^2}.$$
We have $\esper[X_n^{\mathrm{A}}]=0$, and 
$$\widetilde{t}_n=\esper[(X_n^{\mathrm{A}})^2]=\frac{1}{2}\sum_{k=1}^n \frac{\Gamma''(k)}{\Gamma(k)} - \left(\frac{\Gamma'(k)}{\Gamma(k)}\right)^2 = \frac{1}{2}\sum_{k=1}^n \psi_1(k),$$
where $\psi_1(z)$ is the trigamma function $\frac{d^2}{dz^2}(\log\Gamma(z))$, and is given on integers by the remainder of the series $\zeta(2)$:
$$\psi_1(k) = \sum_{m=k}^\infty \frac{1}{m^2}.$$
Therefore, $\widetilde{t}_n = \frac{1}{2}\sum_{m=1}^\infty \frac{\min(n,m)}{m^2} = \frac{1}{2}(\log n + \gamma + 1 + O(n^{-1}))$. So, $(X_n^\mathrm{A})_{n \in \N}$ is mod-Gaussian convergent with parameters $(\widetilde{t}_n)_{n \in \N}$ and limit
$$\widetilde{\theta}(\xi) = \frac{\left(G(1+\frac{\I \xi}{2})\right)^2}{G(1+\I \xi)}\,\E^{\frac{(\gamma+1)\xi^2}{4}},$$
which satisfies $\widetilde{\theta}'(0)=\widetilde{\theta}''(0)=0$. With these conventions, we can write the residues $\widetilde{\theta}_n(\xi)$ as
$$\widetilde{\theta}_n(\xi)= \left(\prod_{k=1}^n \frac{\Gamma(k)\Gamma(k+\I \xi)}{\left(\Gamma(k+\frac{\I\xi}{2})\right)^2}\right)\,\E^{\frac{\widetilde{t}_n\xi^2}{2}} = \prod_{k=1}^n \left(\frac{\Gamma(k)\Gamma(k+\I \xi)}{\left(\Gamma(k+\frac{\I\xi}{2})\right)^2}\,\E^{\frac{\psi_1(k)\xi^2}{4}}\right).$$
Denote $\vartheta_k(\xi)$ the terms of the product on the right-hand side; we use a similar strategy as in Section \ref{subsec:independent} for computing a zone of control. The function $\varphi_k(\xi) = \log \vartheta_k(\xi)$ vanishes at $0$, has its two first derivatives that also vanish at $0$, and therefore writes as
$$\varphi_k(\xi) = \left(\int_0^1 \varphi_k'''(t\xi)\, (1-t)^2\DD{t}\right)\frac{\xi^3}{2}.$$
The third derivative of $\varphi_k(\xi)$ is given by
$$\varphi_k'''(\xi) = -\I \,\psi_2(k+\I\xi) + \frac{\I}{2}\,\psi_2\left(k+\frac{\I\xi}{2}\right),$$
with $\psi_2(z) = -2\sum_{j=0}^\infty \frac{1}{(j+z)^3}$. As a consequence, $\varphi_k'''(\xi)$ is \emph{uniformly bounded} on $\R$ by
$$3\,\sum_{j=0}^\infty \frac{1}{(j+k)^3} \leq \frac{3\,\zeta(3)}{k^2}.$$
Therefore, 
\begin{align*}
|\varphi_k(\xi)| &\leq \frac{\zeta(3)\, |\xi|^3}{2k^2};\\
|\vartheta_k(\xi)| &\leq \E^{\frac{\zeta(3)\, |\xi|^3}{2k^2}};\\
|\vartheta_k(\xi)-1| &\leq \frac{\zeta(3)\, |\xi|^3}{2k^2}\,\E^{\frac{\zeta(3)\, |\xi|^3}{2k^2}}.
\end{align*}
It follows that for any $n$ and any $\xi \in \R$,
$|\widetilde{\theta}_n(\xi)-1| \leq S\exp S$
with $S = \sum_{k=1}^\infty \frac{\zeta(3)\, |\xi|^3}{2k^2} = \frac{3\,\zeta(3)\,|\xi|^3}{\pi^2}$. Set $K_1=K_2 = \frac{3\,\zeta(3)}{\pi^2}$, and $K=\frac{1}{4K_2} = \frac{\pi^2}{12\,\zeta(3)}$. We have a zone of control $[-Kt_n,Kt_n]$ of index $(3,3)$, with constants $K_1$ and $K_2$. We conclude with Theorem \ref{thm:kolmogorov}:
\begin{proposition}
Let $M_n$ be a random unitary matrix taken according to the Haar measure. For $n$ large enough,
$$\dkol\left(\frac{\mathrm{Re}(\log \det(I_n-M_n))}{\sqrt{\mathrm{Var}(\mathrm{Re}(\log \det(I_n-M_n)))}}\,,\,\gauss\right) \leq \frac{C}{(\log n)^{3/2}}$$
with a constant $C \leq 18$. Up to a change of the constant, the same result holds if one replaces $\mathrm{Re}(\log \det(I_n-M_n))$ by $\mathrm{Im}(\log \det(I_n-M_n))$, or by
$$\log \det(I_{2n}-P_n) - \esper[\det(I_{2n}-P_n)],$$
with $P_n$ Haar distributed in the unitary compact symplectic group $\mathrm{USp}(n)$ or in the even special orthogonal group $\mathrm{SO}(2n)$.
\end{proposition}

\section{Cumulants and dependency graphs}\label{sec:dependency}

\subsection{Cumulants, zone of control and Kolmogorov bound}
In this section, we will see that appropriate bounds on the cumulants
of a sequence of random variables $(S_n)_{n \in \N}$
imply the existence of a large zone of control
for a renormalized version of $S_n$.
In this whole section and in the next one, the reference stable law is the {\em standard Gaussian law.}
We also assume that the random variables $S_n$ are centered.
\medskip

Let us first recall the definition of cumulants.
If $X$ is a real-valued random variable with exponential generating function
$$\esper[\E^{zX}] = \sum_{r=0}^\infty \frac{\esper[X^r]}{r!}\,z^r $$
convergent in a neighborhood of $0$, then its \emph{cumulants} $\kappa^{(1)}(X),\kappa^{(2)}(X),\ldots$ are the coefficients of the series
$$ \log \esper[\E^{zX}] = \sum_{r=1}^\infty \frac{\kappa^{(r)}(X)}{r!}\,z^r,$$
which is also well defined for $z$ in a neighborhood of $0$ (see for instance \cite{LS59}). For example, $\kappa^{(1)}(X)=\esper[X]$, $\kappa^{(2)}(X)=\esper[X^2]-\esper[X]^2=\var(X)$, and 
$$\kappa^{(3)}(X)=\esper[X^3]-3\,\esper[X^2]\,\esper[X] + 2\,\esper[X]^3 .$$
We are interested in the case where cumulants can be bounded in an appropriate way.
\begin{definition}\label{def:virtual}
Let $(S_n)_{n \in \N}$ be a sequence of (centered) real-valued random variables.
We say that $(S_n)_{n \in \N}$ admits \emph{uniform bounds on cumulants} with parameters $(D_n,N_n,A)$ if,
for any $r\geq 2$, we have
$$|\kappa^{(r)}(S_n)| \leq N_n\,r^{r-2}\,(2D_n)^{r-1}\,A^r.$$
\end{definition}
\noindent 
In the following, it will be convenient to set
$(\widetilde{\sigma}_n)^2 = \frac{\var(S_n)}{N_n D_n}$. The inequality of Definition \ref{def:virtual} with $r = 2$ gives $(\widetilde{\sigma}_n)^2 \leq 2A^2$.

\begin{lemma}
  Let $(S_n)_{n \in \N}$ be a sequence with uniform bounds on cumulants with parameters $(D_n,N_n,A)$.
  Set \[X_n = \frac{S_n}{( N_n)^{1/3}\,(D_n)^{2/3}} 
  \ \text{ and }\ t_n=(\widetilde{\sigma}_n)^2 (\tfrac{N_n}{D_n})^{1/3}=\var(X_n).\]
  Then, we have for $(X_n)_{n \in \N}$ a zone of control $ [-K\,t_n\,,\,K\,t_n]$ of index $(3,3)$, 
  with the following constants:
  \[K=\frac{1}{(8+4\E)\,A^3}, \qquad K_1 = K_2 = (2+\E)\,A^3.\]
  \label{lem:Cumulants_Zone}
\end{lemma}

\begin{proof}
  From the definition of cumulants, since $X_n$ is centered and has variance $t_n$, we can write
\begin{align*}
 \theta_n(\xi)&=\esper[\E^{\I \xi X_n}] \,\exp\left(\frac{t_n\,\xi^2}{2}\right) \\
 &= \exp\left(\sum_{r= 3}^\infty \frac{\kappa^{(r)}(S_n)}{r!}\,\frac{(\I\xi)^r}{(N_n(D_n)^2)^{r/3}}\right) = \exp(z),
 \end{align*}
with
\[
|z| \leq \frac{1}{2} \frac{N_n}{D_n}\,\sum_{r=3}^\infty \frac{r^{r-2}}{\E^r\,r!} \,\left(\left(\frac{D_n}{N_n}\right)^{\!1/3}\,2\E A\, |\xi|\right)^{\!r}.
\]
We set $y=(\frac{D_n}{N_n})^{1/3}\,2\E A\, |\xi|$ and suppose that $y \leq 1$, that is to say that $\xi \in [-L\,(\frac{N_n}{D_n})^{1/3}\,,\,L\,(\frac{N_n}{D_n})^{1/3}]$ with $L=\frac{1}{2\E A}$. \medskip

By Stirling's bound, the series $S(y)=\sum_{r=3}^\infty \frac{r^{r-2}}{r!\,\E^r}\,y^r$ is convergent for any $y \in [0,1]$, and we have the inequality $S(y) \le \frac{y^3}{2\E^3\,(1-y)}$, which implies
\[
|z| \leq 2\,\frac{(A\,|\xi|)^3}{1-\left(\frac{D_n}{N_n}\right)^{\!1/3}\,2\E A\, |\xi|}.
\]
We now consider the zone of control $[-Kt_n,Kt_n]$ with $K = \frac{1}{(4\E+8) A^3}$. If $\xi$ is in this zone, then we have indeed
$$|\xi| \leq \frac{(\widetilde{\sigma}_n)^2}{4\E A^3} \left(\frac{N_n}{D_n}\right)^{\!1/3} \leq \frac{1}{2\E A}\,\left(\frac{N_n}{D_n}\right)^{\!1/3} = L\left(\frac{N_n}{D_n}\right)^{\!1/3}$$
by using the remark just before the lemma. Then, 
$$|z| \leq \frac{2A^3}{1-\frac{2\E A\,(\widetilde{\sigma}_n)^2}{(4\E+8)A^3}}\,|\xi|^3\leq \frac{2A^3}{1-\frac{\E}{\E+2}}\,|\xi|^3  =(2+\E)A^3\,|\xi|^3.$$
Thus, on the zone of control, $|\theta_n(\xi)-1| = |\E^{z}-1| \leq |z|\,\E^{|z|}$, whence a control of index $(3,3)$ and with constants
\begin{equation*}
K_1 = K_2 = (2+\E)A^3.
\end{equation*}
We have chosen $K$ so that $K=\frac{1}{4K_2}$, hence, the inequalities of Condition \ref{hyp:secondcondition} are satisfied.\qed
\end{proof}

Using the results of Section~\ref{sec:testfunctions}, we obtain:
\begin{corollary}
  \label{cor:Cumulants_Kol}
  Let $S_n$ be a sequence with a uniform bounds on cumulants with parameters $(D_n,N_n,A)$
  and let $Y_n=\tfrac{S_n}{\sqrt{\var(S_n)}}$. Then we have
	$$\dkol(Y_n\,,\,\gauss) \leq \frac{76.36\,A^3}{(\widetilde{\sigma}_n)^3}\,\sqrt{\frac{D_n}{N_n}}.$$
\end{corollary}

\begin{proof}
We can apply Theorem~\ref{thm:kolmogorov}, choosing $\gamma=1$, and $\lambda=0.193$. It yields a constant smaller than $77.911$.
It is however possible to get the better constant given above by redoing some of the computations in this specific setting. With $\rho>4$, set $K=\frac{1}{(4\E+\rho)\,A^3}$, and $\eps_n = \frac{1}{K\,(t_n)^{3/2}}$. On the zone $\xi \in [-\frac{1}{\eps_n},\frac{1}{\eps_n}]$, we have a bound $|\theta_n(\xi)-1|\leq M|\xi|^3\exp(M|\xi|^3)$, this time with $M = (2 + \frac{8\E}{\rho})A^3$. Hence, 
\begin{align*}
|\esper[f_{a,\eps_n}(Y_n)] - \esper[f_{a,\eps_n}(Y)]| &\leq \frac{M}{2\pi\,(t_n)^{3/2}}\int_{-\frac{1}{\eps_n}}^{\frac{1}{\eps_n}}  |\xi|^{2}\,\E^{-\frac{\xi^2}{2}+ M \frac{|\xi|^3}{(t_n)^{3/2}}}\DD{\xi} \\
&\leq \frac{M}{2\pi\,(t_n)^{3/2}}\int_{\frac{1}{\eps_n}}^{\frac{1}{\eps_n}}  |\xi|^{2}\,\E^{-\xi^2\left(\frac{1}{2}-\frac{2}{\rho}\right)}\DD{\xi} 
\\ 
&\leq \frac{M}{\sqrt{2\pi}\,(t_n)^{3/2}\,(1-\frac{4}{\rho})^{3/2}}  =\frac{1}{\sqrt{2\pi}}\,\frac{2}{\rho(1-\frac{4}{\rho})^{3/2}}\,\eps_n.
\end{align*}
By Theorem \ref{theorem:tao}, we get for any $\lambda > 0$:
\begin{align*}
&\dkol\left(\frac{S_n}{\sqrt{\var(S_n)}}\,,\,\gauss\right) \\
&\leq \frac{(1+\lambda)}{\sqrt{2\pi}}\,\left(\frac{2}{\rho(1-\frac{4}{\rho})^{3/2}}+\frac{1}{\sqrt[3]{\pi}}\left(4\sqrt[3]{1+\frac{1}{\lambda}}+3\sqrt[3]{3}\right)\right)\,\eps_n.
\end{align*}
The best constant is then obtained with $\rho=6.79$ and $\lambda=0.185$.\qed
\end{proof}

\begin{remark}
There is a trade-off in the bound of Corollary \ref{cor:Cumulants_Kol} between the \emph{a priori} upper bound on $(\widetilde{\sigma}_n)^2$, and the constant $C$ that one obtains such that the Kolmogorov distance is smaller than $$\frac{CA^3}{(\widetilde{\sigma}_n)^3}\,\sqrt{\frac{D_n}{N_n}}.$$ 
This bound gets worse when $(\widetilde{\sigma}_n)^2$ is small, but on the other hand, the knowledge of a better \emph{a priori} upper bound (that is precisely when $(\widetilde{\sigma}_n)^2$ is small) yields a better constant $C$. So, for instance, if one knowns that $(\widetilde{\sigma}_n)^2\leq A^2$ (instead of $2A^2$), then one gets a constant $C=52.52$. A general bound that one can state and that takes into account this trade-off is:
$$\dkol(Y_n\,,\,\gauss) \leq 27.55 \left(\left(\frac{A}{\widetilde{\sigma}_n}\right)^3 + \frac{A}{\widetilde{\sigma}_n}\right)\,\sqrt{\frac{D_n}{N_n}}.$$
We are indebted to Martina Dal Borgo for having pointed out this phenomenon. In the sequel, we shall freely use this small improvement of Corollary \ref{cor:Cumulants_Kol}.
\end{remark}
\medskip

The above corollary ensures asymptotic normality with a bound on the speed of convergence
when $$\left(\frac{1}{(\widetilde{\sigma}_n)^3}\,\sqrt{\tfrac{D_n}{N_n}} \to 0 \right) \iff \left( (\widetilde{\sigma}_n)^2\,\left(\frac{N_n}{D_n}\right)^{\!1/3} \to +\infty \right).$$
Using a theorem of Janson, 
the asymptotic normality can be obtained under a less restrictive hypothesis,
but without bound on the speed of convergence.
Even if the main topic of the paper is to find bounds on the speed of convergence,
we will recall the result here for the sake of completeness.
\begin{proposition}
  As above, let $S_n$ be a sequence with a uniform bounds on cumulants with parameters $(D_n,N_n,A)$
	and assume that
    $$\lim_{n \to \infty}\, \frac{\var(S_n)}{N_n\, D_n}\,\left(\frac{N_n}{D_n}\right)^{\!\eps} = \lim_{n \to \infty}(\widetilde{\sigma}_n)^2\,\left(\frac{N_n}{D_n}\right)^{\!\eps}= +\infty $$
	for some parameter $\eps \in (0,1)$. Then, 
	$$\frac{S_n}{\sqrt{\var(S_n)}} \rightharpoonup \gauss.$$
  \label{thm:CLT_Janson}
\end{proposition}

\begin{proof}
Note that the bounds on cumulants can be rewritten as
\begin{align*}
|\kappa^{(r)}(Y_n)| &\leq C_r\left(\frac{\var(S_n)}{N_n\,D_n}\,\left(\frac{N_n}{D_n}\right)^{1-\frac{2}{r}}\right)^{\!\!-\frac{r}{2}}
\end{align*}
for some constant $C_r$. Choosing $r$ large enough so that $1-\frac{2}{r} \geq \eps$, we conclude that $\kappa^{(r)}(Y_n) \to 0 $ for $r$ large enough. This is a sufficient condition for the convergence in law towards a Gaussian distribution, see \cite[Theorem 1]{Jan88} and \cite{Gri92}.\qed
\end{proof}
\begin{remark}
  Up to a change of parameters, it would be equivalent to consider bounds of the kind
  \[|\kappa^{(r)}(S_n)| \leq (Cr)^r \alpha_n (\beta_n)^r,\,\,
  \text{ or } \left|\kappa^{(r)}\!\left(\tfrac{S_n}{\sqrt{\Var(S_n)}}\right) \right| \le \frac{r!}{\Delta_n^{r-2}},\]
  as done by Saulis and Statulevi\v{c}ius in \cite{LivreOrange:Cumulants}
  or by the authors of this paper in \cite{FMN16}.
  In particular, it was proved in \cite[Chapter 5]{FMN16}, that under slight additional assumptions
  on the second and third cumulants, we have the following:
  the sequence $(X_n)_{n \in \N}$ defined in Lemma \ref{lem:Cumulants_Zone}
  converges in the mod-Gaussian sense with parameters $(t_n)_{n \in \N}$.
\end{remark}

\subsection{Dependency graphs}
\label{subsec:dependency}
In this paragraph, we will see that the uniform bounds on cumulants are satisfied 
for sums $S_n = \sum_{i=1}^n A_{i,n}$ of dependent random variables
with specific dependency structure.\medskip

More precisely, if $(A_v)_{v \in V}$ is a family of real valued random variables, we call \emph{dependency graph} for this family a graph $G=(V,E)$ with the following property:
if $V_1$ and $V_2$ are two disjoint subsets of $V$ with no edge $e \in E$ joining a vertex $v_1 \in V_1$ to a vertex $v_2 \in V_2$, then $(A_v)_{v \in V_1}$ and $(A_v)_{v \in V_2}$ are independent random vectors. For instance, let $(A_1,\ldots,A_7)$ be a family of random variables with  dependency graph drawn on Figure \ref{fig:dependency}.
Then the vectors $(A_1,A_2,A_3,A_4,A_5)$ and $(A_6,A_7)$ corresponding to different connected components must be independent.
Moreover, note that
 $(A_1,A_2)$ and $(A_4,A_5)$ must be independent as well:
 although they are in the same connected component of the graph $G$,
 they are not directly connected by an edge $e \in E$.
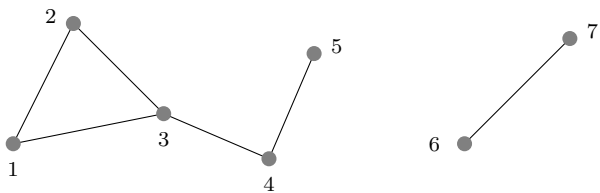
\begin{figure}[ht]
\begin{center}
\begin{tikzpicture}[scale=2]
\draw (1,0.2) -- (0.4,0.8) -- (0,0) -- (1,0.2) -- (1.7,-0.1) -- (2,0.6);
\draw (3,0) -- (3.7,0.7);
\fill[color=white!50!black] (0,0) circle [radius=0.5mm];
\fill[color=white!50!black] (1,0.2) circle [radius=0.5mm];
\fill[color=white!50!black] (0.4,0.8) circle [radius=0.5mm];
\fill[color=white!50!black] (1.7,-0.1) circle [radius=0.5mm];
\fill[color=white!50!black] (2,0.6) circle [radius=0.5mm];
\fill[color=white!50!black] (3,0) circle [radius=0.5mm];
\fill[color=white!50!black] (3.7,0.7) circle [radius=0.5mm];
\draw (0,-0.17) node {$1$};
\draw (0.25,0.85) node {$2$};
\draw (1,0.03) node {$3$};
\draw (1.7,-0.27) node {$4$};
\draw (2.15,0.65) node {$5$};
\draw (2.8,0) node {$6$};
\draw (3.85,0.75) node {$7$};
\end{tikzpicture}
\end{center}

\caption{A dependency graph for $7$ real-valued random variables.\label{fig:dependency}}
\end{figure}

\begin{theorem}[F\'eray--M\'eliot--Nikeghbali, see \cite{FMN16}]\label{thm:boundcumulant}
Let $(A_v)_{v \in V}$ be a family of random variables, with $|A_v| \leq A$ a.s., for all $v \in V$. 
We suppose that $G=(V,E)$ is a dependency graph for the family and denote
\begin{itemize}
  \item $N=\tfrac{\sum_{v\in V} \esper|A_v|}{A} \le \card\, V$;
  \item $D$ the maximum degree of a vertex in $G$ \emph{plus one}.
\end{itemize}
    If $S=\sum_{v \in V} A_v$, 
then for all $r \geq 1$,
$$ |\kappa^{(r)}(S)| \leq N\,r^{r-2}\,(2D)^{r-1}\,A^r.$$
\end{theorem}

Consider a sequence $(S_n)_{n \in \N}$, where each $S_n$ writes as $\sum_{v \in V_n} A_{v,n}$,
with the $A_{v,n}$ uniformly bounded by $A$ (in a lot of examples, the $A_{v,n}$ are indicator variables,
so that we can take $A=1$). Set $$N_n=\tfrac{\sum_{v\in V_n} \esper|A_{v,n}|}{A}$$ 
and assume that, for each $n$, $(A_{v,n})_{v \in V_n}$ has a dependency graph of maximal degree $D_n-1$.
Then the sequence $(S_n)_{n \in \N}$ admits uniform bounds on cumulants with parameters $(D_n,N_n,A)$
and the result of the previous section applies. Note that in this setting we have the bound $\widetilde{\sigma}_n\leq A$, so the bound of Corollary \ref{cor:Cumulants_Kol} holds with the better constant 52.52.

\begin{remark}
The parameter $D$ is equal to the maximal number of neighbors of a vertex $v \in V$, plus $1$. In the following, we shall simply call $D$ the maximal degree, being understood that one always has to add $1$. Another way to deal with this convention is to think of dependency graphs as having one loop at each vertex.
\end{remark}

\begin{example}
The following example, though quite artificial, shows that one can construct families of random variables with arbitrary parameters $N$ and $D$ for their dependency graphs. Let $(U_k)_{k \in \Z / N\Z}$ be a family of independent Bernoulli random variables with $\proba[U_k=1]=1-\proba[U_k=0]=q$; and for $i \in \Z / N\Z$,
$$A_{i} = 2\left(\prod_{k=i+1}^{i+D} U_k\right)-1.$$
Each $A_i$ is a Bernoulli random variable with $\proba[A_i=1] = 1-\proba[A_i=-1] = q^D$, which we denote $p$ ($p$ is considered independent of $N$).
We are interested in the fluctuations of $S=\sum_{i=1}^N A_i$. Note that the partial sums $\sum_{i=1}^{k \leq N}A_i$ correspond to random walks with correlated steps: as $D$ increases, the consecutive steps of the random walk have a higher probability to be in the same direction, and therefore, the variance of the sum $S = \sum_{i=1}^N A_i$ grows. We refer to Figure \ref{fig:correlated_walk}, where three such random walks are drawn, with parameters $p=\frac{1}{2}$, $N=1000$ and $D \in \{5,15,30\}$.\medskip

\begin{figure}[ht]
\begin{center}

\begin{tikzpicture}[xscale=1,yscale=1]
\draw [->] (-0.3,0) -- (10.5,0);
\draw [->] (0,-1.5) -- (0,3);
\draw [thick,DarkOrchid] (0.0000,0.01000) -- (2.390,2.400) -- (2.400,2.390) -- (2.690,2.100) -- (2.700,2.110) -- (2.760,2.170) -- (2.770,2.160) -- (3.320,1.610) -- (3.330,1.620) -- (3.660,1.950) -- (3.670,1.940) -- (3.960,1.650) -- (3.970,1.660) -- (4.040,1.730) -- (4.050,1.720) -- (4.340,1.430) -- (4.350,1.440) -- (4.410,1.500) -- (4.420,1.490) -- (4.900,1.010) -- (4.910,1.020) -- (5.000,1.110) -- (5.010,1.100) -- (5.720,0.3900) -- (5.730,0.4000) -- (5.860,0.5300) -- (5.870,0.5200) -- (6.380,0.01000) -- (6.390,0.02000) -- (6.420,0.05000) -- (6.430,0.04000) -- (6.870,-0.4000) -- (6.880,-0.3900) -- (7.590,0.3200) -- (7.600,0.3100) -- (7.890,0.02000) -- (7.900,0.03000) -- (8.130,0.2600) -- (8.140,0.2500) -- (8.610,-0.2200) -- (8.620,-0.2100) -- (8.900,0.07000) -- (8.910,0.06000) -- (9.290,-0.3200) -- (9.300,-0.3100) -- (9.380,-0.2300) -- (9.390,-0.2400) -- (9.680,-0.5300) -- (9.690,-0.5200) -- (9.990,-0.2200) ; 
\draw (10.7,-0.22) node {\textcolor{DarkOrchid}{$D=30$}};
\draw [thick,Red] (0.0000,-0.01000) -- (0.03000,-0.04000) -- (0.04000,-0.03000) -- (0.2500,0.1800) -- (0.2600,0.1700) -- (0.4000,0.03000) -- (0.4100,0.04000) -- (0.5300,0.1600) -- (0.5400,0.1500) -- (0.6800,0.01000) -- (0.6900,0.02000) -- (0.8600,0.1900) -- (0.8700,0.1800) -- (1.200,-0.1500) -- (1.210,-0.1400) -- (1.250,-0.1000) -- (1.260,-0.1100) -- (1.400,-0.2500) -- (1.410,-0.2400) -- (1.910,0.2600) -- (1.920,0.2500) -- (2.060,0.1100) -- (2.070,0.1200) -- (2.110,0.1600) -- (2.120,0.1500) -- (2.420,-0.1500) -- (2.430,-0.1400) -- (2.730,0.1600) -- (2.740,0.1500) -- (2.880,0.01000) -- (2.890,0.02000) -- (2.940,0.07000) -- (2.950,0.06000) -- (3.090,-0.08000) -- (3.100,-0.07000) -- (3.160,-0.01000) -- (3.170,-0.02000) -- (3.600,-0.4500) -- (3.610,-0.4400) -- (4.020,-0.03000) -- (4.030,-0.04000) -- (4.420,-0.4300) -- (4.430,-0.4200) -- (4.450,-0.4000) -- (4.460,-0.4100) -- (4.600,-0.5500) -- (4.610,-0.5400) -- (4.810,-0.3400) -- (4.820,-0.3500) -- (5.160,-0.6900) -- (5.170,-0.6800) -- (5.300,-0.5500) -- (5.310,-0.5600) -- (5.450,-0.7000) -- (5.460,-0.6900) -- (5.490,-0.6600) -- (5.500,-0.6700) -- (5.950,-1.120) -- (5.960,-1.110) -- (5.980,-1.090) -- (5.990,-1.100) -- (6.140,-1.250) -- (6.150,-1.240) -- (6.380,-1.010) -- (6.390,-1.020) -- (6.560,-1.190) -- (6.570,-1.180) -- (6.600,-1.150) -- (6.610,-1.160) -- (6.810,-1.360) -- (6.820,-1.350) -- (7.070,-1.100) -- (7.080,-1.110) -- (7.220,-1.250) -- (7.230,-1.240) -- (7.300,-1.170) -- (7.310,-1.180) -- (7.450,-1.320) -- (7.460,-1.310) -- (7.860,-0.9100) -- (7.870,-0.9200) -- (8.010,-1.060) -- (8.020,-1.050) -- (8.300,-0.7700) -- (8.310,-0.7800) -- (8.590,-1.060) -- (8.600,-1.050) -- (9.350,-0.3000) -- (9.360,-0.3100) -- (9.640,-0.5900) -- (9.650,-0.5800) -- (9.670,-0.5600) -- (9.680,-0.5700) -- (9.820,-0.7100) -- (9.830,-0.7000) -- (9.880,-0.6500) -- (9.890,-0.6600) -- (9.990,-0.7600) ;
\draw (10.7,-0.76) node {\textcolor{Red}{$D=15$}};
\draw [thick,BurntOrange] (0.0000,0.01000) -- (0.07000,0.08000) -- (0.08000,0.07000) -- (0.1300,0.02000) -- (0.1400,0.03000) -- (0.3200,0.2100) -- (0.3300,0.2000) -- (0.3700,0.1600) -- (0.3800,0.1700) -- (0.5300,0.3200) -- (0.5400,0.3100) -- (0.6300,0.2200) -- (0.6400,0.2300) -- (0.6500,0.2200) -- (0.6900,0.1800) -- (0.7000,0.1900) -- (0.7200,0.2100) -- (0.7300,0.2000) -- (0.8100,0.1200) -- (0.8200,0.1300) -- (0.8400,0.1500) -- (0.8500,0.1400) -- (0.8900,0.1000) -- (0.9000,0.1100) -- (1.030,0.2400) -- (1.040,0.2300) -- (1.150,0.1200) -- (1.160,0.1300) -- (1.170,0.1200) -- (1.360,-0.07000) -- (1.370,-0.06000) -- (1.390,-0.04000) -- (1.400,-0.05000) -- (1.500,-0.1500) -- (1.510,-0.1400) -- (1.580,-0.07000) -- (1.590,-0.08000) -- (1.670,-0.1600) -- (1.680,-0.1500) -- (1.740,-0.09000) -- (1.750,-0.1000) -- (1.790,-0.1400) -- (1.800,-0.1300) -- (1.810,-0.1200) -- (1.820,-0.1300) -- (1.920,-0.2300) -- (1.930,-0.2200) -- (2.020,-0.1300) -- (2.030,-0.1400) -- (2.070,-0.1800) -- (2.080,-0.1700) -- (2.200,-0.05000) -- (2.210,-0.06000) -- (2.300,-0.1500) -- (2.310,-0.1400) -- (2.360,-0.09000) -- (2.370,-0.1000) -- (2.410,-0.1400) -- (2.420,-0.1300) -- (2.480,-0.07000) -- (2.490,-0.08000) -- (2.530,-0.1200) -- (2.540,-0.1100) -- (2.600,-0.05000) -- (2.610,-0.06000) -- (2.730,-0.1800) -- (2.740,-0.1700) -- (2.780,-0.1300) -- (2.790,-0.1400) -- (2.840,-0.1900) -- (2.850,-0.1800) -- (2.880,-0.1500) -- (2.890,-0.1600) -- (3.030,-0.3000) -- (3.040,-0.2900) -- (3.140,-0.1900) -- (3.150,-0.2000) -- (3.320,-0.3700) -- (3.330,-0.3600) -- (3.340,-0.3500) -- (3.350,-0.3600) -- (3.430,-0.4400) -- (3.440,-0.4300) -- (3.540,-0.3300) -- (3.550,-0.3400) -- (3.660,-0.4500) -- (3.670,-0.4400) -- (3.790,-0.3200) -- (3.800,-0.3300) -- (3.840,-0.3700) -- (3.850,-0.3600) -- (3.870,-0.3400) -- (3.880,-0.3500) -- (3.920,-0.3900) -- (3.930,-0.3800) -- (3.990,-0.3200) -- (4.000,-0.3300) -- (4.070,-0.4000) -- (4.080,-0.3900) -- (4.100,-0.3700) -- (4.110,-0.3800) -- (4.150,-0.4200) -- (4.160,-0.4100) -- (4.390,-0.1800) -- (4.400,-0.1900) -- (4.440,-0.2300) -- (4.450,-0.2200) -- (4.500,-0.1700) -- (4.510,-0.1800) -- (4.600,-0.2700) -- (4.610,-0.2600) -- (4.620,-0.2500) -- (4.630,-0.2600) -- (4.730,-0.3600) -- (4.740,-0.3500) -- (4.900,-0.1900) -- (4.910,-0.2000) -- (4.950,-0.2400) -- (4.960,-0.2300) -- (5.110,-0.08000) -- (5.120,-0.09000) -- (5.160,-0.1300) -- (5.170,-0.1200) -- (5.350,0.06000) -- (5.360,0.05000) -- (5.410,0.0000) -- (5.420,0.01000) -- (5.530,0.1200) -- (5.540,0.1100) -- (5.580,0.07000) -- (5.590,0.08000) -- (5.770,0.2600) -- (5.780,0.2500) -- (5.820,0.2100) -- (5.830,0.2200) -- (5.880,0.2700) -- (5.890,0.2600) -- (6.080,0.07000) -- (6.090,0.08000) -- (6.130,0.1200) -- (6.140,0.1100) -- (6.200,0.05000) -- (6.210,0.06000) -- (6.310,0.1600) -- (6.320,0.1500) -- (6.460,0.01000) -- (6.470,0.02000) -- (6.500,0.05000) -- (6.510,0.04000) -- (6.560,-0.01000) -- (6.570,0.0000) -- (6.670,0.1000) -- (6.680,0.09000) -- (6.720,0.05000) -- (6.730,0.06000) -- (6.740,0.05000) -- (6.780,0.01000) -- (6.790,0.02000) -- (6.850,0.08000) -- (6.860,0.07000) -- (6.910,0.02000) -- (6.920,0.03000) -- (7.170,0.2800) -- (7.180,0.2700) -- (7.220,0.2300) -- (7.230,0.2400) -- (7.370,0.3800) -- (7.380,0.3700) -- (7.430,0.3200) -- (7.440,0.3300) -- (7.730,0.6200) -- (7.740,0.6100) -- (7.820,0.5300) -- (7.830,0.5400) -- (7.860,0.5700) -- (7.870,0.5600) -- (7.910,0.5200) -- (7.920,0.5300) -- (7.940,0.5500) -- (7.950,0.5400) -- (8.040,0.4500) -- (8.050,0.4600) -- (8.070,0.4800) -- (8.080,0.4700) -- (8.140,0.4100) -- (8.150,0.4200) -- (8.160,0.4100) -- (8.200,0.3700) -- (8.210,0.3800) -- (8.290,0.4600) -- (8.300,0.4500) -- (8.340,0.4100) -- (8.350,0.4200) -- (8.360,0.4100) -- (8.410,0.3600) -- (8.420,0.3700) -- (8.430,0.3800) -- (8.440,0.3700) -- (8.480,0.3300) -- (8.490,0.3400) -- (8.660,0.5100) -- (8.670,0.5000) -- (8.810,0.3600) -- (8.820,0.3700) -- (8.990,0.5400) -- (9.000,0.5300) -- (9.100,0.4300) -- (9.110,0.4400) -- (9.240,0.5700) -- (9.250,0.5600) -- (9.290,0.5200) -- (9.300,0.5300) -- (9.350,0.5800) -- (9.360,0.5700) -- (9.400,0.5300) -- (9.410,0.5400) -- (9.530,0.6600) -- (9.540,0.6500) -- (9.580,0.6100) -- (9.590,0.6200) -- (9.730,0.7600) -- (9.740,0.7500) -- (9.780,0.7100) -- (9.790,0.7200) -- (9.810,0.7400) -- (9.820,0.7300) -- (9.870,0.6800) -- (9.880,0.6900) -- (9.930,0.7400) -- (9.940,0.7300) -- (9.980,0.6900) -- (9.990,0.7000) ;
\draw (10.6,0.7) node {\textcolor{BurntOrange}{$D=5$}};
\end{tikzpicture}

\end{center}
\caption{Random walks with $1000$ steps and correlation lengths $D=5$, $D=15$ and $D=30$.\label{fig:correlated_walk}}
\end{figure}
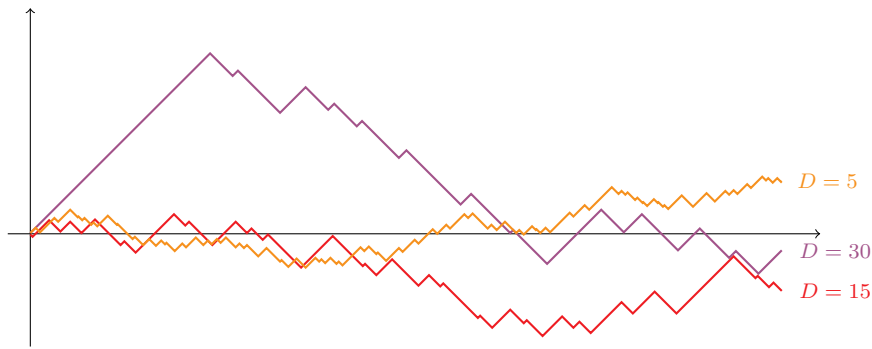

If $d(i,j) \geq D$ in $\Z/N\Z$, then $A_i$ and $A_j$ do not involve any common variable $U_k$, and they are independent. It follows that if $G$ is the graph with vertex set $\Z/N\Z$ and with an edge $e$ between $i$ and $j$ if $d(i,j) \leq D$, then $G$ is a dependency graph for the $A_i$'s. This graph has $N$ vertices, and maximal degree $2D-1$. Moreover, one can compute exactly the expectation and the variance of $S=\sum_{i=1}^N A_i$:
\begin{align*}
\esper[S] &= N(2p-1); \\
\var(S) &= 4 \sum_{i,j=1}^N \mathrm{cov}(U_{i+1}\cdots U_{i+D}, U_{j+1}\cdots U_{j+N}) \\
&=4 \sum_{i,j=1}^N \big(q^{D+\min(D,d(i,j))} -p^2\big) = 4Np \sum_{j=-(D-1)}^{D-1} (q^{|j|}-p) \\
&= 4Np \left(\frac{1+p^{\frac{1}{D}}-2p}{1-p^{\frac{1}{D}}}-(2D-1)p\right).
\end{align*}
If $N$ and $D$ go to infinity with $D=o(N)$, then $q=p^{\frac{1}{D}} = 1+\frac{\log p}{D} + O(\frac{1}{D^2})$, and
\begin{align*}
\esper[S] &= N(2p-1); \\
\var(S) &= 8N(D+O(1))\,p\left(\frac{1-p}{-\log p}-p \right).
\end{align*}
So, one can apply Corollary \ref{cor:Cumulants_Kol} to the sum $S$, and one obtains:
$$\dkol\left(\frac{S-\esper[S]}{\sqrt{\var(S)}},\,\gauss\right) \leq \frac{6}{\left(p\left(\frac{1-p}{-\log p}-p\right)\right)^{3/2}}\,\sqrt{\frac{D}{N}}.$$
\end{example}

\begin{example}
Fix $p \in (0,1)$, and consider a random \emph{Erd\"os--R\'enyi graph} $G=G_n=G(n,p)$, which means that one keeps at random each edge of the complete graph $K_n$ with probability $p$, independently from every other edge.
Note that we only consider the case of fixed $p$ here;
for $p \to 0$, we would get rather weak bounds,
see \cite[Section 10.3.3]{FMN16} for a discussion on bounds on cumulants in this framework.
\medskip

Let $H=(V_H,E_H)$ and $G=(V_G,E_G)$ be two graphs.
The number of copies of $H$ in $G$ is the number of injections $i : V_H \to V_G$ such that,
if $(h_1,h_2) \in E_H$, then $(i(h_1),i(h_2)) \in E_G$.
In random graph theory, this is called the \emph{subgraph count} statistics;
we denote it by $I(H,G)$.
We refer to Figure \ref{fig:counttriangle} for an example, where $H=K_3$ is the triangle and
$G$ is a random Erd\"os--R\'enyi graph of parameters $n=30$ and $p=\frac{1}{10}$.

\begin{figure}[ht]
\begin{center}
\begin{tikzpicture}[scale=1]
\draw [gray] (0:3cm) -- (48:3cm);
\draw [gray] (0:3cm) -- (84:3cm);	
\draw [gray] (0:3cm) -- (216:3cm);
\draw [gray] (0:3cm) -- (252:3cm);
\draw [gray] (12:3cm) -- (24:3cm);
\draw [gray] (12:3cm) -- (108:3cm);
\draw [gray] (12:3cm) -- (132:3cm);
\draw [gray] (12:3cm) -- (348:3cm);
\draw [gray] (24:3cm) -- (204:3cm);
\draw [gray] (24:3cm) -- (336:3cm);
\draw [gray] (36:3cm) -- (60:3cm);
\draw [gray] (36:3cm) -- (144:3cm);
\draw [gray] (36:3cm) -- (336:3cm);
\draw [gray] (48:3cm) -- (264:3cm);
\draw [gray] (60:3cm) -- (72:3cm);
\draw [gray] (72:3cm) -- (108:3cm);
\draw [gray] (72:3cm) -- (264:3cm);
\draw [gray] (96:3cm) -- (276:3cm);
\draw [gray] (120:3cm) -- (144:3cm);
\draw [gray] (120:3cm) -- (156:3cm);
\draw [gray] (120:3cm) -- (216:3cm);
\draw [gray] (120:3cm) -- (348:3cm);
\draw [gray] (132:3cm) -- (228:3cm);
\draw [gray] (132:3cm) -- (324:3cm);
\draw [gray] (144:3cm) -- (264:3cm);
\draw [gray] (156:3cm) -- (276:3cm);
\draw [gray] (192:3cm) -- (228:3cm);
\draw [gray] (192:3cm) -- (288:3cm);
\draw [gray] (204:3cm) -- (240:3cm);
\draw [gray] (216:3cm) -- (240:3cm);
\draw [gray] (228:3cm) -- (240:3cm);
\draw [gray] (252:3cm) -- (336:3cm);
\draw [gray] (276:3cm) -- (312:3cm);
\draw [gray] (312:3cm) -- (336:3cm);

\draw [DarkOrchid!33!NavyBlue,very thick] (48:3cm) -- (96:3cm);
\draw [DarkOrchid!33!NavyBlue,very thick] (48:3cm) -- (132:3cm);
\draw [DarkOrchid!33!NavyBlue,very thick] (96:3cm) -- (132:3cm);
\draw [DarkOrchid!66!NavyBlue,very thick] (228:3cm) -- (252:3cm);
\draw [DarkOrchid!66!NavyBlue,very thick] (228:3cm) -- (264:3cm);
\draw [DarkOrchid!66!NavyBlue,very thick] (252:3cm) -- (264:3cm);
\draw [DarkOrchid,very thick] (24:3cm) -- (84:3cm);
\draw [DarkOrchid,very thick] (24:3cm) -- (156:3cm);
\draw [DarkOrchid,very thick] (84:3cm) -- (156:3cm);
\draw [very thick,NavyBlue] (60:3cm) -- (96:3cm);
\draw [very thick,NavyBlue] (60:3cm) -- (216:3cm);
\draw [very thick,NavyBlue] (96:3cm) -- (216:3cm);
\foreach \x in {0,12,24,36,48,60,72,84,96,108,120,132,144,156,168,180,192,204,216,228,240,252,264,276,288,300,312,324,336,348}
	\fill (\x:3cm) circle (2pt);
\end{tikzpicture}

\end{center}
\caption{Count of triangles in a random Erd\"os--R\'enyi graph of parameters $n=30$ and $p=0.1$. Here, there are $4 \times 3! = 24$ ways to embed a triangle in the graph.}\label{fig:counttriangle}
\end{figure}
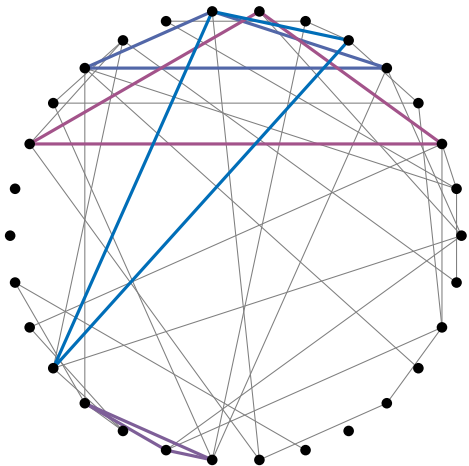
\medskip

One can always write $I(H,G)$ as a sum of dependent random variables. Identify $V_H$ with $\lle 1,k\rre$ and $V_G$ with $\lle 1,n\rre$, and denote $\mathfrak{A}(n,k)$ the set of arrangements $(a_1,\ldots,a_k)$ of size $k$ in $\lle 1,n\rre$. Given such an arrangement, the \emph{induced subgraph} $G[a_1,\ldots,a_k]$ is the graph with vertex set $\lle 1,k\rre$, and with an edge between $i$ and $j$ if $(a_i,a_j) \in E_G$. Then,
$$I(H,G) = \sum_{A \in \mathfrak{A}(n,k)} I_A(H,G),$$
where $I_A(H,G)=1$ if $H \subset G[A]$, and $0$ otherwise.\bigskip

A dependency graph for the random variables $I_A(H,G_n)$ has vertex set $\mathfrak{A}(n,k)$ of cardinality $N_n = n^{\downarrow k} =n(n-1)(n-2)\cdots (n-k+1)$, and an edge between two arrangements $(a_1,\ldots,a_k)$ and $(b_1,\ldots,b_k)$ if they share at least two points (otherwise, the random variables $I_A(H,G_n)$ and $I_B(H,G_n)$ involve disjoint sets of edges and are therefore independent). As a consequence, the maximal degree of the graph is smaller than
$$D_n=\left(\binom{k}{2}^2 2(n-2)(n-3)\cdots (n-k+1) \right),$$
and of order $n^{k-2}$. Therefore, $\frac{D_n}{N_n} \leq \frac{2 \binom{k}{2}^2}{n(n-1)} = O(\frac{1}{n^2})$, and on the other hand, if $h$ is the number of edges of $H$, one can compute the asymptotics of the expectation and of the variance of $I(H,G_n)$:
\begin{align*}
\esper[I(H,G_n)] &= n^{\downarrow k} p^h ;\\
 \var(I(H,G_n)) &= 2h^2p^{2h-1}(1-p)\,n^{2k-2} + O(n^{2k-3}),
\end{align*}
see \cite[Section 10]{FMN16} for the details of these computations. In particular,
$$\lim_{n \to \infty} \frac{\var(S_n)}{N_n\,D_n} = p^{2h-1}(1-p) \left(\frac{h}{\binom{k}{2}}\right)^{\!2} = \widetilde{\sigma}^2 >0. $$
Thus, using Corollary~\ref{cor:Cumulants_Kol}, we get
$$\dkol\left(\frac{I(H,G_n) - \esper[I(H,G_n)]}{\sqrt{\var(I(H,G_n))}},\,\gauss\right) \leq \frac{4.65\,(k(k-1))^4}{p^{3h}\,(\frac{1}{p}-1)^{3/2}\,h^3} \,\,\frac{1}{n} $$
for $n$ large enough. For instance, if $T_n=I(K_3,G_n)$ is the number of triangles in $G_n$, then
\begin{align*}
\esper[T_n]&=n^{\downarrow 3}\,p^3 \\
\var[T_n]&=18\,n^{\downarrow 4}\,p^5(1-p) + 6\,n^{\downarrow 3}\,p^3(1-p^3)
\end{align*}
and
\begin{align*}
&\dkol\left(\frac{T_n-n^{\downarrow 3}p^3}{\sqrt{18\,n^{\downarrow 4}\,p^5(1-p) + 6\,n^{\downarrow 3}\,p^3(1-p^3)}},\,\gauss\right) \\
&\leq \frac{234}{p^9(\frac{1}{p}-1)^{\frac{3}{2}}}\,\frac{1}{n} \, (1 +O(\tfrac1n)).
\end{align*}

\noindent 
This result is not new, except maybe the explicit constant.
We refer to \cite{BKR89} for an approach of speed of convergence for subgraph counts
using Stein's method.
More recently, Krokowski, Reichenbachs and Thäle \cite{KRT15} applied Malliavin 
calculus to the same problem.
Our result corresponds to the case where $p$ is constant of their Theorem 1. 
Similar bounds could be obtained by Stein's method, see \cite{Rin94}.
\end{example}

To conclude our presentation of the convergence of sums of bounded random variables with sparse dependency graphs, let us analyse precisely the case of \emph{uncorrelated} random variables.
\begin{corollary}\label{cor:uncorrelated}
Let $S_n=\sum_{i=1}^{N_n} A_{i,n}$ be a sum of centered and bounded random variables, that are uncorrelated and with $\esper[(A_{i,n})^2]=1$ for all $i$. We suppose that the random variables have a dependency graph of parameters $N_n \to +\infty$ and $D_n$.
\begin{enumerate}
 	\item If $D_n = O((N_n)^{1/2-\eps})$ for $\eps>0$, then $Y_n=\frac{S_n}{\sqrt{N_n}}$ converges in law towards the Gaussian distribution. 
 	\item If $D_n = o((N_n)^{1/4})$, then the Kolmogorov distance between $Y_n$ and $\gauss$ is a $O((D_n)^2/(N_n)^{1/2})$.
 \end{enumerate}  
\end{corollary}

\begin{proof}
  It is an immediate consequence of Corollary~\ref{cor:Cumulants_Kol} and Proposition~\ref{thm:CLT_Janson},
  since $S_n$ admits uniform control on cumulants (see Theorem~\ref{thm:boundcumulant})
  and
  \begin{equation*}
  \var(S_n) = \sum_{i=1}^{N_n} \esper[(A_{i,n})^2] = N_n.\qquad\qed
  \end{equation*}
\end{proof}

\subsection{Unbounded random variables and truncation methods}
A possible generalization regards sums of \emph{unbounded} random variables.
In the following, we develop a truncation method that yields a criterion of asymptotic normality similar to Lyapunov's condition (see \cite[Chapter 27]{Bil95}). A small modification of this method would similarly yield a Lindeberg type criterion. Let $S_n = \sum_{i=1}^{N_n} A_{i,n}$ be a sum of centered random variables, with 
$$\left(\esper[|A_{i,n}|^{2+\delta}]\right)^{\frac{1}{2+\delta}} \leq A$$
for some constant $A$ independent of $i$ and $n$, and some $\delta>0$. We suppose as before that the family of random variables $(A_{i,n})_{i\in \lle 1,N_n\rre}$ has a (true) dependency graph $G_n$ of parameters $N_n$ and $D_n$. Note that in this case, 
\begin{align*}
\var(S_n) &= \sum_{i,j=1}^{N_n} \mathrm{cov}(A_{i,n},A_{j,n}) \leq \sum_{i=1}^{N_n} \sum_{j \sim i} \|A_{i,n}\|_2\,\|A_{j,n}\|_2 \\
&\leq \sum_{i=1}^{N_n} \sum_{j \sim i} \|A_{i,n}\|_{2+\delta}\,\|A_{j,n}\|_{2+\delta} \leq A^2 \,N_n D_n.
\end{align*}
We set
\begin{align*}
A_{i,n}^- &= A_{i,n}\,1_{|A_{i,n}| \leq L_n} \qquad;\qquad A_{i,n}^+ = A_{i,n}\,1_{|A_{i,n}| > L_n}\,; \\
S_n^- &=\sum_{i=1}^{N_n} A_{i,n}^- \qquad\qquad\,\,\,;\qquad S_n^+=\sum_{i=1}^{N_n} A_{i,n}^+\,;
\end{align*}
where $L_n$ is a truncation level, to be chosen later. Notice that $G_n$ is still a dependency graph for the family of truncated random variables $(A_{i,n}^-)_{i \in \lle 1,N_n\rre}$. Therefore, we can apply the previously developed machinery
(Theorem \ref{thm:boundcumulant} and Corollary \ref{cor:Cumulants_Kol})
to the scaled sum $S_n^-/L_n$. On the other hand, by Markov's inequality,
\begin{align*}
\dkol(S_n,S_n^-) &= \sup_{s \in \R} |\proba[S_n \geq s] - \proba[S_n^- \geq s]| \leq \proba[S_n^+ = 0] \\
&\leq \sum_{i=1}^{N_n} \proba[|A_{i,n}| \geq L_n] \leq N_n\,\left(\frac{A}{L_n}\right)^{\!\!2+\delta}.
\end{align*}
Combining the two arguments leads to the following result
(this replaces the previous assumption of boundedness $|A_{i,n}| \leq A$). 

\begin{theorem}\label{thm:unbounded}
Let $(S_n = \sum_{i=1}^{N_n} A_{i,n})_{n \in \N}$ be a sum of centered random variables, with dependency graphs of parameters $N_n \to +\infty$ and $D_n$, and with
$$\|A_{i,n}\|_{2+\delta}=(\esper[|A_{i,n}|^{2+\delta}])^{1/(2+\delta)} \leq A$$ 
for all $i,n$ and for some $\delta>0$. 
We set $Y_n = S_n/\sqrt{\var(S_n)}$. Recall that 
$(\widetilde{\sigma}_n)^2 = \frac{\var(S_n)}{N_n D_n}$.
\begin{enumerate}[label=(U\arabic*)]
	\item\label{hyp:unbounded1} Set
      \[ V_n=(\widetilde{\sigma}_n)^2\,\left(\frac{N_n}{D_n}\right)^{\!1/3} \,\frac{1}{(N_n)^{2/(2+\delta)}}\]
      and suppose that $\lim_{n \to \infty} V_n= +\infty$ 
	(which is only possible for $\delta > 4$). Then, for $n$ large enough,
	$$\dkol(Y_n,\,\gauss) \leq 78\left(\frac{A^2}{V_n}\right)^{\!\frac{3(\delta+2)}{2(\delta+5)}}=o\!\left(\frac{1}{V_n}\right).$$
  \item\label{hyp:unbounded2} More generally, for $\eps \in (\frac{2}{2+\delta},1)$,
    set
     \[W_n=(\widetilde{\sigma}_n)^2\,\left(\frac{N_n}{D_n}\right)^{\!\eps} \,\frac{1}{(N_n)^{2/(2+\delta)}}\]
     and suppose that $\lim_{n \to \infty} W_n= +\infty$.
      Then, $Y_n \rightharpoonup \gauss$.
\end{enumerate}
\end{theorem}

\begin{remark}
It should be noticed that if $\delta \to +\infty$, then one essentially recovers the content of Corollary \ref{cor:Cumulants_Kol}
(which can be applied because of Theorem \ref{thm:boundcumulant}).
On the other hand, the inequality $\delta > 4$ amounts to the existence of bounded moments of order strictly higher than $6$ for the random variables $A_{i,n}$. In practice, one can for instance ask for bounded moments of order $7$
(\emph{i.e.}~$\delta=5$), in which case the first condition \ref{hyp:unbounded1} reads
$$\lim_{n \to \infty} \frac{\var(S_n)}{D_n\,N_n} \, \frac{(N_n)^{1/9}}{(D_n)^{1/3}} = +\infty.$$
Moreover, we will see in the proof of Theorem \ref{thm:unbounded} that in this setting ($\delta=5$), the constant $78$ can be improved to $39$, so that, for $n$ large enough:
$$\dkol(Y_n,\,\gauss)\leq 39\left(\frac{A^2}{V_n}\right)^{\frac{21}{20}}.$$
\end{remark}

\begin{proof} 
We write as usual $Y_n = \frac{S_n}{\sqrt{\var(S_n)}}$, and $Y_n^- = \frac{S_n^-}{\sqrt{\var(S_n^-)}}$. In all cases, we have
\begin{align*}
&\dkol(Y_n,\,\gauss) \leq \dkol(S_n,S_n^-) + \dkol\left(Y_n^-,\,\sqrt{\frac{\var(S_n)}{\var(S_n^-)}}\,\gauss\right) \\
&\leq N_n\left(\frac{A}{L_n}\right)^{2+\delta} + \dkol(Y_n^-,\gauss) + \dkol\left(\gauss,\,\mathcal{N}_{\R}\!\left(0,\,\frac{\var(S_n)}{\var(S_n^-)}\right)\right)
\end{align*}
by using the invariance of the Kolmogorov distance with respect to multiplication of random variables by a positive constant. In the sequel, we denote $a$, $b$ and $c$ the three terms on the second line of the inequality.
The Kolmogorov distance between two Gaussian distributions is
\begin{align*}
\dkol(\gauss,\,\mathcal{N}_{\R}(0,\lambda^2)) &= \frac{1}{\sqrt{2\pi}}\sup_{s \in \R_+} \left( \int_{s}^{\lambda s} \E^{-\frac{u^2}{2}}\DD{u} \right)\\
&\leq \frac{\lambda-1}{\sqrt{2\pi}}\,\sup_{s \in \R_+} \left(s\,\E^{-\frac{s^2}{2}}\right) =\sqrt{\frac{1}{2\pi\E}}\,|\lambda-1|
\end{align*}
if $\lambda \geq 1$. One gets the same result if $\lambda \leq 1$, hence, 
\begin{align*}
\dkol\left(\gauss,\,\mathcal{N}_{\R}\!\left(0,\,\frac{\var(S_n)}{\var(S_n^-)}\right)\right) &= \dkol\left(\gauss,\,\mathcal{N}_{\R}\!\left(0,\,\frac{\var(S_n^-)}{\var(S_n)}\right)\right) \\
&\leq \sqrt{\frac{1}{2\pi\E}} \left|\sqrt{\frac{\var(S_n^-)}{\var(S_n)}}-1\right| \\
&\leq \sqrt{\frac{1}{2\pi\E}}\,\frac{|\var(S_n^-)-\var(S_n)|}{\var(S_n)}.
\end{align*}
To evaluate the difference between the variances, notice that
\begin{align*}
\var(S_n^-) &= \var\left(S_n - \sum_{i=1}^{N_n} A_{i,n}^+\right) \\
&=\var(S_n) - 2 \sum_{i,j=1}^{N_n} \cov(A_{i,n}^+,A_{j,n}) + \sum_{i,j=1}^{N_n} \cov(A_{i,n}^+,A_{j,n}^+)
\end{align*}
If $j$ is not connected to $i$ in $G_n$, or equal to $i$, then $A_{i,n}$ and $A_{j,n}$ are independent, hence, $\cov(A_{i,n}^+,A_{j,n})=0$. Otherwise, using H\"older and Bienaym\'e-Chebyshev inequalities,
\begin{align*}
|\cov(A_{i,n}^+,A_{j,n})| &\leq \sqrt{\esper[(A_{i,n})^2\,1_{|A_{i,n}|\geq L_n}]\,\esper[(A_{j,n})^2\,1_{|A_{i,n}|\geq L_n}]} \\
&\leq \sqrt{\esper[(A_{i,n})^{2+\delta}]^{\frac{2}{2+\delta}}\,\proba[|A_{i,n}| \geq L_n]^{\frac{2\delta}{2+\delta}}\,\esper[(A_{j,n})^{2+\delta}]^{\frac{2}{2+\delta}} } \\
&\leq A^2\,\left(\frac{A}{L_n}\right)^{\!\delta}.
\end{align*}
Similarly, 
\begin{align*}
|\cov(A_{i,n}^+,A_{j,n}^+)| &\leq \esper[|A_{i,n}^+A_{j,n}^+|] + \esper[|A_{i,n}^+|]\,\esper[|A_{j,n}^+|]\\
&\leq A^2\left(\left(\frac{A}{L_n}\right)^\delta + \left(\frac{A}{L_n}\right)^{2+2\delta}\right),
\end{align*}
hence, assuming that the level of truncation $L_n$ is larger than $A$,
$$\frac{|\var(S_n)^- - \var(S_n)|}{\var(S_n)} \leq \frac{N_n\,D_n}{\var(S_n)}\,3A^2\,\left(\frac{A}{L_n}\right)^{\!\delta}.$$
Let us now place ourselves in the setting of Hypothesis \ref{hyp:unbounded1}; we set 
$$ V_n = \frac{\var(S_n)}{N_n\,D_n}\,\left(\frac{N_n}{D_n}\right)^{\!1/3}\,\frac{1}{(N_n)^{2/(2+\delta)}}.$$
Suppose that $L_n = K_n\,(N_n)^{\frac{1}{2+\delta}}$, with $K_n$ going to infinity. We then have $a \leq \frac{A^{2+\delta}}{(K_n)^{2+\delta}}$, and on the other hand,
\begin{align*}
\left|\frac{\var(S_n^-)}{\var(S_n)}-1\right| &\leq  3A^{2+\delta}\left(\frac{N_n\,D_n}{\var(S_n)}\,\frac{1}{(N_n)^{\delta/(2+\delta)}}\,\frac{1}{(K_n)^{\delta}}\right) \\
&\leq 3A^{2+\delta} \left(\frac{1}{V_n\,(D_n)^{1/3}\,(N_n)^{2/3}\,(K_n)^\delta}\right) \to 0
\end{align*}
since by hypothesis, $\lim_{n \to \infty} \frac{1}{V_n} = 0$. So, 
$$c \leq \frac{3A^{2+\delta}}{\sqrt{2\pi\E}}\,\frac{1}{V_n\,(D_n)^{1/3}\,(N_n)^{2/3}\,(K_n)^\delta} \leq \frac{3A^{2+\delta}}{\sqrt{2\pi\E}}\,\frac{1}{V_n\,(K_n)^\delta}.$$ 
Now, the sequence $(S_n^-/L_n)_{n \in \N}$ is a sequence of sums of centered random variables all bounded by $1$, 
and to apply Corollary~\ref{cor:Cumulants_Kol} to this sequence, we need
$$\lim_{n \to \infty} \frac{\var(S_n^-)}{N_n\,D_n}\,\left(\frac{N_n}{D_n}\right)^{\!1/3}\,\frac{1}{(L_n)^2} = +\infty.$$
However, the previous computation shows that one can replace $\var(S_n)$ by $\var(S_n^-)$ in this expression without changing the asymptotic behavior, so
$$
\lim_{n \to \infty} \frac{\var(S_n^-)}{N_n\,D_n}\,\left(\frac{N_n}{D_n}\right)^{\!1/3}\,\frac{1}{(L_n)^2} = \lim_{n \to \infty} \frac{V_n}{(K_n)^2}, $$
which is $+\infty$ if $(K_n)^2$ is not growing too fast to $+\infty$ (in comparison to the sequence $V_n$). 
Then, by Corollary~\ref{cor:Cumulants_Kol},
$$b=\dkol(Y_n^-,\,\gauss) \leq 77\, \frac{(K_n)^3}{(V_n)^{3/2}}$$
for $n$ large enough. Set $K_n = B(V_n)^{\frac{3}{2(5+\delta)}}$. Then,
\begin{align*}
\dkol(Y_n,\,\gauss)&\leq a+b+c \leq \left(\frac{A^{2+\delta}}{B^{2+\delta}}+77\,B^3+o(1)\right) \left(\frac{1}{V_n}\right)^{\frac{3(2+\delta)}{2(5+\delta)}} \\
&\leq \left(\left(\frac{231}{2+\delta}\right)^{\frac{2+\delta}{5+\delta}} + 77\,\left(\frac{2+\delta}{231}\right)^{\frac{3}{5+\delta}}\right) \left(\frac{A^2}{V_n}\right)^{\frac{3(2+\delta)}{2(5+\delta)}}
\end{align*}
for $n$ large enough, and by choosing $B$ in an optimal way. The term in parenthesis is maximal when $\delta=229$, and is then equal to $78$. This ends the proof of \ref{hyp:unbounded1}, and one gets a better constant smaller than $39$ when $\delta=5$.
\bigskip

Under the Hypothesis \ref{hyp:unbounded2}, we set $$W_n= \frac{\var(S_n)}{N_n\, D_n}\,\left(\frac{N_n}{D_n}\right)^{\!\eps} \,\frac{1}{(N_n)^{2/(2+\delta)}}.$$
In order to prove the convergence in law $Y_n \rightharpoonup \gauss$, it suffices to have:
\begin{itemize}
	\item $S_n-S_n^- = S_n^+$ that converges in probability to $0$. This happens as soon as the level $L_n$ is $K_n\, (N_n)^{1/(2+\delta)}$ with $K_n \to +\infty$.
	\item $\frac{|\var(S_n)-\var(S_n^-)|}{\var(S_n)} \to 0$. With $L_n = K_n\,(N_n)^{1/(2+\delta)}$, the previous computations show that this quantity is a 
	$$O\left(\frac{1}{W_n\,(N_n)^{1 -\eps}\,(D_n)^{\eps}\,(K_n)^{\delta}}\right),$$
	which goes to $0$.
	\item and by Theorem \ref{thm:CLT_Janson}, 
	$$\frac{\var(S_n)}{N_n\,D_n}\left(\frac{N_n}{D_n}\right)^\eps\,\frac{1}{(L_n)^2} \to +\infty .$$
	This follows from the Hypothesis \ref{hyp:unbounded2} if $K_n$ is chosen to grow sufficiently slow. 
\end{itemize}
Thus, the second part of Theorem \ref{thm:unbounded} is proven.
\qed
\end{proof}

\begin{example}
Let $(X_i)_{i \in \lle 1,N\rre}$ be a centered Gaussian vector with $\esper[(X_i)^2]=1$ for any $i$, and with the covariance matrix $(\cov(X_i,X_j))_{1\leq i,j \leq N}$ that is sparse in the following sense: for any $i$, the set of indices $j$ such that $\cov(X_i,X_j) \neq 0$ has cardinality smaller than $D$. We set $A_i=\exp(X_i)$; the random variables $A_i$ follow the \emph{log-normal distribution} of density
$$\frac{1}{\sqrt{2\pi}}\,\frac{1}{u^{1+\frac{\log u}{2}}}\,1_{u>0}\DD{u}, $$
and they have moments of all order:
$\esper[(A_i)^{k}] = \esper[\E^{kX_i}] = \E^{\frac{k^2}{2}}.$
\begin{figure}[ht]
\begin{center}
\begin{tikzpicture}[xscale=1,yscale=2]
\fill[domain=0.1:6.2, smooth, line width=1.2pt, color=DarkOrchid!15!white] (0,0) -- plot (\x,{\x^(-1-ln(\x)/2)/sqrt(6.28)}) -- (6.2,0) ;
\draw[domain=0.1:6.2, smooth, line width=1.2pt, color=DarkOrchid] (0,0) -- plot (\x,{\x^(-1-ln(\x)/2)/sqrt(6.28)}) ;
\draw[->] (-0.5,0) -- (6.5,0);
\draw[->] (-0,-0.25) -- (0,1);
\foreach \x in {1,2,3,4,5,6}
		\draw (\x,-0.05) -- (\x,0.05);
\end{tikzpicture}
\end{center}
\caption{The density of the log-normal distribution.}
\end{figure}
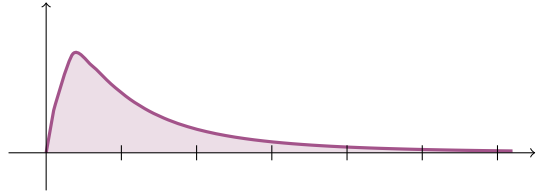

The variables $A_i$ have the same dependency graph as the variables $X_i$. Moreover, if $\rho_{ij}=\cov(X_i,X_j)$, then the covariance of two variables $A_i$ and $A_j$ is $\E(\E^{\rho_{ij}}-1)$. Using moments of order $2+\delta$, we see that if
\begin{align*}
Y_N &= \frac{\sum_{i=1}^N (A_i-\E^{\frac{1}{2}})}{\sqrt{\E \sum_{1\leq i,j\leq N} (\E^{\rho_{ij}}-1)}}\\ 
V_{N,\delta} &= \frac{\E\sum_{1\leq i,j\leq N} (\E^{\rho_{ij}}-1)}{ND} \frac{N^{\frac{1}{3}- \frac{2}{2+\delta} }}{D^{\frac{1}{3}}}\to +\infty,
\end{align*}
then $\dkol(Y_N,\,\gauss)\leq 78\left(\frac{\E^{\delta+2}}{V_{N,\delta}}\right)^{\frac{3(\delta+2)}{2(\delta+5)}}$ for $N$ large enough.\medskip

To make this result more explicit, let us consider the following dependency structure for the Gaussian vector $X = (X_i)_{i \in \lle 1,N\rre}$: 
$$\cov(X) = \begin{pmatrix}
1 & * & \cdots & * \\
* & 1 & \ddots & \vdots \\
\vdots & \ddots & \ddots & * \\
* & \cdots & * & 1
\end{pmatrix},$$
where the non-diagonal terms $*$ are all smaller than $\frac{\rho}{D}$ in absolute value, and with less than $D$ non-zero terms on each row or column. When $\rho \in [0,1)$, the matrix is diagonally dominant, hence positive-definite, so there exists indeed a Gaussian vector $X$ with these covariances. We then have
$$V_{N,\delta} \geq \E\,(1-D(\E^{\frac{\rho}{D}}-1))\,\frac{N^{\frac{1}{3}-\frac{2}{2+\delta}}}{D^{\frac{4}{3}}} ,$$
so if $1 \ll D \ll N^{\frac{1}{4}-\eps}$, then one can apply Theorem \ref{thm:unbounded} to get
 $$\dkol(Y_N,\,\gauss) \leq 78 \,\left(\frac{\E^{\frac{3}{2\eps}-1}}{1-\rho}\right)^{\!\frac{3}{2}}\,\left(\frac{D}{N^{\frac{1}{4}-\eps}}\right)^{\frac{2}{2\eps+1}}$$
for $N$ large enough. Moreover, as soon as $1 \ll D \ll N^{\frac{1}{2}-\eps}$, $Y_N \rightharpoonup \gauss$.
\end{example}

\section{Ising model and Markov chains}\label{sec:markov}

In this section, we present examples of random variables that admit uniform bounds on cumulants,
which do not come from dependency graphs.
Their structure is nevertheless not so different
since the variables that we consider
write as sums of random variables that are {\em weakly dependent}.
The technique to prove uniform bounds on cumulants relies then on the notion of \emph{uniform weighted dependency graph}, which generalizes the notion of standard dependency graph
(see Proposition \ref{Prop:BoundCumulantUWDG}).

\subsection{Weighted graphs and spanning trees}
An {\em edge-weighted graph} $\WDep$, or {\em weighted graph} for short, is a graph $\WDep$ in which
each edge $e$ is assigned a weight $w_G(e)$. Here we restrict ourselves to weights $w_G(e)$
with $w_G(e) \in \R_+$.
Edges not in the graph can be thought of as edges of weight $0$,
all our definitions being consistent with this convention.
\medskip

If $B$ is a multiset of vertices of $\WDep$,
we can consider the graph $\WDep[B]$ induced by $\WDep$ on $B$ and defined as follows:
the vertices of $\WDep[B]$ correspond to elements of $B$ (if $B$ contains an element with multiplicity $m$,
then $m$ vertices correspond to this element),
and there is an edge between two vertices if the corresponding vertices of $\WDep$ are equal 
or connected by an edge in $\WDep$. This new graph  has a natural weighted graph structure:
put on each edge of $\WDep[B]$ the weight of the corresponding edge in $\WDep$
(if the edge connects two copies of the same vertex of $\WDep$,
we put weight $1$).

\begin{definition}
A spanning tree of a graph $\WDep=(V,E)$ is
a subset $E'$ of $E$ such that $(V,E')$ is a tree.
In other words, it is a subgraph of $\WDep$ that is a tree
and covers all vertices.
\end{definition}
The set of spanning trees of $T$ is denoted $\ST(\WDep)$. If $\WDep$ is a weighted graph, 
we say that the weight $w(T)$ of a spanning tree of $\WDep$
is defined as the {\em product} of the weights of the edges in $T$.

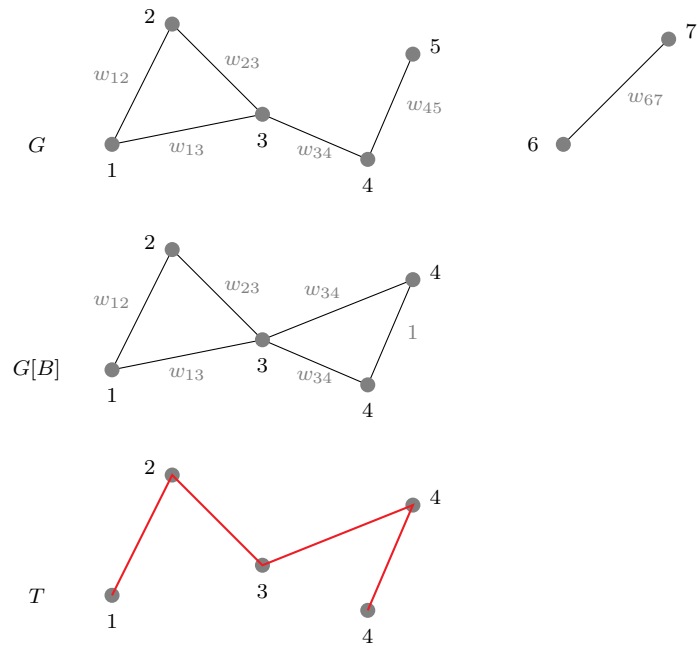
\begin{figure}[ht]
\begin{center}
\begin{tikzpicture}[scale=2]
\draw (1,0.2) -- (0.4,0.8) -- (0,0) -- (1,0.2) -- (1.7,-0.1) -- (2,0.6);
\draw (3,0) -- (3.7,0.7);
\fill[color=white!50!black] (0,0) circle [radius=0.5mm];
\fill[color=white!50!black] (1,0.2) circle [radius=0.5mm];
\fill[color=white!50!black] (0.4,0.8) circle [radius=0.5mm];
\fill[color=white!50!black] (1.7,-0.1) circle [radius=0.5mm];
\fill[color=white!50!black] (2,0.6) circle [radius=0.5mm];
\fill[color=white!50!black] (3,0) circle [radius=0.5mm];
\fill[color=white!50!black] (3.7,0.7) circle [radius=0.5mm];
\draw (0,-0.17) node {$1$};
\draw[color=white!50!black] (0,0.45) node {$w_{12}$};
\draw (0.25,0.85) node {$2$};
\draw[color=white!50!black] (0.87,0.55) node {$w_{23}$};
\draw[color=white!50!black] (0.5,-0.03) node {$w_{13}$};
\draw (1,0.03) node {$3$};
\draw[color=white!50!black] (1.35,-0.05) node {$w_{34}$};
\draw (1.7,-0.27) node {$4$};
\draw[color=white!50!black] (2.08,0.25) node {$w_{45}$};
\draw (2.15,0.65) node {$5$};
\draw (2.8,0) node {$6$};
\draw[color=white!50!black] (3.55,0.3) node {$w_{67}$};
\draw (3.85,0.75) node {$7$};
\draw (-0.5,0) node {$G$};
\begin{scope}[shift={(0,-1.5)}]
\draw (1,0.2) -- (0.4,0.8) -- (0,0) -- (1,0.2) -- (1.7,-0.1) -- (2,0.6) -- (1,0.2);
\fill[color=white!50!black] (0,0) circle [radius=0.5mm];
\fill[color=white!50!black] (1,0.2) circle [radius=0.5mm];
\fill[color=white!50!black] (0.4,0.8) circle [radius=0.5mm];
\fill[color=white!50!black] (1.7,-0.1) circle [radius=0.5mm];
\fill[color=white!50!black] (2,0.6) circle [radius=0.5mm];
\draw (-0.5,0) node {$G[B]$};
\draw (0,-0.17) node {$1$};
\draw[color=white!50!black] (0,0.45) node {$w_{12}$};
\draw (0.25,0.85) node {$2$};
\draw[color=white!50!black] (0.87,0.55) node {$w_{23}$};
\draw[color=white!50!black] (0.5,-0.03) node {$w_{13}$};
\draw (1,0.03) node {$3$};
\draw (1.7,-0.27) node {$4$};
\draw (2.15,0.65) node {$4$};
\draw[color=white!50!black] (1.35,-0.05) node {$w_{34}$};
\draw[color=white!50!black] (1.4,0.52) node {$w_{34}$};
\draw[color=white!50!black] (2.0,0.25) node {$1$};
\end{scope}
\begin{scope}[shift={(0,-3)}]
\draw (-0.5,0) node {$T$};
\fill[color=white!50!black] (0,0) circle [radius=0.5mm];
\fill[color=white!50!black] (1,0.2) circle [radius=0.5mm];
\fill[color=white!50!black] (0.4,0.8) circle [radius=0.5mm];
\fill[color=white!50!black] (1.7,-0.1) circle [radius=0.5mm];
\fill[color=white!50!black] (2,0.6) circle [radius=0.5mm];
\draw (0,-0.17) node {$1$};
\draw (0.25,0.85) node {$2$};
\draw (1,0.03) node {$3$};
\draw (1.7,-0.27) node {$4$};
\draw (2.15,0.65) node {$4$};
\draw [Red,thick] (0,0) -- (0.4,0.8) -- (1,0.2) -- (2,0.6) -- (1.7,-0.1);
\end{scope}
\end{tikzpicture}
\end{center}

\caption{A weighted dependency graph $\WDep$ for $7$ random variables; the induced graph $\WDep[B]$ with $B=\{1,2,3,4,4\}$; and a spanning tree $T$ of $\WDep[B]$, with $w(T)=w_{12}w_{23}w_{34}$.\label{fig:weighteddependency}}
\end{figure}

\subsection{Uniform weighted dependency graphs}
If $A_1,\ldots,A_r$ are real-valued random variables, there is a notion of joint cumulant that generalize the cumulants of Section \ref{sec:dependency}:
$$\kappa(A_1,A_2,\ldots,A_r) = [z_1z_2\cdots z_r]\left(\log \esper[\E^{z_1A_1+z_2A_2+\cdots+z_rA_r}]\right).$$
The joint cumulants are multilinear and symmetric functionals of $A_1,\ldots,A_r$. On the other hand, 
$$\kappa^{(r)}(X) = \kappa(X,X,\ldots,X)$$
with $r$ occurrences of $X$ in the right-hand side. In particular, if $S=\sum_{v \in V}A_v$ is a sum of random variables, then
$$\kappa^{(r)}(S) = \sum_{v_1,\ldots,v_r \in V} \kappa(A_{v_1},A_{v_2},\ldots,A_{v_r}).$$

\begin{definition}\label{def:UWDG}
Let $(A_v)_{v \in V}$ be a family of random variables defined on the same probability space. A weighted graph $\WDep=(V,E,w_G)$ is a $C$-uniform weighted dependency graph for $(A_v)_{v \in V}$ if, 
for any multiset \hbox{$B=\{v_1,\ldots,v_r\}$} of elements of $V$,
one has
$$
    \left| \ka\big( A_v,\, v \in B \big) \right| \leq
    C^{|B|} \sum_{T \in \ST(\WDep[B])} w(T).
$$
\end{definition}

\begin{proposition}\label{Prop:BoundCumulantUWDG}
  Let $(A_v)_{v \in V}$ be a finite family of random variables
  with a $C$-uniform weighted dependency graph $\WDep$.
  Assume that $\WDep$ has $N=|V|$ vertices, and maximal weighted degree $D-1$, that is:
  $$\forall v \in V,\,\,\,\sum_{\substack{v':\,v' \ne v \\ \{v,v'\} \in E_G}} w_G(\{v,v'\}) \leq D-1.$$
  Then, for $r \ge 1$,
  \[ \left| \ka^{(r)}\left( \sum_{v \in V} A_v \right) \right|
  \le N\, r^{r-2} D^{r-1}\,C^r.\]
\end{proposition}
\noindent 
Consider a sequence $(S_n)_{n\in \N}$, where each $S_n$ writes as $\sum_{v \in V_n} A_{v,n}$.
Set $N_n=|V_n|$ and
assume that, for each $n$, $(A_v)_{v \in V_n}$ has $C$-uniform weighted dependency graph of maximal degree $D_n-1$
(by assumption, $C$ does not depend on $n$).
Then
the sequence $(S_n)$ admits uniform bounds on cumulants with parameters $(\tfrac{D_n}{2},N_n,C)$
and the results of Section~\ref{sec:dependency},
in particular Corollary~\ref{cor:Cumulants_Kol}, apply.
\begin{proof}
  By multilinearity and definition of a uniform weighted dependency graph,
  we have
  \begin{equation}
    \left| \ka_r\left( \sum_{v \in V} A_v \right) \right|
  \le C^r \sum_{v_1,\ldots,v_r} \sum_{T \in \ST(\WDep[v_1,\ldots,v_r])} w(T). 
  \label{eq:BoundCumulants1}
\end{equation}
  By possibly adding edges of weight $0$, we may assume that $\WDep[v_1,\ldots,v_r]$
  is always the complete graph $K_r$ so that $\ST(\WDep[v_1,\ldots,v_r]) \simeq \ST(K_r)$ as sets.
  The weight of a tree $T$ however depends on $v_1,\cdots,v_r$, namely
  \[w(T)=\prod_{\{i,j\} \in T} w_\WDep(\{v_i,v_j\}),\]
  where $w_\WDep(\{v_i,v_j\})$ is the weight of the edge $\{v_i,v_j\}$ in $\WDep$ 
  (or $1$ if $v_i=v_j$). \medskip
  
  With this viewpoint, we can exchange the order of summation in \eqref{eq:BoundCumulants1}.
  We claim that the contribution of a fixed tree $T \in \ST(K_r)$ can then be bounded as follows:
  \begin{equation}
    \sum_{v_1,\ldots,v_r} w(T) = \sum_{v_1,\ldots,v_r} \prod_{\{i,j\} \in T} w_\WDep(\{v_i,v_j\}) \leq N D^{r-1}.
    \label{eq:ContributionT}
  \end{equation}
  Let us prove this claim by induction on $r$. The case $r=1$ is trivial.
  Up to renaming the vertices of $T$, we may assume that $r$ is a leaf of $T$ so that $T$
  is obtained from a spanning tree $\widetilde{T}$ of $K_{r-1}$ by adding an edge $\{i_0,r\}$ for some $i_0<r$.
  Then
  \begin{align*}
  &\sum_{v_1,\ldots,v_r} \prod_{\{i,j\} \in T} w_\WDep(\{v_i,v_j\})\\
  &= \sum_{v_1,\ldots,v_{r-1}} \left( \prod_{\{i,j\} \in \widetilde{T}} w_\WDep(\{v_i,v_j\}) 
  \right) \left[\sum_{v_r \in V}  w_\WDep(\{v_{i_0},v_r\}) \right].
  \end{align*}
  The expression in square brackets is by definition smaller than $D$ for all $v_{i_0}$
  (the sum for $v_r \ne v_{i_0}$ is smaller than $D-1$ and the term for $v_r = v_{i_0}$ is 1).
  By induction hypothesis, the sum of the parenthesis is smaller than $N D^{r-2}$.
  This concludes the proof of \eqref{eq:ContributionT} by induction. The lemma now follows immediately, since the number of spanning trees of $K_r$ is well known
  to be $r^{r-2}$.\qed
\end{proof}

\begin{remark}
  A classical dependency graph with a uniform bound $A$ on all variables $A_v$ can be seen as
  a $C$-uniform weighted dependency graph for $C=2A$ (all edges have weight 1);
  see \cite[Section 9.4]{FMN16}. 
  In this case, Proposition~\ref{Prop:BoundCumulantUWDG} reduces to Theorem~\ref{thm:boundcumulant}.
  The proof of Proposition~\ref{Prop:BoundCumulantUWDG} given here is a simple adaptation of the second part
  of the proof of Theorem~\ref{thm:boundcumulant} (see \cite[Chapter 9]{FMN16}) to the weighted context.
  The first, and probably the hardest part of the proof of Theorem~\ref{thm:boundcumulant}
  consisted in showing that a classical dependency graph is indeed a $C$-uniform weighted dependency graph.
\end{remark}

\begin{remark}
  In the case where the set $V$ is a subset of $\Z^d$ and the weight function only depends on the distance,
  the notion of uniform weighted dependency graph coincides with the notion of {\em strong cluster properties},
  proposed by Duneau, Iagolnitzer and Souillard in \cite{Duneau1}.
  These authors also observed that this implies uniform bounds on cumulants when $D$ is bounded,
  see \cite[Eq. (10)]{Duneau1}.
\end{remark}

\begin{remark}
  A weaker notion of weighted dependency graph, where the bound on cumulant is not uniform on $r$,
  was recently introduced in \cite{WDG}.
  This weaker notion only enables to prove central limit theorem, 
  without normality zone or speed of convergence results.
  However, it seems to have a larger range of applications.
\end{remark}

\subsection{Magnetization of the Ising model}
\label{subsec:Ising}
We consider here the nearest-neighbour Ising model on a finite subset $\Lambda$ of $\Z^d$ with a quadratic potential,
\emph{i.e.}~for a {\em spin configuration} $\sigma$ in $\{-1,+1\}^{\Lambda}$, its energy is given by
\[\HHH^\Lambda_{\beta,h}(w)=  - \beta\!\!\! \sum_{\substack{ i,j \in \Lambda \\ \{i,j\} \in E_{\Z^d}}} \!\!\! \sigma_i\sigma_j\,\,\,-h \sum_{i \in \Lambda} \sigma_i,\]
where $E_{\Z^d}$ is the set of edges of the lattice $\Z^d$ and $h$ and $\beta$ are two real parameters with $\beta > 0$.
The probability $\mu_{\beta,h,\Lambda}[\sigma]$ of taking a configuration $\sigma$ is then proportional to $\exp(- \HHH^\Lambda_{\beta,h}(\sigma))$.\medskip

We now want to make $\Lambda$ grow to the whole lattice $\Z^d$.
It is well known that for $h \ne 0$ or $\beta$ smaller than a critical value $\beta_c(d)$ (thus, at high temperature),
there is a unique limiting probability measure $\mu_{\beta,h}$ on spin configurations on $\Z^d$, see {\em e.g.}~\cite[Theorem 3.41]{Velenik}.
In the following, we take parameters $(\beta,h)$ in this domain of uniqueness 
and consider a random spin configuration $\sigma$, whose law is $\mu_{\beta,h}$.
\bigskip

In \cite{Duneau2}, Duneau, Iagolnitzer and Souillard proved what they call the {\em strong cluster properties}
for spins in the Ising model for $h \ne 0$ or sufficiently small $\beta$. Their result is actually more general (it holds for other models than the Ising model) but for simplicity, we stick to the Ising model here.
Reformulated with the terminology of the present article, we have:
\begin{theorem}[Duneau, Iagolnitzer and Souillard, 1974]
  Fix the dimension $d\geq 1$ and $h \ne 0,\,\beta>0$.
  \begin{enumerate}
  	\item There exist $C=C(d,\beta,h)$ and $\eps=\eps(d,\beta,h)<1$
      such that under the probability measure $\mu_{\beta,h}$, 
  the family $\{\sigma_i,\, i \in \Z^d\}$ has a $C$-uniform weighted dependency graph $\WDep$,
  where the weight of the edge $\{i,j\}$ in $\WDep$ is $\eps^{\|i-j\|_1}$.
  \item The same holds for $h=0$ and $\beta$ is sufficiently small (\emph{i.e.}~$\beta<\beta_1(d)$,
    for some $\beta_1(d)$ depending on the dimension; 
    this is sometimes refered to as the very high temperature regime).
  \end{enumerate}
  \label{thm:Cumu_Ising}
\end{theorem}
Note that the maximal weighted degree of this graph is a constant $C' < \infty$.
\bigskip

We now consider the magnetization in a finite box $\Delta$ defined
as $M_{\Delta}= \sum_{i \in \Delta} \sigma_i$.
We see $M_\Delta$ as a sequence of random variables (indexed by the countably many finite subsets of $\Z^d$).
Restricting the uniform weighted dependency graph above to $\{\sigma_i,\, i \in \Delta\}$,
each $M_\Delta$ is the sum of random variables with a $C$-uniform weighted dependency graph
and maximal weighted degree at most $C'$.
Therefore, using Proposition~\ref{Prop:BoundCumulantUWDG}, we know that
$M_{|\Delta|}$ admits uniform bounds on cumulants with parameters $(\tfrac{C'}{2},|\Delta|,C)$. 
Moreover, since all spins are positively correlated by the FKG inequality (see \cite[Section 3.6]{Velenik}),
we have, using translation invariance
\[\Var(M_\Delta) \ge \sum_{i \in \Delta} \Var(\sigma_i) = \Var(\sigma_0) |\Delta|.\]
Note that $\Var(\sigma_0)$ is independent of $\Delta$.
With the notation of Section~\ref{sec:dependency},
this inequality ensures that $\widetilde{\sigma}_\Delta$ is bounded from below by a constant.
Applying Corollary~\ref{cor:Cumulants_Kol}, we get:
\begin{proposition}
  Fix the dimension $d \geq 1$ and parameters $h \ne 0,\,\beta>0$.
  The exists a constant $K=K(d,\beta,h)$ such that, for all subsets $\Delta$ of $\Z^d$, we have under $\mu_{\beta,h}$
  \[\dkol\left( \frac{M_\Delta-\esper[M_\Delta]}{\sqrt{\Var(M_\Delta)}}\,,\, \gauss\right) \le \frac{K}{\sqrt{|\Delta|}}.\]
  The same holds for $h=0$ and $\beta$ sufficiently small (very high temperature).
  \label{prop:Kol_Ising}
\end{proposition}
\begin{remark}
  In this remark, we discuss mod-Gaussian convergence in this setting.
  Consider a sequence $\Delta_n$ of subsets of $\Z^d$,
  tending to $\Z^d$ in the Van Hove sense
  (\emph{i.e.}~the sequence is increasing with union $\Z^d$,
  and the size of the boundary of $\Delta_n$ is asymptotic negligible,
  compared to the size of $\Delta_n$ itself).
Then
it is known from \cite[Lemma 5.7.1]{Ellis} that 
\[\lim_{n \to \infty} \tfrac{1}{|\Delta|} \Var(M_\Delta) = \sum_{k \in \Z^d} \Cov(\sigma_0,\sigma_k),\]
and the right-hand side has a finite value $\widetilde{\sigma}^2(\beta,h)$
for parameters $(\beta,h)$ for which Theorem~\ref{thm:Cumu_Ising} applies.
Similarly, we have
\[\lim_{n \to \infty} \tfrac{1}{|\Delta|} \,\kappa^{(3)}(M_\Delta) = \sum_{k,l \in \Z^d} \kappa(\sigma_0,\sigma_k,\sigma_l) < \infty.\]
We call $\rho(\beta,h)$ this quantity, and denote $L=\tfrac{\rho(\beta,h)}{\widetilde{\sigma}^3(\beta,h)}$. Let us then consider the rescaled variables
$$X_n=\frac{M_{\Delta_n}-\esper[M_{\Delta_n}]}{(\Var(M_{\Delta_n}))^{1/3}}.$$ 
From \cite[Section 5]{FMN16} (with $\alpha_n=\Var(M_\Delta)$ and $\beta_n=1$),
we know that $X_n$ converges in the complex mod-Gaussian sense
with parameters $t_n=(\Var(M_{\Delta_n}))^{1/3}$ and limiting function $\psi(z)=\exp(\frac{L z^3}{6})$.
This mod-Gaussian convergence takes place on the whole complex plane.
Using the results of \cite{FMN16}, this implies a normality zone for $(M_\Delta-\esper[M_\Delta])/\sqrt{\var(M_\Delta)}$ of size $o(|\Delta|^{1/6})$, see Proposition 4.4.1 in \emph{loc. cit.}; and moderate deviation estimates at the edge of this normality zone, see Proposition 5.2.1.
\end{remark}

\begin{remark}
  For $h=0$ and  $\beta>\beta_c(d)$ (low temperature regime),
  there is no weighted dependency graph as above.
  Indeed, this would imply the analycity of the partition function
  in terms of the magnetic field $h$, and the latter
  is known not to be analytic at $h=0$ for $\beta>\beta_c(d)$;
  see \cite[Chapter 6, §5]{MM91} for details.
\end{remark}

\subsection{Functionals of Markov chains}
In this section, we consider a discrete time
Markov chain $(M_t)_{t \ge 0}$ on a finite state space $\spa$, which is ergodic (irreducible and aperiodic) with invariant measure $\pi$. Its transition matrix is denoted $P$. To simplify the discussion, we shall also assume that the Markov chain is stationary, that is to say that the initial measure (\emph{i.e.}~the law of $M_0$) is $\pi$; most results have easy corollaries for any initial measure, using the fast mixing of such Markov chains.
 \medskip

Let us consider a sequence $(f_t)_{t \ge 0}$ of functions on $\spa$  that is uniformly bounded by a constant $K$.
We set $Y_t=f_t(M_t)$. We will show that $\{Y_t\}_{t \in \N}$ admits a uniform weighted dependency graph.
The proof roughly follows the one of \cite[Section 10]{WDG}, where it was proved
that it has a (non-uniform) weighted dependency graph, taking extra care of the dependence in the
order $r$ of the cumulant in the bounds. 
Instead of working directly with classical (joint) cumulants,
we start by giving a bound for the so-called {\em Boolean cumulants}. 
Classical cumulants are then expressed in terms of Boolean cumulants
thanks to a formula of Saulis and Statulevi\v{c}ius \cite[Lemma 1.1]{LivreOrange:Cumulants};
see also a recent article of
Arizmendi, Hasebe, Lehner and Vargas \cite{ArizmendiHasebeLehnerVargas2014}
(we warn the reader that, in \cite{LivreOrange:Cumulants},
Boolean cumulants are called centered moments).\bigskip

Let $Z_1,\ldots,Z_r$ be random variables with finite moments defined on the same probability space.
By definition, their Boolean (joint) cumulant is
\begin{align*}
  &B^{(r)}(Z_1,\ldots,Z_r) \\
  &= \sum_{l=0}^{r-1} (-1)^{l} \!\!\!\!\sum_{1\le d_1 < \cdots <d_l \le r-1} \!\!\!\!\!\!\!
\esper[Z_1 \cdots Z_{d_1}] \, \esper[Z_{d_1+1} \cdots Z_{d_2}] \, \cdots \, \esper[Z_{d_l+1} \cdots Z_r].
\end{align*}
While not at first sight, this definition is quite similar to the definition of classical (joint) cumulants,
replacing the lattice of all set partitions by the lattice of interval set partitions;
see \cite[Section 2]{ArizmendiHasebeLehnerVargas2014} for details.
Note however that, unlike classical cumulants, Boolean cumulants are not symmetric functionals.

\begin{proposition}\label{prop:Bound_Boolean}
  Let $r\geq 1$. With the above notation, there exists a constant $\theta_{P}$ depending on $P$
  with the following property.
  For any integers $t_1 \le t_2 \le \cdots \le t_r$, we have
  \[ \left|B^{(r)}(Y_{t_1},\ldots,Y_{t_r}) \right| \le M^{\frac{r-1}{2}} \,K^r\, (\theta_P)^{t_r-t_1},\]
  where $M = |\spa|$.
\end{proposition}
The proof of this bound relies on arguments due to Diaconis, Stroock and Fill, see \cite{DS91,Fil91}.
We also refer to \cite[Section 4.1]{LivreOrange:Cumulants} for an alternate approach.
Given an ergodic transition matrix $P$ on $\spa$ with invariant measure $\pi$, we denote $\widetilde{P}$ the time reversal of $P$, which is the stochastic matrix defined by the equation
$$\widetilde{P}(x,y) = \frac{\pi(y)\,P(y,x)}{\pi(x)}.$$
This new transition matrix is again ergodic, with stationary measure $\pi$. The multiplicative reversiblization of $P$ is the matrix $M(P)=P\widetilde{P}$. It is a stochastic matrix, which is ergodic with stationary measure $\pi$, and with all its eigenvalues that are real and belong to $[0,1]$. Indeed, if $D$ is the diagonal matrix $D = \mathrm{diag}(\pi)$, then $\widetilde{P} = D^{-1}P^tD$, and
\begin{align*}
\mathrm{Spec}(M(P)) &= \mathrm{Spec}\left(D^{1/2}P\widetilde{P}D^{-1/2}\right) \\
&= \mathrm{Spec}\left((D^{1/2}PD^{-1/2})(D^{-1/2}P^t D^{1/2})\right) \\
&= \mathrm{Spec}\left((D^{1/2}PD^{-1/2})(D^{1/2}PD^{-1/2})^t\right).
\end{align*}
Thus, $M(P)$ has the spectrum of a symmetric positive matrix, so it belongs to $\R_+$, and in fact to $[0,1]$ since $M(P)$ is also stochastic. We denote 
\begin{equation}
  (\theta_P)^2 = \max\{|z|\,|\,z \text{ eigenvalue of }M(P),\,z \neq 1\}.
  \label{eq:def_thetaP}
\end{equation}
Notice that if $P$ is reversible, then $\widetilde{P}=P$ and $M(P)=P^2$, so in this case 
$$ \theta_P=\max\{|z|\,|\,z \text{ eigenvalue of }P,\,z \neq 1\}. $$ 
In general, one can think of $\theta_P$ as the analogue of the second largest eigenvalue for non-reversible transition matrices. The following result estimates the rate of convergence of the Markov chain associated to $P$ in terms of $\theta$:
\begin{theorem}[Fill, 1991]\label{thm:fill}
For any $x \in \spa$,
\begin{align*}
\sum_{y \in \spa} |P^t(x,y)-\pi(y)| &\leq \frac{(\theta_P)^t}{\sqrt{\pi(x)}}; \\
\sum_{y \in \spa} \frac{|P^t(x,y)-\pi(y)|}{\sqrt{\pi(y)}} &\leq \sqrt{M}\,\frac{(\theta_P)^t}{\sqrt{\pi(x)}}
\end{align*}
where $M = |\spa|$.
\end{theorem}
\begin{proof}
For completeness, we reproduce here the discussion of \cite[Section 2]{Fil91}, which relies on the following identity due to Mihail.
If $f$ is a function on $\spa$, we denote $\var(f)$ its variance under the stationary probability measure $\pi$. We also introduce the Dirichlet form
$$\mathscr{E}(f,g) = \frac{1}{2}\sum_{x,y\in \spa} (f(x)-f(y))(g(x)-g(y))\,\pi(x)\,M(P)(x,y).$$ 
Then, for any function $f$,
$$\var(f) = \var(\widetilde{P}f) + \mathscr{E}(f,f).$$
Indeed, one can assume w.l.o.g.~that $\pi(f)=\sum_{x\in \spa}\pi(x)f(x)=0$. If $\scal{f}{g}_{\pi} = \sum_{x\in\spa} \pi(x)\,f(x)\,g(x)$, then
\begin{align*}
\mathscr{E}(f,f) &= \scal{f}{(\id - M(P))f}_\pi = \scal{f}{f}_\pi - \scal{f}{P\widetilde{P}f}_\pi \\
&=\scal{f}{f}_\pi - \scal{\widetilde{P}f}{\widetilde{P}f}_\pi = \var(f)-\var(\widetilde{P}f)
\end{align*}
since $\widetilde{P}$ is the adjoint of $P$ for the action on the left of functions and with respect to the scalar product $\scal{\cdot}{\cdot}_\pi$. Consider now a Markov chain $(X_t)_{t \in \N}$ on $\spa$ with arbitrary initial distribution $\pi_0$, and denote $\pi_t = \pi_0P^t$ the distribution at time $t$. We introduce the quantity 
$$(\chi_t)^2 = \sum_{y \in \spa} \frac{(\pi_t(y)-\pi(y))^2}{\pi(y)}.$$
This is the variance of $f_t=\frac{\pi_t}{\pi}$ with respect to the probability measure $\pi$. By Mihail's identity,
$$(\chi_{t+1})^2 = \var(f_{t+1})=\var(\widetilde{P}f_t) =\var(f_t) - \mathscr{E}(f_t,f_t) = (\chi_t)^2 - \mathscr{E}(f_t,f_t). $$
By the minimax characterization of eigenvalues of symmetric positive matrices, 
$$(\theta_P)^2 = 1-\inf\left\{ \frac{\mathscr{E}(f,f)}{\var(f)},\,f\text{ non-constant}\right\}.$$
Therefore, $(\chi_{t+1})^2 \leq (\theta_P)^2\,(\chi_t)^2$, and $(\chi_t)^2 \leq (\theta_P)^{2t}\,(\chi_0)^2$ by induction on $t$. On the other hand, the Cauchy-Schwarz inequality yields
$$\sum_{y \in \spa} |\pi_t(y)-\pi(y)| \leq \sqrt{\sum_{y\in \spa} \pi(y)}\,\sqrt{\sum_{y \in \spa} \frac{(\pi_t(y)-\pi(y))^2}{\pi(y)}} = \chi_t.$$
If we choose $\pi_0=\delta_x$, we finally obtain:
\begin{equation*}
\sum_{y \in \spa} |P^t(x,y)-\pi(y)| \leq (\theta_P)^t\,\chi_0 = (\theta_P)^t \sqrt{\frac{1}{\pi(x)}-1} \leq \frac{(\theta_P)^t}{\sqrt{\pi(x)}}.
\end{equation*}
Similarly, 
\begin{equation*}\sum_{y \in \spa} \frac{|P^t(x,y)-\pi(y)| }{\sqrt{\pi(y)}} \leq \sqrt{M}\,\chi_t \leq \sqrt{M}\,(\theta_P)^t\,\chi_0 \leq \sqrt{M}\,\frac{(\theta_P)^t}{\sqrt{\pi(x)}}.\qquad\qed
\end{equation*}
\end{proof}

\begin{proof}[Proposition \ref{prop:Bound_Boolean}]
If $f:\spa \to \R$, denote $D_f=\mathrm{diag}(f(x),x\in \spa)$. Then, the Boolean cumulant has the following matrix expression:
$$B^{(r)}(Y_{t_1},\ldots,Y_{t_r}) = \pi\,D_{f_{t_1}}\,(P^{t_2-t_1}-\textbf{1}\,\pi)\,\cdots\, D_{f_{t_{r-1}}}\,(P^{t_r-t_{r-1}}-\textbf{1}\,\pi)\,D_{f_{t_r}}\,\textbf{1},$$
where $\textbf{1}$ is the column vector with all its entries equal to $1$; see \cite[Lemma 10.1]{WDG}. If we expand this expression as a sum, and denote $Q_t = P^t - \textbf{1}\,\pi$ and $\delta_i=t_{i+1}-t_i$, then 
\begin{align*}
  &B^{(r)}(Y_{t_1},\ldots,Y_{t_r}) \\
  & = \sum_{x_1,\ldots,x_r}\pi(x_1)\,f_{t_1}(x_1)\,Q_{t_2-t_1}(x_1,x_2)\,f_{t_2}(x_2)\,\cdots\, Q_{t_r-t_{r-1}}(x_{r-1},x_r)\,f_{t_r}(x_r) \end{align*}
  and we obtain
  \begin{align*}
|B^{(r)}&(Y_{t_1},\ldots,Y_{t_r})|\leq K^r \sum_{x_1,\ldots,x_r}\pi(x_1)\,|Q_{\delta_1}(x_1,x_2)|\,\cdots\,|Q_{\delta_{r-1}}(x_{r-1},x_r)| \\
&\leq K^r\,(\theta_P)^{\delta_{r-1}} \sum_{x_1,\ldots,x_{r-1}}\pi(x_1)\,|Q_{\delta_1}(x_1,x_2)|\,\cdots\,\frac{|Q_{\delta_{r-2}}(x_{r-2},x_{r-1})|}{\sqrt{\pi(x_{r-1})}}\\
&\qquad\vdots \\
&\leq K^r\,M^{\frac{r-2}{2}}\,(\theta_P)^{t_r-t_1} \sum_{x_1}\sqrt{\pi(x_1)} \leq K^{r} M^{\frac{r-1}{2}}\,(\theta_P)^{t_r-t_{1}}.\qquad\qed
\end{align*}
\end{proof}

\begin{proposition}
  \label{prop:UWDG_Markov}
The family of random variables $\{Y_t\}_{t \in \N}$ admits a $K\sqrt{M}$-uniform weighted dependency graph,
where, for integers $s<t$, the weight between $Y_t$ and $Y_s$ is $2(\theta_P)^{t-s}$.
\end{proposition}
\begin{proof}
  A lemma of Saulis and Statulevi\v{c}ius \cite[Lemma 1.1]{LivreOrange:Cumulants} expresses usual cumulants
  in terms of Boolean cumulants:
 \begin{equation}
   \kappa^{(r)}(Y_{t_1},\ldots,Y_{t_r}) = \sum_{\pi \in \PPP([r])} (-1)^{|\pi|-1}
 N(\pi) \prod_{C \in \pi} B^{(|C|)}(Y_{t_j},\,  j \in C).
 \label{eq:Boolean2Classical}
 \end{equation}
 Here, the sum runs over set-partitions $\pi$ of $[r]:=\{1,\ldots,r\}$;
 $|\pi|$ is the number of blocks of a set-partition $\pi$, the product runs over blocks $C$ in $\pi$
 and $B^{(|C|)}(Y_{t_j},\, j \in C)$ is the Boolean cumulant of the subfamily $(Y_{t_j})$
 indexed by integers $j$ in the block $C$, with the times ordered in increasing order (recall that the Boolean are not symmetric functionals).
 Finally $N(\pi)$ is a combinatorial factor that can be computed as follows.
 For each block $C$ of $\pi$, denote $m_C$ and $M_C$ its smallest and biggest elements;
 then call $n_C$ the number of blocks $C' \ne C$
 such that $m_C$ is in the interval $[m_{C'};M_{C'}]$. 
 We finally define $N(\pi)=\prod_{C \in \pi; 1 \notin C} n_C$.
 In other terms, $N(\pi)$ counts the functions $g$ mapping 
 each block $C$ of $\pi$ (except the one containing $1$) to a block $g(C) \ne C$
 such that $m_C \in [m_{g(C)};M_{g(C)}]$.\medskip

 Let us make an observation.
 If $\pi$ is a partition and $k$ an integer such that each block
 of $\pi$ either contains only numbers smaller than or equal to $k$
 or only numbers bigger than $k$ ($\pi$ is then said to be {\em disconnected}),
 then no function $g$ as above exists (there is no possible image for the block $C$ containing $k+1$)
 and $N(\pi)=0$. On the other hand,
 for {\em connected} partitions $\pi$, we always have $N(\pi)>0$,
 so that the sum in \eqref{eq:Boolean2Classical} is in effect a sum over connected partitions.
 \medskip

 Eq.~\eqref{eq:Boolean2Classical} and Proposition~\ref{prop:Bound_Boolean} imply the bound
 \begin{equation*}
   \left|\kappa^{(r)}(Y_{t_1},\ldots,Y_{t_r}) \right| \leq \left(K\sqrt{M}\right)^r
   \sum_{\pi \in \PPP([r])} 
   N(\pi) \prod_{C \in \pi} (\theta_P)^{t_{M_C}-t_{m_C}}.
 \end{equation*}
 We would like to prove 
\begin{equation*}
   \left|\kappa^{(r)}(Y_{t_1},\ldots,Y_{t_r}) \right| \leq 2^{r-1}\,\left(K\sqrt{M}\right)^r
   \sum_{T \in \ST(K_r)} w(T),
\end{equation*}
where $w(T)=\prod_{ \{j,j'\} \in E_T, \, j<j'} (\theta_P)^{t_{j'}-t_j}$.
Therefore it is sufficient for us to find an injective mapping $\eta$ from pairs $(\pi,g)$ as above
to edge-bicolored spanning trees $\widetilde{T}$ such that
\begin{equation}
  w\big( \eta(\pi,g) \big) =  \prod_{C \in \pi} (\theta_P)^{t_{M_C}-t_{m_C}}; 
  \label{eq:eta_Conserve_Poids}
\end{equation}
here, by convention, the weight $w(\widetilde{T})$ of a colored tree 
is the weight $w(T)$ of its uncolored version.
In the following, we describe such a mapping, concluding the proof of the proposition.
\bigskip

Let $(\pi,g)$ be a pair of objects as above:
$\pi$ is a set-partition of $[r]$ and $g$ is function mapping 
each block $C$ of $\pi$ (except the one containing $1$) to a block $g(C) \ne C$  
such that $m_C \in [m_{g(C)};M_{g(C)}]$.
For each block $C$ of $\pi$, we consider the set 
$$S(C)=C \cup \{m_{C'},\, C' \in g^{-1}(C)\}.$$
Let us call $P_C$ the path with vertex-set $S(C)$,
where the vertices are in increasing order along the path.
We also color in blue (resp. in red) edges of this path whose extremity with smaller
label is in $C$ (resp. in $\{m_{C'},\, C' \in g^{-1}(C)\}$).
\medskip

As an example, take $\pi=\{C_1,\cdots,C_6\}$ with $C_1=\{1,5,10\}$,
$C_2=\{2,11\}$, $C_3=\{3,9\}$, $C_4=\{4,6,13\}$, $C_5=\{7,12\}$, $C_6=\{8\}$.
As function $g$, we take $g(C_2)=C_1$, $g(C_3)=C_1$, $g(C_4)=C_2$, $g(C_5)=C_1$
and $g(C_6)=C_4$.
In this case, we get $S(C_1)=\{1,2,3,5,7,10\}$, $S(C_2)=\{2,4,11\}$,
$S(C_4)=\{4,6,8,13\}$ and $S(C_i)=C_i$ for $i \in \{3,5,6\}$.
The associated bicolored paths are drawn in Fig.~\ref{fig:example_eta}.\bigskip

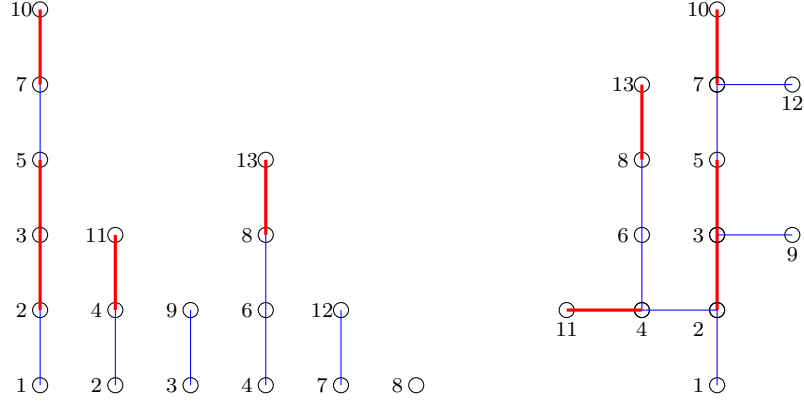
\begin{figure}[ht]
\begin{center}
 \begin{tikzpicture}
\draw (0,0) circle (1mm);
\draw (-0.25,0) node {$1$};
\draw (0,1) circle (1mm);
\draw (-.25,1) node {$2$};
\draw (0,2) circle (1mm);
\draw (-.25,2) node {$3$};
\draw (0,3) circle (1mm);
\draw (-.25,3) node {$5$};
\draw (0,4) circle (1mm);
\draw (-.25,4) node {$7$};
\draw (0,5) circle (1mm);
\draw (-.25,5) node {$10$};
\draw[blue] (0,0) -- (0,1);
\draw[red,very thick] (0,1) -- (0,2);
\draw[red,very thick] (0,2) -- (0,3);
\draw[blue]  (0,3) -- (0,4);
\draw[red,very thick] (0,4) -- (0,5);
\begin{scope}[xshift=1cm]
\draw (0,0) circle (1mm);
\draw (-0.25,0) node {$2$};
\draw (0,1) circle (1mm);
\draw (-.25,1) node {$4$};
\draw (0,2) circle (1mm);
\draw (-.25,2) node {$11$};
\draw[blue] (0,0) -- (0,1);
\draw[red,very thick] (0,1) -- (0,2);
\end{scope}
\begin{scope}[xshift=2cm]
\draw (0,0) circle (1mm);
\draw (-0.25,0) node {$3$};
\draw (0,1) circle (1mm);
\draw (-.25,1) node {$9$};
\draw[blue] (0,0) -- (0,1);
\end{scope}
\begin{scope}[xshift=3cm]
\draw (0,0) circle (1mm);
\draw (-0.25,0) node {$4$};
\draw (0,1) circle (1mm);
\draw (-.25,1) node {$6$};
\draw (0,2) circle (1mm);
\draw (-.25,2) node {$8$};
\draw (0,3) circle (1mm);
\draw (-.25,3) node {$13$};
\draw[blue] (0,0) -- (0,1);
\draw[blue] (0,1) -- (0,2);
\draw[red,very thick] (0,2) -- (0,3);
\end{scope}
\begin{scope}[xshift=4cm]
\draw (0,0) circle (1mm);
\draw (-0.25,0) node {$7$};
\draw (0,1) circle (1mm);
\draw (-.25,1) node {$12$};
\draw[blue] (0,0) -- (0,1);
\end{scope}
\begin{scope}[xshift=5cm]
\draw (0,0) circle (1mm);
\draw (-0.25,0) node {$8$};
\end{scope}
\begin{scope}[xshift=9cm]
\draw (0,0) circle (1mm);
\draw (-0.25,0) node {$1$};
\draw (0,1) circle (1mm);
\draw (-.25,.75) node {$2$};
\draw (0,2) circle (1mm);
\draw (-.25,2) node {$3$};
\draw (0,3) circle (1mm);
\draw (-.25,3) node {$5$};
\draw (0,4) circle (1mm);
\draw (-.25,4) node {$7$};
\draw (0,5) circle (1mm);
\draw (-.25,5) node {$10$};
\draw[blue] (0,0) -- (0,1);
\draw[red,very thick] (0,1) -- (0,2);
\draw[red,very thick] (0,2) -- (0,3);
\draw[blue]  (0,3) -- (0,4);
\draw[red,very thick] (0,4) -- (0,5);
\begin{scope}[xshift=-1cm,yshift=1cm]
\draw (0,0) circle (1mm);
\draw (0,1) circle (1mm);
\draw (-.25,1) node {$6$};
\draw (0,2) circle (1mm);
\draw (-.25,2) node {$8$};
\draw (0,3) circle (1mm);
\draw (-.25,3) node {$13$};
\draw[blue] (0,0) -- (0,1);
\draw[blue] (0,1) -- (0,2);
\draw[red,very thick] (0,2) -- (0,3);
\end{scope}
\begin{scope}[yshift=1cm]
\draw (0,0) circle (1mm);
\draw (-1,0) circle (1mm);
\draw (-1,-.25) node {$4$};
\draw (-2,0) circle (1mm);
\draw (-2,-.25) node {$11$};
\draw[blue] (0,0) -- (-1,0);
\draw[red,very thick] (-1,0) -- (-2,0);
\end{scope}
\begin{scope}[yshift=2cm]
\draw (0,0) circle (1mm);
\draw (1,0) circle (1mm);
\draw (1,-.25) node {$9$};
\draw[blue] (0,0) -- (1,0);
\end{scope}
\begin{scope}[yshift=4cm]
\draw (0,0) circle (1mm);
\draw (1,0) circle (1mm);
\draw (1,-.25) node {$12$};
\draw[blue] (0,0) -- (1,0);
\end{scope}
\end{scope}
  \end{tikzpicture}
  \caption{Illustration of the construction $\eta$ in the proof of Proposition~\ref{prop:UWDG_Markov}:
  the path $P_{C_i}$ and their gluing $\eta(\pi,g)$. For readers of a black-and-white printed version,
  red edges are thicker.}
  \label{fig:example_eta}
  \end{center}
\end{figure}

As in Fig.~\ref{fig:example_eta},
we then take the union of the paths $P_{C_i}$, identifying vertices with the same label
(the minimum $m_C \ne 1$ of a block $C$ appears in the path $S(C)$ and in the path $S(g(C))$).
Doing that, we get an edge-bicolored graph that we call $\eta(\pi,g)$. Let us first check that $\eta(\pi,g)$ is a tree.
To this purpose, we order the blocks $C_1,\ldots, C_{|\pi|}$ of $\pi$
in increasing order of their minima (as done in the example).
Observe that this implies that $g(C_i)=C_j$ for some $j<i$.
We will prove by induction that, for each $i \le |\pi|$, $P_{C_1} \cup \cdots \cup P_{C_i}$ is a tree.
The case $i=1$ is trivial. For $i>1$, the graph $P_{C_1} \cup \cdots \cup P_{C_i}$
is obtained by gluing the path $P_{C_i}$ on the graph $P_{C_1} \cup \cdots \cup P_{C_{i-1}}$,
identifying $m_{C_i}$ which appears in both.
Since $P_{C_1} \cup \cdots \cup P_{C_{i-1}}$ is a tree by induction hypothesis,
the resulting graph $P_{C_1} \cup \cdots \cup P_{C_i}$ is a tree as well, concluding the induction.
Thus $\eta(\pi,g)= P_{C_1} \cup \cdots \cup P_{C_{|\pi|}}$ is a tree.\bigskip

The equality~\eqref{eq:eta_Conserve_Poids} is easy:
since the edge set of $\eta(\pi,g)$ is the union of the edge sets of the $P_{C_i}$, we have
\[w\big( \eta(\pi,g) \big) = \prod_{i=1}^{|\pi|} w(P_{C_i}) = \prod_{i=1}^{|\pi|} (\theta_P)^{t_{M_{C_i}}-t_{m_{C_i}}}.\]
We finally need to prove that $\eta$ is injective, {\em i.e.}~that we can recover $(\pi,g)$ from $\eta(\pi,g)$.
We start by the following easy observation: in $\eta(\pi,g)$,
vertices with a red incident edge going to a vertex with bigger label
are exactly the vertices with a label which is the minimum $m_C \ne 1$ of a block $C$ of $\pi$.
By construction, such vertices have at most three incident edges, which are as follow:
\begin{enumerate}[label=(E\arabic*)]
  \item \label{item:edge2} as said above, a first one is red and goes to a vertex to bigger label;
  \item \label{item:edge3} a second one is either blue or red and goes to a vertex of lower label.
  \item \label{item:edge1} possibly, a last one is blue and goes to a vertex to bigger label (there is such an edge when $m_C$ is not alone in its block);
\end{enumerate}
Indeed, in the construction, edges \ref{item:edge2} and \ref{item:edge3} comes from $S_{g(C)}$
while edge \ref{item:edge1} comes from $S_C$.
We split the vertex $m_C$ into two, keeping edges \ref{item:edge2} and \ref{item:edge3}
in the same component.
Doing that for the $|\pi|-1$ vertices $m_C \ne 1$,
we inverse the gluing step of the construction of $\eta(\pi,g)$.
It is then straightforward to recover $(\pi,g)$.\qed
\end{proof}

\begin{theorem}\label{thm:kolmogorovmarkov}
  Let $(X_t)_{t \in \N}$ be an ergodic Markov chain on a finite state space $\spa$ of size $M$, and $\theta_P<1$ the constant associated by \eqref{eq:def_thetaP} with the transition matrix $P$.
  We consider a sum $S_n= \sum_{t = 1}^n f_t(X_t)$ with $\|f_t\|_{\infty} \le K$. Then, for any $r \ge 1$,
  \begin{equation}
    \left|\ka^{(r)}(S_n)\right| \le n \, r^{r-2}  \left(2\,\frac{1+\theta_P}{1-\theta_P}\right)^{r-1} \left(K\sqrt{M}\right)^r.
    \label{eq:BorneCumulantMarkov}
  \end{equation}
    As a consequence:
    \begin{enumerate}
     	\item When $\frac{\Var(S_n)}{n^{2/3}} \to +\infty$,
  we have
  \[\dkol\left( \frac{S_n-\esper[S_n]}{\sqrt{\Var(S_n)}}\,,\, \gauss \right)
  \le 76.36\,\left(\frac{K\sqrt{M}}{\sqrt{\frac{\var(S_n)}{n}}}\right)^{\!3}\,\left(\frac{1+\theta_P}{1-\theta_P}\right)^2\,\frac{1}{\sqrt{n}},\]
  and in particular, $\frac{S_n-\esper[S_n]}{\sqrt{\var(S_n)}}$ converges in law to a standard Gaussian.

  \item  This convergence in law happens as soon as $\frac{\var(S_n)}{n^{\eps}} \to \infty$ for some $\eps >0$.
     \end{enumerate} 
\end{theorem}

\begin{proof}
  Combining Propositions~\ref{prop:UWDG_Markov} and \ref{Prop:BoundCumulantUWDG},
  the sum $S_n$ admits uniform bounds on cumulants with parameters 
$$D_n=\left(1+2\sum_{t=1}^\infty (\theta_P)^t\right) = \frac{1+\theta_P}{1-\theta_P},$$
$N_n=n$ and $A = K\sqrt{M}$. Observe that $D_n$ is here independent of $n$.
We can apply Corollary~\ref{cor:Cumulants_Kol} to get the first part of the theorem.
The second follows from Theorem~\ref{thm:CLT_Janson}.\qed
\end{proof}

\begin{remark}
  A bound similar to Eq. \eqref{eq:BorneCumulantMarkov} is given in 
  \cite[Theorem 4.19]{LivreOrange:Cumulants}.
  We believe however that the proof given there is not correct.
  Indeed, the proof associates with each partition $\pi$ such that $N(\pi) \ne 0$
  a sequence of number $q_j$;
  the authors then claim that ``obviously $q_j \le q_{j+1}$'' (p.~93, after eq. (4.53)).
  This is unfortunately not the case as can be seen on the example
  of partitions given p.~81 in \emph{loc.~cit.}: for this partition $q_3=3$, while $q_4=2$.
  As a consequence of this mistake,
  the authors forget many partitions $\pi$ when expressing classical cumulants
  in terms of Boolean cumulants (since they encode partitions with only non-decreasing sequences $q_i$),
  which make the resulting bound on classical cumulants too sharp. We have not found a simpler way around this error than the use of uniform weighted dependency graphs presented here.
  Note nevertheless that our proof still uses several ingredients from \cite{LivreOrange:Cumulants}:
  the use of Boolean cumulants and the relation between Boolean and classical cumulants.
\end{remark}

\begin{remark}
If the functions $f_t$ are indicators $f_t(x)=1_{x=s_t}$, then one can remove the size $(\sqrt{M})^3$ in the bound on the Kolmogorov distance. Indeed, in this case, we have
$$B^{(r)}(Y_{t_1},\ldots,Y_{t_r}) = \pi(s_1)\,Q_{t_2-t_1}(s_1,s_2)Q_{t_3-t_2}(s_2,s_3)\cdots Q_{t_{r}-t_{r-1}}(s_{r-1},s_r).$$
On the other hand, the individual terms of the matrix $Q_{t}(x,y)$ can be bounded by
$$|Q_{t}(x,y)| \leq \sqrt{\frac{\pi(y)}{\pi(x)}}\,(\theta_P)^t,$$
by adapting the proof of Theorem \ref{thm:fill}. Therefore,
\begin{align*}
\left|B^{(r)}(Y_{t_1},\ldots,Y_{t_r})\right| &\leq (\theta_P)^{t_r-t_1}\, \pi(s_1)\,\sqrt{\frac{\pi(s_2)}{\pi(s_1)}} \cdots \sqrt{\frac{\pi(s_r)}{\pi(s_{r-1})}} \\
&\leq (\theta_P)^{t_r-t_1} \sqrt{\pi(s_1)\pi(s_r)} \leq (\theta_P)^{t_r-t_1}.
\end{align*}
Thus, in this case, one has the bound of Theorem \ref{thm:kolmogorovmarkov} without the factor $(\sqrt{M})^3$.
\end{remark}

\subsection{The case of linear functionals of the empirical measure} 
As a particular case of Theorem \ref{thm:kolmogorovmarkov},
one recovers the central limit theorem for linear functionals of empirical measures of Markov chains, that are random variables
$$Y_n = \frac{S_n-\esper[S_n]}{\sqrt{n}} =  \frac{1}{\sqrt{n}} \sum_{t=1}^n (f(X_t) - \pi(f))$$
with $f : \spa \to \R$ fixed function (independent of the time $t$). Thus, assuming for instance $\lim_{n \to \infty} \var(Y_n) = \varSigma^2(f)>0$, we have
\begin{equation}
\dkol\left(\frac{S_n - \esper[S_n]}{\sqrt{\var(S_n)}}\,,\,\gauss\right) \leq 77\,\left(\frac{\|f\|_\infty\,\sqrt{M}}{\varSigma(f)}\right)^{\!3}\,\left(\frac{1+\theta_P}{1-\theta_P}\right)^2\,\frac{1}{\sqrt{n}}\label{eq:boundlinearfunctional}
\end{equation}
for $n$ large enough. We refer to \cite{Cog72,KV86,Jon04,Hagg05} and the references therein for the general background of this Markovian CLT, and to \cite{Bol80,Lez96,Mann96} for estimates of the Kolmogorov distance. It seems that we recover some results of \cite{Mann96} (see \cite[\S 2.1.3]{SC97}), but we were not able to find and read this paper. In this last paragraph, we discuss the problem of the variance $\Var(Y_n)$,
giving sufficient conditions, which are simple to check on examples and ensure $\varSigma^2(f)>0$.
We also remark that, provided that $\varSigma^2(f)>0$,
we can also prove complex mod-Gaussian convergence, which implies a zone of normality result and moderate deviation estimates by \cite{FMN16}.\bigskip

Denote $g = f - \pi(f)$, which has expectation $0$ under the stationary measure $\pi$.
By eventually replacing $f$ with $g$, we can assume that $f$ is centered.
The variance of $Y_n$ tends to
$$\varSigma^2(f) = \esper[(f(X_0))^2]+2\sum_{t=1}^\infty \esper[f(X_0)\,f(X_t)]<+\infty,$$
see \cite[Lemma 3.3]{Cog72}. If $\varSigma^2(f)>0$, then $\frac{\var(S_n)}{n^{2/3}} = n^{1/3}\,\var(Y_n) \to + \infty$ and Theorem \ref{thm:kolmogorovmarkov} applies. Unfortunately, one can easily construct non-trivial examples with \hbox{$\varSigma^2(f)=0$}. Thus, consider the Markov chain with $3$ states and transition matrix
$$P=\begin{pmatrix}
0 & 1 & 0 \\
1/2 & 0 & 1/2 \\
1 & 0 & 0
\end{pmatrix};$$
it admits for invariant measure $\pi(1)=\pi(2)=\frac{2}{5}$ and $\pi(3)=\frac{1}{5}$. Set $f(1)=1$, $f(2)=-1$ and $f(3)=0$; one has $\pi(f)=0$, and one computes 
$$\esper[f(X_0)f(X_k)] = \frac{1}{5} \left(\frac{2+\I}{(-1-\I)^k} + \frac{2-\I}{(-1+\I)^k}\right).$$
It follows that $\varSigma^2(f)=0$, although $f$ is non zero.\medskip

In the general case of an ergodic Markov chain, fix an element $a$ of the state space $\spa$, and denote $\tau_a\geq 1$ the first hitting time of $a$ by the Markov chain, which is almost surely finite and with expectation $1/\pi(a)$ when starting from $a$. Then, the asymptotic variance $\varSigma^2(f)$ can be rewritten as
$$\varSigma^2(f) = \pi(a)\,\,\esper_a\!\left[\left(\sum_{k=1}^{\tau_a} g(X_k)\right)^2\right],$$
see \cite[Chapter 4]{KS76}. Therefore, a general condition in order to obtain the bound of Eq. \eqref{eq:boundlinearfunctional} is:
\begin{proposition}
We have $\varSigma^2(f)>0$ if and only if there exists a cycle $(x_1,\ldots,x_n)$ in the graph of transitions of the Markov chain such that the sum $\sum_{i=1}^n g(x_i)$ of the values of $g=f-\pi(f)$ along this cycle is non-zero. In this case, the bound \eqref{eq:boundlinearfunctional} holds.
\end{proposition}
\noindent The proposition explains readily why the irreducible aperiodic Markov chain
\begin{center}
\begin{tikzpicture}[scale=2]
\draw (0,0) circle (1mm);
\draw (-0.25,0) node {$1$};
\draw (2,0) circle (1mm);
\draw (2.25,0) node {$2$};
\draw (1,1) circle (1mm);
\draw (1,1.3) node {$3$};
\draw[->,color = NavyBlue] plot [smooth] coordinates {(0.15,0.05) (1,0.2) (1.85,0.05)};
\draw (1,0.33) node {\textcolor{NavyBlue}{$1$}};
\draw[->,color = NavyBlue] plot [smooth] coordinates {(1.85,-0.05) (1,-0.2) (0.15,-0.05)};
\draw (1,-0.35) node {\textcolor{NavyBlue}{$1/2$}};
\draw[->,color = NavyBlue] (1.9,0.15) -- (1.1,0.9);
\draw (1.8,0.5) node {\textcolor{NavyBlue}{$1/2$}};
\draw[->,color = NavyBlue] (0.9,0.9) -- (0.1,0.15);
\draw (0.3,0.5) node {\textcolor{NavyBlue}{$1$}};
\end{tikzpicture}
\end{center}
previously studied gives asymptotic variance $0$ to the function $f(1)=1$, $f(2)=-1$ and $f(3)=0$: the minimal cycles of the chains are $(1,2)$ and $(1,2,3)$, and the sum of the values of $f$ along these cycles is always $0$.\bigskip

Another simple criterion to apply Theorem \ref{thm:kolmogorovmarkov} to linear functionals of the empirical measure is:
\begin{proposition}
Suppose that the ergodic Markov chain is reversible: $$\pi(x)\,P(x,y) = \pi(y)\,P(y,x)$$ for any $x,y$. Then, if $f$ is a non-constant function, $\varSigma^2(f)>0$ and the bound \eqref{eq:boundlinearfunctional} holds.
\end{proposition}
\begin{proof}
To say that $P$ is reversible is equivalent to the fact that $P$ is a symmetric operator of the Hilbert space $\leb^2(\frac{1}{\pi})$. In particular, $P$ has only real eigenvalues. Besides, the restriction of the operator $I+2\sum_{k=1}^\infty P^k$ to the space of functions $f$ with $\pi(f)=0$ is well defined, and it is an auto-adjoint operator on this space with eigenvalues $$\frac{1+\lambda_2}{1-\lambda_2},\ldots,\frac{1+\lambda_M}{1-\lambda_M},$$ 
where $\lambda_2 \geq \lambda_3 \geq \cdots \geq \lambda_M$ are the real eigenvalues of $P$ different from $1$. The quantities above are all positive, and larger than $\frac{1-\theta_P}{1+\theta_P}$ (this value being obtained if $\lambda_M = - \theta_P$). Therefore,
\begin{equation*}
\varSigma^2(f) =  \scal{f}{\left(I+2\sum_{k=1}^\infty P^k\right) f}_{\leb^2(\pi)} \geq \frac{1-\theta_P}{1+\theta_P}\,\pi(f^2) > 0. 
\end{equation*}
We then obtain
$$\dkol\left(\frac{S_n- \esper[S_n]}{\sqrt{\var(S_n)}}\,,\,\gauss\right) \leq 77\,\left(\frac{\|g\|_\infty\,\sqrt{M}}{\sqrt{\pi(g^2)}}\right)^{\!3}\,\left(\frac{1+\theta_P}{1-\theta_P}\right)^{\!\frac{7}{2}}\,\frac{1}{\sqrt{n}}$$
for $n$ large enough.\qed
\end{proof}

\begin{remark}
  In this remark, we discuss mod-Gaussian convergence for linear statistics of Markov chains.
  We use the above notation and assume that $\varSigma^2(f)>0$.
  Consider the third cumulant of $S_n$. One can easily prove that
  \[\rho= \frac{1}{n} \lim_{n \to \infty} \kappa_3(S_n) = \sum_{j,k \in \Z} \kappa(f(X_0,X_j,X_k)),\]
  where $(X_t)_{t \in \Z}$ is a bi-infinite stationary Markov chain with transition matrix $P$.
  (The sum on the right-hand side is finite, as consequence of Proposition~\ref{prop:UWDG_Markov}).
  Let us call $\rho$ this limit. We then consider the rescaled random variables
  $$X_n=\left(\frac{S_n- \esper[S_n]}{(\Var(S_n))^{1/3}}\right)_{n\in \N}.$$ 
As for the magnetization in the Ising model,
from \cite[Section 5]{FMN16} (with $\alpha_n=\Var(S_n)$ and $\beta_n=1$),
we know that $X_n$ converges in the mod-Gaussian sense
with parameters $t_n=(\Var(S_n))^{1/3}$ and limit $\psi(z)=\exp(\frac{L z^3}{6})$, with $L=\tfrac{\rho}{\varSigma^3(f)}$.
Again, this mod-Gaussian convergence takes place on the whole complex plane
and implies a normality zone for $S_n/\sqrt{\var(S_n)}$ of size $o(n^{1/6})$, see Proposition 4.4.1 in \emph{loc. cit.};
we also have moderate deviation estimates at the edge of this normality zone, see Proposition 5.2.1. This mod-Gaussian convergence could also have been proved by using an argument of the perturbation theory of linear operators, for which we refer to \cite{Kato80}. Indeed, the Laplace transform of $X_n$ writes explicitly as
$$\esper[\E^{zX_n}] = \pi\,(P_{z,f})^n\,1,$$
where $1$ is the column vector with all its entries equal to $1$, $\pi$ is the stationary measure of the process, and $P_{z,f}$ is the infinitesimal modification of the transition matrix defined by
$$P_{z,f}(i,j) = P(i,j)\,\E^{\frac{z (f(j)-\pi(f))}{(\var (S_n))^{1/3}}}.$$
For $z$ in a zone of control of size $O(n^{1/3})$, one can give a series expansion of the eigenvalues and eigenvectors of $P_{z,f}$, which allows one to recover the mod-Gaussian convergence. The theory of cumulants and weighted dependency graphs allows one to bypass these difficult analytic arguments.
\end{remark}

\bibliographystyle{alpha}
\bibliography{speed.bib}

\begin{thebibliography}{CDMN15}

\bibitem[AHLV15]{ArizmendiHasebeLehnerVargas2014}
O.~Arizmendi, T.~Hasebe, F.~Lehner, and C.~Vargas.
\newblock {Relations between cumulants in noncommutative probability}.
\newblock {\em Adv. Math.}, 282:56 -- 92, 2015.

\bibitem[AS08]{AS08}
N.~Alon and J.~Spencer.
\newblock {\em The probabilistic method}.
\newblock Wiley-Interscience Series in Discrete Mathematics and Optimization.
  Wiley, 2008.
\newblock 3rd edn.

\bibitem[Ber41]{Berry1941}
A.~C. Berry.
\newblock The accuracy of the {G}aussian approximation to the sum of
  independent variates.
\newblock {\em Trans. Amer. Math. Soc.}, 49(1):122--136, 1941.

\bibitem[BHNY08]{BHNY08}
P.~Bourgade, C.~Hughes, A.~Nikeghbali, and M.~Yor.
\newblock The characteristic polynomial of a random unitary matrix: a
  probabilistic approach.
\newblock {\em Duke Math. J.}, 145:45--69, 2008.

\bibitem[Bil95]{Bil95}
P.~Billingsley.
\newblock {\em Probability and Measure}.
\newblock Wiley Series in Probability and Mathematical Statistics. Wiley, 3rd
  edition, 1995.

\bibitem[BKN09]{BKN09}
A.~D. Barbour, E.~Kowalski, and A.~Nikeghbali.
\newblock Mod-discrete expansions.
\newblock {\em Probab. Th. Rel. Fields}, 158(3):859--893, 2009.

\bibitem[BKR89]{BKR89}
A.~D. Barbour, M.~Karon{\'s}ki, and A.~Rucin{\'s}ki.
\newblock A central limit theorem for decomposable random variables with
  applications to random graphs.
\newblock {\em J. Combin. Theory Ser. B}, 47(2):125--145, 1989.

\bibitem[Bol80]{Bol80}
E.~Bolthausen.
\newblock The {B}erry--{E}sseen theorem for functionals of discrete {M}arkov
  chains.
\newblock {\em Z. Wahr. verw. Geb.}, 54:59--73, 1980.

\bibitem[B{\'o}n10]{Bon10}
M.~B{\'o}na.
\newblock {On three different notions of monotone subsequences}.
\newblock In {\em {Permutation Patterns}}, volume 376 of {\em {London Math.
  Soc. Lecture Note Series}}, pages 89--113. Cambridge University Press, 2010.

\bibitem[BR89]{BR89}
P.~Baldi and Y.~Rinott.
\newblock On normal approximations of distributions in terms of dependency
  graphs.
\newblock {\em Ann. Probab.}, 17(4):1646--1650, 1989.

\bibitem[Bul96]{Bul96}
A.~V. Bulinskii.
\newblock Rate of convergence in the central limit theorem for fields of
  associated random variables.
\newblock {\em Th. Prob. Appl.}, 40(1):136--144, 1996.

\bibitem[CDMN15]{CDMN15}
R.~Chhaibi, F.~Delbaen, P.-L. M{\'e}liot, and A.~Nikeghbali.
\newblock Mod-$\phi$ convergence: {A}pproximation of discrete measures and
  harmonic analysis on the torus.
\newblock \texttt{arXiv:1511.03922}, 2015.

\bibitem[Cog72]{Cog72}
R.~Cogburn.
\newblock The central limit theorem for {M}arkov processes.
\newblock In {\em Proceedings of the Sixth Annual Berkeley Symposium on
  Mathematical Statistics and Probability}, volume~2, pages 485--512.
  University of California Press, 1972.

\bibitem[DIS73]{Duneau1}
M.~Duneau, D.~Iagolnitzer, and B.~Souillard.
\newblock Decrease properties of truncated correlation functions and
  analyticity properties for classical lattices and continuous systems.
\newblock {\em Comm. Math. Phys.}, 31(3):191--208, 1973.

\bibitem[DIS74]{Duneau2}
M.~Duneau, D.~Iagolnitzer, and B.~Souillard.
\newblock Strong cluster properties for classical systems with finite range
  interaction.
\newblock {\em Comm. Math. Phys.}, 35:307--320, 1974.

\bibitem[DKN15]{DKN11}
F.~Delbaen, E.~Kowalski, and A.~Nikeghbali.
\newblock Mod-$\phi$ convergence.
\newblock {\em Intern. Math. Res. Not.}, 2015(11):3445--3485, 2015.

\bibitem[DS91]{DS91}
P.~Diaconis and D.~Stroock.
\newblock Geometric bounds for eigenvalues of {M}arkov chains.
\newblock {\em Ann. Appl. Probab.}, 1(1):36--61, 91.

\bibitem[Ell85]{Ellis}
R.~Ellis.
\newblock {\em Entropy, Large Deviations, and Statistical Mechanics}.
\newblock Springer, 1985.

\bibitem[Fel71]{Fel71}
W.~Feller.
\newblock {\em An Introduction to Probability Theory and Its Applications,
  Volume II}.
\newblock Wiley Series in Probability and Mathematical Statistics. Wiley, 2nd
  edition, 1971.

\bibitem[F{\'e}r16]{WDG}
V.~F{\'e}ray.
\newblock Weighted dependency graphs.
\newblock arXiv preprint 1605.03836, 2016.

\bibitem[Fil91]{Fil91}
J.~A. Fill.
\newblock Eigenvalue bounds on convergence to stationarity for nonreversible
  {M}arkov chains, with an application to the exclusion process.
\newblock {\em Ann. Appl. Probab.}, 1(1):62--87, 1991.

\bibitem[FMN16]{FMN16}
V.~F{\'e}ray, P.-L. M{\'e}liot, and A.~Nikeghbali.
\newblock {\em Mod-$\phi$ convergence: {N}ormality zones and precise
  deviations}.
\newblock SpringerBriefs in Probability and Mathematical Statistics. Springer,
  2016.

\bibitem[FV16]{Velenik}
S.~Friedli and Y.~Velenik.
\newblock {\em Statistical Mechanics of Lattice Systems: a Concrete
  Mathematical Introduction}.
\newblock Cambridge University press, 2016.
\newblock In press.

\bibitem[Gri92]{Gri92}
G.~Grimmett.
\newblock Weak convergence using higher-order cumulants.
\newblock {\em J. Theoretical Probability}, 5(4):767--773, 1992.

\bibitem[GW16]{GW16}
L.~Goldstein and N.~Wiroonsri.
\newblock {Stein's method for positively associated random variables with
  applications to Ising, percolation and voter models}.
\newblock preprint arXiv:1603.05322, to appear in {\em Ann. Inst. Henri
  Poincar\'e}, 2016.

\bibitem[H{\"a}g05]{Hagg05}
O.~H{\"a}ggstr{\"o}m.
\newblock On the central limit theorem for geometrically ergodic {M}arkov
  chains.
\newblock {\em Probab. Th. Rel. Fields}, 132:74--82, 2005.

\bibitem[Jan88]{Jan88}
S.~Janson.
\newblock Normal convergence by higher semi-invariants with applications to
  sums of dependent random variables and random graphs.
\newblock {\em Ann. Probab.}, 16:305--312, 1988.

\bibitem[JKN11]{JKN11}
J.~Jacod, E.~Kowalski, and A.~Nikeghbali.
\newblock Mod-{G}aussian convergence: new limit theorems in probability and
  number theory.
\newblock {\em Forum Mathematicum}, 23(4):835--873, 2011.

\bibitem[Jon04]{Jon04}
G.~L. Jones.
\newblock On the {M}arkov chain central limit theorem.
\newblock {\em Probability Surveys}, 1:299--320, 2004.

\bibitem[Kat80]{Kato80}
T.~Kato.
\newblock {\em Perturbation Theory for Linear Operators}.
\newblock Classics in Mathematics. Springer-Verlag, 2nd edition, 1980.

\bibitem[KN10]{KN10}
E.~Kowalski and A.~Nikeghbali.
\newblock Mod-{P}oisson convergence in probability and number theory.
\newblock {\em Internat. Math. Res. Not.}, 18:3549--3587, 2010.

\bibitem[KN12]{KN12}
E.~Kowalski and A.~Nikeghbali.
\newblock {Mod-{G}aussian distribution and the value distribution of
  $\zeta(\frac{1}{2}+\mathrm{i}t)$ and related quantities}.
\newblock {\em J. London Math. Soc.}, 86(2):291--319, 2012.

\bibitem[KRT15]{KRT15}
K.~Krokowski, A.~Reichenbachs, and C.~Th{\"a}le.
\newblock Discrete {M}alliavin--{S}tein method: {B}erry--{E}sseen bounds for
  random graphs, point processes and percolation.
\newblock \texttt{arXiv:1503.01029v1 [math.PR]}, 2015.

\bibitem[KS76]{KS76}
J.~G. Kemeny and J.~L. Snell.
\newblock {\em Finite Markov Chains}.
\newblock Undergraduate Texts in Mathematics. Springer-Verlag, 1976.

\bibitem[KS00]{KS2000}
J.~P. Keating and N.~C. Snaith.
\newblock {Random matrix theory and $\zeta(\frac{1}{2}+\mathrm{i}t)$}.
\newblock {\em Comm. Math. Phys.}, 214(1):57--89, 2000.

\bibitem[KS10]{KS10}
V.~Y. Korolev and I.~G. Shevtsova.
\newblock On the upper bound for the absolute constant in the
  {B}erry–{E}sseen inequality.
\newblock {\em Theory of Probability and its Applications}, 54(4):638--658,
  2010.

\bibitem[KV86]{KV86}
C.~Kipnis and S.~R.~S. Varadhan.
\newblock Central limit theorem for additive functionals of reversible {M}arkov
  processes and applications to simple exclusions.
\newblock {\em Communications in Mathematical Physics}, 104:1--19, 1986.

\bibitem[Lan93]{Lang93}
S.~Lang.
\newblock {\em Real and Functional Analysis}, volume 142 of {\em Graduate Texts
  in Mathematics}.
\newblock Springer-Verlag, 3rd edition, 1993.

\bibitem[Lez96]{Lez96}
P.~Lezaud.
\newblock {\em Chernoff-type bound for finite {M}arkov chains}.
\newblock PhD thesis, Universit{\'e} Paul Sabatier, Toulouse, 1996.

\bibitem[LS59]{LS59}
V.~P. Leonov and A.~N. Shiryaev.
\newblock On a method of calculation of semi-invariants.
\newblock {\em Theory Prob. Appl.}, 4:319--329, 1959.

\bibitem[Mal95]{Mal95}
P.~Malliavin.
\newblock {\em Integration and Probability}, volume 157 of {\em Graduate Texts
  in Mathematics}.
\newblock Springer-Verlag, 1995.

\bibitem[Man96]{Mann96}
B.~Mann.
\newblock {\em Berry--{E}sseen central limit theorems for {M}arkov chains}.
\newblock PhD thesis, Harvard University, 1996.

\bibitem[MM91]{MM91}
V.~A. Malyshev and R.~A. Minlos.
\newblock {\em Gibbs random fields}.
\newblock Springer, 1991.

\bibitem[MN15]{MN15}
P.-L. M\'eliot and A.~Nikeghbali.
\newblock Mod-{G}aussian convergence and its applications for models of
  statistical mechanics.
\newblock In {\em In Memoriam Marc Yor - Séminaire de Probabilités XLVII},
  volume 2137 of {\em Lecture Notes in Mathematics}, pages 369--425.
  Springer-Verlag, 2015.

\bibitem[New80]{New80}
C.~M. Newman.
\newblock Normal fluctuations and the {FKG} inequalities.
\newblock {\em Comm. Math. Phys.}, 74(2):119--128, 1980.

\bibitem[PY05]{PY05}
M.~Penrose and J.~Yukich.
\newblock Normal approximation in geometric probability.
\newblock {\em Stein's Method and Applications, Lecture Note Series, Institute
  for Mathematical Sciences, National University of Singapore}, 5:37--58, 2005.

\bibitem[Rin94]{Rin94}
Y.~Rinott.
\newblock On normal approximation rates for certain sums of dependent random
  variables.
\newblock {\em J. Comp. Appl. Math.}, 55(2):135 -- 143, 1994.

\bibitem[Rud91]{Rud91}
W.~Rudin.
\newblock {\em Functional Analysis}.
\newblock McGraw-Hill, 2nd edition, 1991.

\bibitem[Sat99]{Sato99}
K.-I. Sato.
\newblock {\em Lévy processes and infinitely divisible distributions},
  volume~68 of {\em Cambridge Studies in Advanced Mathematics}.
\newblock Cambridge University Press, 1999.

\bibitem[SC97]{SC97}
L.~Saloff-Coste.
\newblock Lectures on finite {M}arkov chains.
\newblock In {\em Lectures on probability theory and statistics (Saint-Flour,
  1996)}, volume 1665 of {\em Lecture Notes in Mathematics}, pages 301--413.
  Springer-Verlag, 1997.

\bibitem[Spi58]{Spi58}
F.~Spitzer.
\newblock Some theorems concerning $2$-dimensional {B}rownian motion.
\newblock {\em Trans. Amer. Math. Soc.}, 87:187--197, 1958.

\bibitem[SS91]{LivreOrange:Cumulants}
L.~Saulis and V.~A. Statulevičius.
\newblock {\em Limit theorems for large deviations}, volume~73 of {\em
  Mathematics and its Applications (Soviet Series)}.
\newblock Kluwer Academic Publishers, 1991.

\bibitem[Sta66]{Sta66}
V.~A. Statulevi{\v{c}}ius.
\newblock On large deviations.
\newblock {\em Prob. Th. Rel. Fields}, 6(2):133--144, 1966.

\bibitem[Str11]{Stroock}
D.~W. Stroock.
\newblock {\em Probability Theory: An Analytic View}.
\newblock Cambridge University Press, 2nd edition, 2011.

\bibitem[SY83]{SY83}
K.-I. Sato and M.~Yamazato.
\newblock Stationary processes of ornstein-uhlenbeck type.
\newblock In K.~Itô and J.~V. Prokhorov, editors, {\em Probability Theory and
  Mathematical Statistics, Fourth USSR-Japan Symp.}, volume 1021 of {\em
  Lecture Notes in Mathematics}, pages 541--551. Springer-Verlag, 1983.

\bibitem[Tao12]{Tao12}
T.~Tao.
\newblock {\em Topics in Random Matrix Theory}, volume 132 of {\em Graduate
  Studies in Mathematics}.
\newblock American Mathematical Society, 2012.

\end{thebibliography}

\end{document}